\documentclass[a4paper, reqno]{article}
\usepackage{amsmath, amssymb, amsthm}
\usepackage{fullpage}
\usepackage{hyperref}
\usepackage{tikz}
\usepackage{bbold}

\newcommand{\ol}{\overline}

\newcommand{\wt}{\widetilde}
\newcommand{\coleq}{\mathrel{\mathop:}=}

\theoremstyle{definition}
\newtheorem{definition}{Definition}[section]
\newtheorem{notation}[definition]{Notation}
\newtheorem{remark}[definition]{Remark}

\theoremstyle{plain}
\newtheorem{lemma}[definition]{Lemma}
\newtheorem{proposition}[definition]{Proposition}
\newtheorem{theorem}[definition]{Theorem}
\newtheorem{corollary}[definition]{Corollary}

\newtheorem{innercustomthm}{Theorem}
\newenvironment{theorem'}[1]
{\renewcommand\theinnercustomthm{#1}\innercustomthm}
{\endinnercustomthm}

\newtheorem{innercustomcor}{Corollary}
\newenvironment{corollary'}[1]
{\renewcommand\theinnercustomcor{#1}\innercustomcor}
{\endinnercustomcor}

\begin{document}
\title{Uniform convergence to equilibrium for a family of \\ drift-diffusion models with trap-assisted recombination 
and the limiting Shockley--Read--Hall model}

\author{Klemens Fellner\thanks{Institute for Mathematics and Scientific Computing, Karl-Franzens-Universit\"at Graz, Heinrichstr. 36, 8010 Graz, Austria. Email: klemens.fellner@uni-graz.at}
\qquad
Michael Kniely\thanks{Institute for Mathematics and Scientific Computing, Karl-Franzens-Universit\"at Graz, Heinrichstr. 36, 8010 Graz, Austria. Email: michael.kniely@uni-graz.at}}

\maketitle

\begin{abstract}
We consider a family of drift-diffusion-recombination systems, 
where the recombination of electrons and holes is facilitated 
by an intermediate energy-level for electrons in so-called trapped states. 
In particular, it has been proven in \cite{GMS07} that the associated 
quasi-stationary state limit of an instantaneously fast trapped dynamics yields 
the famous Shockley--Read--Hall model for electron and hole recombination in semiconductor devices.

The main result of this paper proves exponential convergence to equilibrium uniformly in the fast reaction limit 
for the drift-diffusion-recombination systems and the limiting Shockley--Read--Hall model. 
The proof applies the so-called entropy method and the key results is to establish 
an entropy-entropy production inequality uniformly in the fast reaction limit.

Moreover, we prove existence of global solutions and show a-priori estimates, which 
are necessary to rigorously verify that solutions satisfy the entropy-entropy production law.



\end{abstract}



%
%
%
%
%
%
%
%

\section{Introduction and main results}
In this paper, we consider the following PDE-ODE drift-diffusion-recombination system: 
\begin{equation}
\label{eqsystem}
\begin{cases}
\begin{aligned}
\partial_t n &= \nabla \cdot J_n(n) + R_n(n,n_{tr}), \\
\partial_t p &= \nabla \cdot J_p(p) + R_p(p,n_{tr}), \\
\varepsilon \, \partial_t n_{tr} &= R_p(p, n_{tr}) - R_n(n, n_{tr}),
\end{aligned}
\end{cases}
\end{equation}
with 
\begin{gather*}
J_n := \nabla n + n\nabla V_n = \mu_n \nabla\left(\frac{n}{\mu_n}\right),  \quad \mu_n := e^{-V_n},  \\
J_p := \nabla p + p\nabla V_p = \mu_p \nabla\left(\frac{p}{\mu_p}\right),  \quad \mu_p := e^{-V_p}, \\
R_n := \frac{1}{\tau_n} \left( n_{tr} - \frac{n}{n_0 \mu_n} (1 - n_{tr}) \right), \quad
R_p := \frac{1}{\tau_p} \left( 1 - n_{tr} - \frac{p}{p_0 \mu_p} n_{tr} \right),
\end{gather*}
where $n_0, p_0, \tau_n, \tau_p>0$ are positive recombination parameters and 
$\varepsilon\in(0,\varepsilon_0]$ for arbitrary $\varepsilon_0>0$ is a positive relaxation parameter to be detailed in the following. 

The physical motivation for system  \eqref{eqsystem} originates from the studies of Shockley, Read and Hall \cite{SR52, H52} on the generation-recombination statistics for electron-hole pairs in semiconductors. The involved physical processes are sketched in Figure \ref{figmodel}. The starting point for our considerations is a basic model of a semiconductor consisting of two electronic energy bands: In this model, charge carriers within the semiconductor are negatively charged electrons in the conduction band and positively charged holes (these are pseudo-particles, which describe vacancies of electrons) in the valence band. The corresponding charge densities of electrons and holes are denoted by $n$ and $p$, respectively. In Figure \ref{figmodel}, the in-between trap-level is a consequence of appropriately distributed foreign atoms in the crystal lattice of the semiconductor material. In general, there might be multiple intermediate energy levels due to various crystal impurities. In the sequel, we will restrict ourselves to exactly one additional trap level. 
The intermediate energy states facilitate the excitation of electrons from the valence band into the conduction band since this transition can now take part in two steps, each requiring smaller amounts of energy. On the other hand, 
charge carriers on the trap level are not mobile and their maximal density $n_{tr}$ is limited.
\medskip

The equations for $n$ and $p$ in system \eqref{eqsystem} include the drift-diffusion terms $\nabla \cdot J_n$ and $\nabla \cdot J_p$ as well as the recombination-terms $R_n$ and $R_p$. The quantities $V_n$ and $V_p$ within the fluxes $J_n$ and $J_p$ are given external time-independent potentials, which generate an additional drift for $n$ and $p$. 
Note that more realistic drift-diffusion models would additionally consider Poisson's equation coupled to \eqref{eqsystem} in order to 
incorporate drift caused by a self-consistent electrostatic potential. However, including a self-consistent drift structure into \eqref{eqsystem}
leads to great and still partially open difficulties in the here presented entropy method and is thus left 
for future works. 

The reaction-term $R_n$ models transitions of electrons from the trap-level to the conduction band (proportional to $n_{tr}$) and vice versa (proportional to $-n(1-n_{tr})$), where the maximum capacity of the trap-level is normalised to one. Similarly, $R_p$ encodes the generation and annihilation of holes in the valence band. But one has to be aware that the rate of hole-generation is equivalent to the rate of an electron moving from the valence band to the trap-level, which is proportional to ($1-n_{tr}$). Similar, the annihilation of a hole corresponds to an electron that jumps from the trap-level to the valence band, which yields a reaction rate proportional to $-pn_{tr}$.
\medskip

The dynamical equation for $n_{tr}$ in \eqref{eqsystem} is an ODE in time and pointwise in space with a right hand side depending on $n$ and $p$ via $R_n$ and $R_p$. In the same manner as above, one can find that all gain- and loss-terms for $n_{tr}$ are taken into account correctly via $R_p - R_n$. We stress that there is no drift-diffusion-term for $n_{tr}$. This is due to the correlation between foreign atoms and the corresponding trap-levels which are locally bound near these crystal impurities. As a consequence, an electron in a trap-level cannot move through the semiconductor, hence, the name trapped state.

In the recombination process, $n_0, p_0 > 0$ represent reference levels for the charge concentrations $n$ and $p$, while $\tau_n, \tau_p > 0$ are inverse reaction parameter. Finally, $\varepsilon > 0$ models the lifetime of the trapped states, where lifetime refers to the expected time until an electron in a trapped state moves either to the valence or the conduction band. Note that the concentration $n_{tr}$ of these trapped states satisfies $n_{tr} \in [0,1]$ provided this holds true for their initial concentration (cf. Theorem \ref{theoremsolution}).
\medskip

\begin{figure}%
\label{figmodel}
\centering
\begin{tikzpicture}
\draw[->] (-1,0) -- (-1,2) node[anchor=south]{Energy};
\draw (0,0) -- (3,0) node[anchor=west]{\quad valence band};
\draw[dashed] (0,1) -- (3,1) node[anchor=west]{\quad \textit{trap-level}};
\draw (0,2) -- (3,2) node[anchor=west]{\quad conduction band};
\draw[->] (.8, .2) -- (.8, .8);
\draw[<-] (1.1, .2) -- (1.1, .8);
\draw[->] (1.9, 1.2) -- (1.9, 1.8);
\draw[<-] (2.2, 1.2) -- (2.2, 1.8);
\end{tikzpicture}
\caption{A schematic picture illustrating the allowed transitions of electrons between the various energy levels.}
\end{figure}
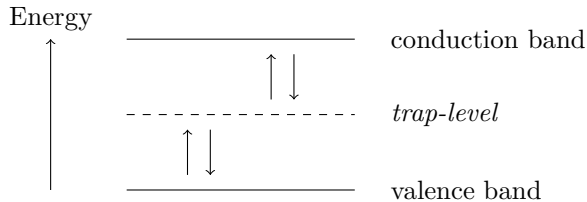

A particular situation is obtained in the (formal) limit $\varepsilon \rightarrow 0$. This quasi-stationary limit allows to derive the well known Shockley--Read--Hall-model for semiconductor recombination, where the concentration of trapped states is determined from the algebraic relation $0 = R_p(p, n_{tr}) - R_n(n, n_{tr})$, which results in
\[
n_{tr} = \frac{\tau_n + \tau_p \frac{n}{n_0 \mu_n}}{\tau_n + \tau_p + \tau_n \frac{p}{p_0 \mu_p} + \tau_p \frac{n}{n_0 \mu_n}}.
\]
Thus, the trapped state concentration $n_{tr}$ and its evolution can (formally) be eliminated from system \eqref{eqsystem}, 
while the evolution of the charge carriers $n$ and $p$ is then subject to the  
Shockley--Read--Hall recombination terms 
\[
R_n(n, n_{tr}) = R_p(p, n_{tr}) = \frac{1 - \frac{n p}{n_0 p_0 \mu_n \mu_p}}{\tau_n \big(1 + \frac{p}{p_0 \mu_p}\big) + \tau_p \big(1 + \frac{n}{n_0 \mu_n}\big)}.
\]
Note that the above quasi-stationary limit has been rigorously performed in \cite{GMS07}, even for more general models. 
See also \cite{MRS90} for semiconductor models assuming a reaction term of Shockley--Read--Hall-type.
\medskip

The main goal of this paper is to prove exponential convergence to equilibrium of system \eqref{eqsystem} 
with rates and constants, which are \emph{independent} of the relaxation time $\varepsilon$.
We will therefore always consider $\varepsilon \in (0, \varepsilon_0]$ for some arbitrary but fixed $\varepsilon_0 > 0$. Our approach also allows us to study the limiting case $\varepsilon \rightarrow 0$. 
\medskip

In the following, system \eqref{eqsystem} is considered on a bounded domain $\Omega \subset \mathbb{R}^m$ with sufficiently smooth boundary $\partial \Omega$.
In addition, we suppose that the volume of $\Omega$ is normalised, i.e. $|\Omega|=1$, which can be achieved by an appropriate scaling of the spatial variables.
We impose no-flux boundary conditions for $J_n$ and $J_p$,
\begin{equation}\label{BC}
\hat n \cdot J_n = \hat n \cdot J_p = 0 \quad \mbox{on} \ \partial \Omega,
\end{equation}
where $\hat n$ denotes the outer unit normal vector on $\partial \Omega$. 

The potentials $V_n$ and $V_p$ are assumed to satisfy
\begin{equation}
\label{eqpot}
V_n, V_p \in W^{2,\infty}(\Omega) \qquad \mbox{and} \qquad \hat n \cdot \nabla V_n, \ \hat n \cdot \nabla V_p \geq 0 \quad \mbox{on} \ \partial \Omega,
\end{equation}
where the last condition means that the potentials are confining. 
For later use, we introduce 
\[
V := \max( \| V_n \|_{L^\infty(\Omega)}, \| V_p \|_{L^\infty(\Omega)}).
\]
Finally, we assume that the initial states fulfil 
\[
(n_I, p_I, n_{tr,I}) \in L_+^\infty(\Omega)^3, \qquad \|n_{tr,I}\|_{L^\infty(\Omega)} \leq 1.
\]

As a consequence of the no-flux boundary conditions, system \eqref{eqsystem} features conservation of charge:
\[
\partial_t (n-p + \varepsilon \, n_{tr}) = \nabla \cdot (J_n - J_p)
\]
and, therefore, 
\begin{equation}
\label{eqconslaw}
\int_{\Omega} (n-p + \varepsilon \, n_{tr}) \, dx = \int_{\Omega} (n_I - p_I + \varepsilon \, n_{tr,I}) \, dx =: M,
\end{equation}
where $M \in \mathbb{R}$ is a real and possibly negative constant and $\varepsilon\in(0,\varepsilon_0]$ for arbitrary $\varepsilon_0>0$.
\medskip


The following Theorem \ref{theoremsolution} comprises the existence and regularity results which provide the framework for our subsequent considerations. In particular, we will show that there exists a global solution to \eqref{eqsystem}, and that $n_{tr}(t,x) \in [0,1]$ for all $t \in [0,\infty)$ and a.a. $x \in \Omega$.

\begin{theorem}[Time-dependent system]
\label{theoremsolution}
Let $n_0, p_0, \tau_n, \tau_p$ and $\varepsilon$ be positive constants. Assume that $V_n$ and $V_p$ satisfy \eqref{eqpot}
and that $\Omega \subset \mathbb{R}^m$ is a bounded,  sufficiently smooth domain. 

Then, for any  
non-negative initial datum $(n_I, p_I, n_{tr,I}) \in L^\infty(\Omega)^3$ satisfying $\|n_{tr,I}\|_{L^\infty(\Omega)} \leq 1$, there exists a unique non-negative global weak solution $(n, p, n_{tr})$ of system \eqref{eqsystem}, where $(n, p)$ satisfy the boundary conditions \eqref{BC} in the weak sense.

More precisely, for all $T \in (0, \infty)$ and by introducing the space 
\begin{equation}\label{W2}
W_2(0, T) := \big\{ f \in L^2((0, T), H^1(\Omega)) \, | \, \partial_t f \in L^2((0, T), H^1(\Omega)^*) \big\} \hookrightarrow C([0, T], L^2(\Omega)),
\end{equation}
where we recall the last embedding e.g. from \cite{Chi00},
we find that 
\begin{equation}
\label{eqsolutionnp}
(n, p) \in \left(C([0, T], L^2(\Omega))\cap W_2(0,T) \cap L^\infty((0, T), L^\infty(\Omega))\right)^2, 
\end{equation}
and 
\begin{equation}
\label{eqsolutionntr}
n_{tr} \in C([0, T], L^\infty(\Omega)), \quad \partial_t n_{tr} \in C([0, T], L^2(\Omega)).
\end{equation}



Moreover, there exist positive constants 
$C_n(\|n_I\|_{\infty},V_n)$, $C_p(\|p_I\|_{\infty},V_p)$ and $K_n(V_n) $, $K_p(V_p)$ independent of $\varepsilon$ such that
\begin{equation}\label{npLinfty}
\|n(t,\cdot)\|_{\infty} \le C_n + K_nt,\quad
\|p(t,\cdot)\|_{\infty} \le C_p + K_pt,\qquad
\text{for all } t\ge0.
\end{equation}

In addition, 
the concentration $n_{tr}(t, x)$ is bounded away from zero and one in the sense that for all times $\tau>0$ there exist positive constants $\eta = \eta(\varepsilon_0, \tau, \tau_n, \tau_p)$, $\theta = \theta(C_n,C_p,K_n,K_p)$ and a sufficiently small constant $\gamma(\tau,C_n,C_p,K_n,K_p)>0$   
such that 
\begin{equation}\label{ntrbounds}
n_{tr}(t,x) \in \left[\min\Bigl\{\eta t, \frac{\gamma}{1+\theta t}\Bigr\}, \, \max\Bigl\{1-\eta t,1-\frac{\gamma}{1+\theta t}\Bigr\}\right] \quad \text{for all } t\ge0 \text{ and a.a. } x\in\Omega   
\end{equation}
where $\eta \tau = \frac{\gamma}{1+\theta \tau}$ such that the linear and the inverse linear bound intersect at time $\tau$. As a consequence of \eqref{ntrbounds}, there exist positive constants $\mu$, $\Gamma>0$ (depending on $\tau$, $\eta$, $\theta$, $\gamma$, $V_n$, $V_p$) such that 
\begin{equation}\label{npbounds}
n(t,x),p(t,x) \ge \min\Bigl\{\mu\frac{ t^2}{2}, \frac{\Gamma}{1+\theta t}\Bigr\} \quad \text{for all } t\ge0 \text{ and a.a. } x\in\Omega
\end{equation}
where $\mu \frac{\tau^2}{2} = \frac{\Gamma}{1+\theta \tau}$ such that the quadratic and the inverse linear bound intersect at the same time $\tau$.
\end{theorem}

\begin{remark}
The existence theory of Theorem \ref{theoremsolution} for the coupled ODE-PDE problem \eqref{eqsystem} applies standard parabolic 
methods and pointwise ODE estimates. The proof is 
therefore postponed to the Appendix. 
It relates to previous results like \cite{GMS07} by assuming 
$L^\infty$ initial data and by proving $L^\infty$-bounds
in order to control nonlinear terms. 
We remark that the main objective of this article is the following 
quantitative study of the large-time behaviour of global solutions 
to system \eqref{eqsystem}.

\end{remark}
\medskip

The main tool in order to quantitatively study the large-time behaviour of global solutions to system \eqref{eqsystem}, 
is the entropy functional 
\begin{equation}
\label{eqentropy}
E(n,p,n_{tr}) = \int_{\Omega} \left( n \ln \frac{n}{n_0 \mu_n} - (n-n_0\mu_n) + p \ln \frac{p}{p_0 \mu_p} - (p-p_0\mu_p) + \varepsilon \int_{1/2}^{n_{tr}} \ln \left( \frac{s}{1-s} \right) ds \right) dx.
\end{equation}
For $n$ and $p$, we encounter Boltzmann-entropy contributions $a \ln a - (a - 1) \geq 0$, whereas $n_{tr}$ enters the entropy functional via a non-negative integral term. Note that the integral $\int_{1/2}^{n_{tr}} \ln \bigl( \frac{s}{1-s} \bigr) ds$ is non-negative and well-defined for all $n_{tr}(x)\in[0,1]$.

It is straight forward to calculate that the entropy functional \eqref{eqentropy} is indeed a Ljapunov functional: By introducing the entropy production functional
\begin{equation}\label{eeplaw}
D := -\frac{d}{dt} E,
\end{equation}
it holds true along solution trajectories of system \eqref{eqsystem} that 
\begin{align}
D(n,p,n_{tr}) &= - \int_{\Omega} \left( \big( \nabla \cdot J_n + R_n \big) \ln \left( \frac{n}{n_0 \mu_n} \right) + \big( \nabla \cdot J_p + R_p \big) \ln \left( \frac{p}{p_0 \mu_p} \right) + \varepsilon \ln \left( \frac{n_{tr}}{1 - n_{tr}} \right) \partial_t n_{tr} \right) dx \nonumber \\
&= \int_{\Omega} \left( J_n \cdot \frac{J_n}{n} + J_p \cdot \frac{J_p}{p} - R_n \ln \left( \frac{n}{n_0 \mu_n} \right) - R_p \ln \left( \frac{p}{p_0 \mu_p} \right) - \ln \left( \frac{n_{tr}}{1 - n_{tr}} \right) \big( R_p - R_n \big) \right) dx \nonumber \\
&= \int_{\Omega} \left(\frac{|J_n|^2}{n} + \frac{|J_p|^2}{p} - R_n \ln \left( \frac{n(1-n_{tr})}{n_0 \mu_n n_{tr}} \right) - R_p \ln \left( \frac{p n_{tr}}{p_0 \mu_p (1-n_{tr})} \right) \right) dx\ge0. \label{eqproduction}
\end{align}
The entropy production functional involves flux-terms, which are obviously non-negative, and reaction-terms of the form $(a-1) \ln a \geq 0$. Thus, the entropy $E$ and its production $D$ are non-negative functionals, which formally implies the entropy $E$ to be monotonically decreasing in time.

In order to rigorously verify (a weak version of) the entropy-entropy production law \eqref{eeplaw}, note that the last two reaction terms in \eqref{eqproduction}
are unbounded for $n_{tr}(t,x)\to0,1$ or $n(t,x),p(t,x)\to0$ and that the 
entropy production is therefore potentially unbounded 
even for smooth solutions.
However, the regularity of $n$ and $p$ of Theorem \ref{theoremsolution} as well as the lower and upper bounds \eqref{ntrbounds} for $n_{tr}$ and the lower bounds \eqref{npbounds} for $n$ and $p$ allow to prove that any solution of Theorem \ref{theoremsolution} satisfies the following weak entropy-entropy production law 
\begin{equation}\label{weakEntropy}
E(t_1) + \int_{t_0}^{t_1} D(s)\,ds = E(t_0), \qquad \text{for all } 0 < t_0 \leq t_1 < \infty.
\end{equation}
Note that \eqref{weakEntropy} implies that solutions of Theorem \ref{theoremsolution} may only feature singularities of $D$ at time zero (due to a lacking regularity of the initial data or due to initial data $n_{tr,I}(x)\in[0,1]$, $n_I(x),p_I(x)\in[0,\infty)$). 
\medskip

We will further prove that there exists a unique equilibrium $(n_\infty, p_\infty, n_{tr,\infty})$ of system \eqref{eqsystem} in a suitable (and natural) function space. This equilibrium can be seen as the unique solution of the stationary system \eqref{eqstat} or, equivalently, as the unique state for which the entropy production \eqref{eqproduction} vanishes. Note that from both viewpoints, uniqueness of the equilibrium is only satisfied once the mass constant $M$ in the conservation law \eqref{eqconslaw} is fixed. For simplicity of the presentation, we shall introduce the following notation for integrated quantities.

\begin{notation}
For any function $f$, we set
\[
\ol{f} := \int_{\Omega} f(x) \, dx 
\]
which is consistent with the usual definition of the average of $f$ since $|\Omega|=1$. Using this notation, the conservation law \eqref{eqconslaw} rewrites as
\[
\ol n - \ol p + \varepsilon \, \ol{n_{tr}} = M \in \mathbb{R}.
\]
\end{notation}

\begin{theorem}[Stationary system]
\label{theoremequilibrium}
Let $M \in \mathbb{R}$, $\varepsilon\in(0,\varepsilon_0]$ for arbitrary $\varepsilon_0>0$ and $(n_\infty, p_\infty, n_{tr,\infty}) \in X$ where $X$ is defined via 
\[
X := \{ (n, p, n_{tr}) \in H^1(\Omega)^2 \times L^\infty(\Omega) \, \big| \, \ol n - \ol p + \varepsilon \ol{n_{tr}} = M \land (\exists \, \gamma > 0) : n, p \geq \gamma \, \mbox{a.e.} \land n_{tr} \in [\gamma, 1-\gamma] \, \mbox{a.e.}\}.
\]
Then, the following statements are equivalent.
\begin{enumerate}
\item $(n_\infty, p_\infty, n_{tr,\infty})$ is a solution of the stationary system
\begin{subequations}
\label{eqstat}
\begin{alignat}{3}
\nabla \cdot J_n(n_\infty) + R_n(n_\infty,n_{tr,\infty}) &= 0, \label{eqstatn} \\
\nabla \cdot J_p(p_\infty) + R_p(p_\infty,n_{tr,\infty}) &= 0, \label{eqstatp} \\
R_p(p_\infty,n_{tr,\infty}) - R_n(n_\infty,n_{tr,\infty}) &= 0. \label{eqstatntr} 
\end{alignat}
\end{subequations}

\item $D(n_\infty, p_\infty, n_{tr,\infty}) = 0$.

\item $J_n(n_\infty) = J_p(p_\infty) = R_n(n_\infty,n_{tr,\infty}) = R_p(p_\infty,n_{tr,\infty}) = 0$ a.e. in $\Omega$.

\item The state $(n_\infty, p_\infty, n_{tr,\infty})$ satisfies
\begin{equation}
\label{eqequilibrium}
n_\infty = n_\ast e^{-V_n}, \qquad p_\infty = p_\ast e^{-V_p}, \qquad n_{tr,\infty} = \frac{n_\ast}{n_\ast + n_0} = \frac{p_0}{p_\ast + p_0}
\end{equation}
where the positive constants $n_\ast, p_\ast>0$ are uniquely determined by the condition 
\begin{equation}\label{eqequilibrium2}
n_\ast p_\ast = n_0 p_0
\end{equation}
and the conservation law 
\begin{equation}\label{eqequilibrium3}
n_\ast \ol{\mu_n} - p_\ast \ol{\mu_p} + \varepsilon \, n_{tr,\infty} = M,
\end{equation}
where the uniqueness 
follows from the strict monotonicity of 
$
f(n_\ast) \coleq n_\ast \ol{\mu_n} - \frac{n_0 p_0 \ol{\mu_p}}{n_\ast} + \varepsilon \, \frac{n_\ast}{n_\ast + n_0}
$
on $(0, \infty)$ and the asymptotics $f(n_\ast) \rightarrow -\infty$ for $n_\ast \rightarrow 0^+$ and $f(n_\ast) \rightarrow \infty$ for $n_\ast \rightarrow \infty$.
\end{enumerate}
Consequently, there exists a unique positive equilibrium $(n_\infty, p_\infty, n_{tr,\infty}) \in X$ given by the formulae in \eqref{eqequilibrium}. Furthermore, this equilibrium satisfies
\begin{equation}
\label{eqequiformulae}
n_ {tr,\infty} = \frac{n_\ast}{n_0}(1 - n_{tr,\infty}), \qquad 1 - n_{tr,\infty} = \frac{p_\ast}{p_0} n_{tr,\infty}.
\end{equation}
\end{theorem}
\begin{remark}
The characterisation of the equilibria of Theorem \ref{theoremequilibrium} can be further improved. 
The below Proposition \ref{propbounds} will  prove 
that for all $M \in \mathbb{R}$ the solutions $n_\ast, p_\ast$ of 
\eqref{eqequilibrium}--\eqref{eqequilibrium3}
are uniformly positive and bounded for all $\varepsilon\in(0,\varepsilon_0]$, i.e. that there exist constants $\gamma(\varepsilon_0, M, n_0, p_0, V)$ and $\Gamma(\varepsilon_0, M, n_0, p_0, V)$ such that
$$
0 < \gamma(\varepsilon_0, M, n_0, p_0, V) \le n_\ast, p_\ast
\le \Gamma(\varepsilon_0, M, n_0, p_0, V) < \infty
$$
for all $\varepsilon\in(0,\varepsilon_0]$ and arbitrary $\varepsilon_0>0$. 
Note that the above bounds also imply that any relevant equilibrium state $(n_\infty,p_\infty,n_{tr,\infty})$ to \eqref{eqsystem}--\eqref{eqpot}  with  
$J_n(n_\infty) = J_p(p_\infty) = R_n(n_\infty,n_{tr,\infty}) = R_p(p_\infty,n_{tr,\infty}) = 0$ a.e. in $\Omega$ lies necessarily in the function space $X$ for a suitable choice $\gamma>0$. 
\end{remark}

The main result of this article is to prove exponential convergence to the unique positive equilibrium $(n_\infty, p_\infty, n_{tr,\infty})$ for solutions of system \eqref{eqsystem}--\eqref{eqpot} 
and to obtain explicit bounds for the rates and constants of convergence. Following the idea of the so-called entropy method, we aim to derive a functional inequality of the form
\[
E(n,p,n_{tr}) - E(n_\infty, p_\infty, n_{tr,\infty}) \leq C\,D(n,p,n_{tr}),
\]
where $n$, $p$ and $n_{tr}$ are non-negative functions satisfying the conservation law \eqref{eqconslaw}, and $C>0$ is a constant which we shall estimate explicitly. This approach, which establishes an upper bound for the relative entropy in terms of the entropy production, is referred to as the entropy method. Using a Gronwall-argument, the entropy-entropy production (EEP) inequality applied to the entropy-entropy production law \eqref{weakEntropy} entails exponential decay of the relative entropy. Finally, by using a Csisz\'ar--Kullback--Pinsker-type estimate, we deduce exponential convergence in $L^1$ for solutions to system \eqref{eqsystem}.
\medskip

The derivation of an EEP-estimate is quite an involved task in our situation. The crucial part is the proof of a functional EEP-inequality, which is first shown in the special case of spatially homogeneous concentrations, which fulfil the conservation law \eqref{eqconslaw} and the natural $L^1$-bounds (cf. Proposition \ref{propeedodeconst}). This core estimate is then extended to the case of arbitrary concentrations  satisfying the same assumptions in Proposition \ref{propeedpde}.

Based on these preliminary results, Theorem \ref{theoremeed} formulates the key EEP-inequality, which is the main ingredient of the entropy method for proving exponential convergence to the equilibrium. Note that our method allows for an expression of the associated constant $C_{\mathrm{EEP}}$  in the subsequent estimate \eqref{eqeedineqabstract}, which is independent of $\varepsilon$ for all $\varepsilon \in (0, \varepsilon_0]$ and for any $\varepsilon_0 > 0$. As a consequence, also the convergence rate of the relative entropy is independent of $\varepsilon$ in this sense.

\begin{theorem}[Entropy-Entropy Production Inequality]
\label{theoremeed}
Let $\varepsilon_0$, $\tau_n$, $\tau_p$, $n_0$, $p_0$, $M_1$ and $V$ be positive constants and consider $M\in \mathbb{R}$. 

Then, for all $\varepsilon \in (0, \varepsilon_0]$ there exists an explicitly computable constant $C_{\mathrm{EEP}} > 0$ such that for all non-negative functions $(n,p,n_{tr}) \in L^1(\Omega)^3$ satisfying $\|n_{tr}\|_{L^{\infty}(\Omega)} \leq 1$, 
the conservation law 
\[
\ol n - \ol p + \varepsilon \ol{n_{tr}} = M
\]
and the $L^1$-bound 
\[
\ol n, \ol p \leq M_1,
\]
the following entropy-entropy production inequality holds true:
\begin{equation}
\label{eqeedineqabstract}
E(n,p,n_{tr}) - E(n_\infty, p_\infty, n_{tr,\infty}) \leq C_{\mathrm{EEP}} D(n,p,n_{tr}),
\end{equation}
where the equilibrium $(n_\infty, p_\infty, n_{tr,\infty}) \in X$ is given in  Theorem \ref{theoremequilibrium}.
\end{theorem}

\begin{remark}
We point out that the functions $(n,p,n_{tr})$ considered in Theorem 
\ref{theoremeed} are not necessarily solutions of \eqref{eqsystem}--\eqref{eqpot}, although we have to assume 
that the functions $(n,p,n_{tr})$ share some few natural properties like the $L^1$-bound. In particular, we emphasise that the above entropy-entropy production inequality \eqref{eqeedineqabstract} does not depend on the 
lower and upper solution bounds \eqref{npLinfty}--\eqref{npbounds}, which are only needed to prove that any solution to \eqref{eqsystem}--\eqref{eqpot} satisfies the weak entropy production law 
\eqref{weakEntropy}. 
\end{remark}

The following main result (Theorem \ref{theoremconvergence}) establishes exponential convergence to equilibrium in relative entropy and in $L^1$. We stress that the convergence rate, subsequently denoted by $K$, is uniformly positive for all $\varepsilon \in (0, \varepsilon_0]$ and arbitrary $\varepsilon_0 > 0$. Up to our knowledge, this is the first time where the entropy method has successfully been applied uniformly in a fast-reaction parameter. 

Moreover, the relative entropy and the $L^1$-distance to the equilibrium of $n$ and $p$ can be estimated from above independent of $\varepsilon$ for $\varepsilon \in (0, \varepsilon_0]$. Only $\|n_{tr} - n_{tr,\infty}\|_{L^1(\Omega)}$ is multiplied with a prefactor $\varepsilon$.

\begin{theorem}[Exponential convergence]
\label{theoremconvergence}
Let $(n,p,n_{tr})$ be a global weak solution of system \eqref{eqsystem} 
as given in Theorem \ref{theoremsolution}
corresponding to the non-negative initial data $(n_I,p_I,n_{tr,I}) \in L^\infty(\Omega)^3$ satisfying $\|n_{tr,I}\|_{L^{\infty}(\Omega)} \leq 1$. 
Then, this solution satisfies the entropy-production law 
\[
E(n,p,n_{tr})(t_1) + \int_{t_0}^{t_1} D(n, p, n_{tr})(s) \, ds = E(n,p,n_{tr})(t_0)
\]
for all $0 < t_0 \leq t_1 < \infty$. 

Moreover, the following versions of the exponential decay towards the equilibrium $(n_\infty, p_\infty, n_{tr,\infty}) \in X$ from Theorem \ref{theoremequilibrium} hold true:
\[
E(n,p,n_{tr})(t) - E_\infty \leq (E_I - E_\infty) e^{-K t},
\]
where $E_I$ and $E_\infty$ denote the initial entropy and the equilibrium entropy of the system, respectively.
Moreover,
\begin{equation}\label{expo}
\| n - n_\infty \|_{L^1(\Omega)}^2 + \| p - p_\infty \|_{L^1(\Omega)}^2 + \varepsilon \| n_{tr} - n_{tr,\infty} \|_{L^1(\Omega)}^2 \leq C (E_I - E_\infty) e^{-K t}
\end{equation}
where $C := C_{\mathrm{CKP}}^{-1}$ and $K := C_{\mathrm{EEP}}^{-1}$ are explicitly computable constants independent of $\varepsilon \in (0, \varepsilon_0]$ for arbitrary $\varepsilon_0 > 0$ (cf. Theorem \ref{theoremeed} and Proposition \ref{propentropyl1}). 
\end{theorem}

\begin{corollary}\label{coro}
The solutions  $n$ and $p$ of Theorem \ref{theoremsolution} are uniformly-in-time bounded in $L^{\infty}$, i.e. there exists a constant $K > 0$ such that 
\begin{equation}\label{npuniform}
\|n(t,\cdot)\|_{\infty} , \|p(t,\cdot)\|_{\infty}\le K \qquad \text{for all}\quad t\ge0.
\end{equation}
This global bound follows from the exponential convergence \eqref{expo} in $L^1$ to the bounded equilibrium $(n_\infty, p_\infty,$ $n_{tr,\infty})$  
and the linearly growing $L^{\infty}$-bounds \eqref{npLinfty} via an interpolation argument.

Moreover, the bounds \eqref{npuniform} allow to improve the bounds \eqref{ntrbounds}, \eqref{npbounds}
and to obtain uniform-in-time bounds in the sense that 
for all $\tau>0$, there exist sufficiently small constants $\eta, \gamma, \mu, \Gamma > 0$ such that 
\begin{equation}\label{ntruniform}
n_{tr}(t,x) \in \left[\min\bigl\{\eta t, \gamma\bigr\}, \max\bigl\{1-\eta t,1-\gamma\bigr\}\right]
\end{equation}
and
\begin{equation}\label{nploweruniform}
n(t,x),p(t,x) \ge \min\Bigl\{\mu\frac{ t^2}{2}, \Gamma\Bigr\}
\end{equation}
for all $t\ge0$ and a.a. $x\in\Omega$ where $\eta t$ and $\gamma$ as well as $\mu t^2/2$ and $\Gamma$ intersect at time $\tau > 0$. 
\end{corollary}

The final results of this paper consider  
the limit $\varepsilon \to 0$.
Up to our knowledge, Theorem \ref{theoremeed} is the 
first result of an entropy-entropy production inequality 
which holds uniformly in a fast-reaction parameter, 
i.e. uniformly for all $0<\varepsilon \le \varepsilon_0$. 
Intuitively, one thus expects the corresponding 
entropy method to extend to the limiting case $\varepsilon = 0$. The details of this singular limit are subject of the last part of this paper. In particular, one has to bypass the $\varepsilon$-dependency of the conservation law \eqref{eqconslaw}.

First, we point out that 
the limiting PDE system for $\varepsilon = 0$ is the following well known Shockley--Read--Hall drift-diffusion recombination model 
\begin{equation}
\label{eqsystem'}
\begin{cases}
\begin{aligned}
\partial_t n &= \nabla \cdot J_n(n) + R(n,p), \qquad J_n = \nabla n + n\nabla V_n, \\
\partial_t p &= \nabla \cdot J_p(p) + R(n,p), \qquad\ J_p = \nabla p + p\nabla V_p,
\end{aligned}
\end{cases}
\end{equation}
where 
\[
R(n, p) = \frac{1 - \frac{n p}{n_0 p_0 \mu_n \mu_p}}{\tau_n \big(1 + \frac{p}{p_0 \mu_p}\big) + \tau_p \big(1 + \frac{n}{n_0 \mu_n}\big)}.
\]
The existence theory of the Shockley--Read--Hall model 
follows from classical methods (see e.g. \cite{MRS90}) or 
can also be carried out similar to Theorem \ref{theoremsolution}.
Therefore, we state here the corresponding results without proof.
\begin{theorem'}{\ref*{theoremsolution}'}[Shockley--Read--Hall for $\varepsilon = 0$]
\label{theoremsolution'}
Under the assumptions of Theorem \ref{theoremsolution}, 
there exists a unique non-negative global weak solution 
$(n, p) \in \left(C([0, T], L^2(\Omega))\cap W_2(0,T) \cap L^\infty((0, T), L^\infty(\Omega))\right)^2, 
$ of system \eqref{eqsystem'}  for all $T \in (0, \infty)$ satisfying the boundary conditions \eqref{BC}.

Moreover, there exist positive constants 
$C_n(\|n_I\|_{\infty},V_n)$, $C_p(\|p_I\|_{\infty},V_p)$ and $K_n(V_n) $, $K_p(V_p)$ such that
\begin{equation}\label{npLinfty'}
\|n(t,\cdot)\|_{\infty} \le C_n + K_nt,\quad
\|p(t,\cdot)\|_{\infty} \le C_p + K_pt,\qquad
\text{for all } t\ge0.
\end{equation} 
Finally, there exist positive constants $\mu$, $\Gamma$, $\theta>0$ (depending on $\tau$, $C_n$, $C_p$, $K_n$, $K_p$, $V_n$, $V_p$) such that 
\begin{equation}\label{npbounds'}
n(t,x), \, p(t,x) \ge \min\Bigl\{\mu t, \frac{\Gamma}{1+\theta t}\Bigr\} \quad \text{for all } t\ge0 \text{ and a.a. } x\in\Omega
\end{equation}
where $\mu \tau= \frac{\Gamma}{1+\theta \tau}$ such that  the bounds $\mu t$ and $\Gamma/(1 + \theta t)$ intersect at time $\tau$.
\end{theorem'}

Secondly, the entropy functional \eqref{eqentropy} extends continuously to the limit $\varepsilon = 0$:
\[
E_0(n, p) \coleq \int_{\Omega} \left( n \ln \frac{n}{n_0 \mu_n} - (n-n_0\mu_n) + p \ln \frac{p}{p_0 \mu_p} - (p-p_0\mu_p) \right) dx,
\]
which is indeed an entropy (the free energy) functional of the Shockley--Read--Hall model with the entropy production (free energy dissipation) functional
\begin{equation}\label{eplimit}
D_0(n,p) \coleq -\frac{d}{dt} E_0(n, p) 
= \int_{\Omega} \left(\frac{|J_n|^2}{n} + \frac{|J_p|^2}{p} 
- R \ln \left(\frac{np}{n_0 \mu_n p_0\mu_p}\right)
 \right) dx
\ge 0. 
\end{equation} 
Next, we define
$n^{eq}_{tr}=n^{eq}_{tr}(n,p)$ such that $R_n(n, n_{tr}^{eq}) = R_p(p, n_{tr}^{eq})$, i.e.
\begin{equation}\label{ntreq}
n^{eq}_{tr} \coleq \frac{\tau_n + \tau_p \frac{n}{n_0 \mu_n}}{\tau_n + \tau_p + \tau_n \frac{p}{p_0 \mu_p} + \tau_p \frac{n}{n_0 \mu_n}},
\end{equation}
and $n^{eq}_{tr}(n,p)$ denotes the pointwise equilibrium value of the trapped states in \eqref{eqsystem} for fixed $n$ and $p$, which corresponds to $\varepsilon = 0$. 

Moreover, we observe that the Shockley--Read--Hall entropy production 
functional \eqref{eplimit} can be identified as 
the entropy production functional $D(n, p, n^{eq}_{tr})$ along trajectories of \eqref{eqsystem} with $\varepsilon = 0$ in the sense that $n_{tr}\equiv n^{eq}_{tr}(n,p)$:
\begin{align*}
D(n, p, n^{eq}_{tr}) 
&= \int_{\Omega} \left(\frac{|J_n|^2}{n} + \frac{|J_p|^2}{p} - R_n \ln \left( \frac{n(1-n^{eq}_{tr})}{n_0 \mu_n n^{eq}_{tr}} \right) - R_p \ln \left( \frac{p n^{eq}_{tr}}{p_0 \mu_p (1-n^{eq}_{tr})} \right) \right) dx \\
&= \int_{\Omega} \left(\frac{|J_n|^2}{n} + \frac{|J_p|^2}{p} - R \ln \left( \frac{np}{n_0 \mu_n p_0 \mu_p } \right) \right) dx=
D_0(n,p)
\end{align*}
where 
one uses $R=R_n=R_p$ at $n_{tr}=n^{eq}_{tr}$ and that the involved integrals are finite.

Analog to Theorem \ref{theoremequilibrium}, 
there exists a unique equilibrium $(n_{\infty,0}, p_{\infty,0}) \in X_0$ in the case $\varepsilon = 0$, where 
\[
X_0 := \{ (n, p) \in H^1(\Omega)^2 \, \big| \, \ol n - \ol p = M \land (\exists \, \gamma > 0) \, n, p \geq \gamma \, \mbox{a.e.} \land n_{tr}^{eq} \in [\gamma, 1-\gamma] \, \mbox{a.e.}\}.
\]
This equilibrium reads
\begin{equation}
\label{eqequilibrium0}
n_{\infty,0} = n_{\ast,0} e^{-V_n}, \qquad p_{\infty,0} = p_{\ast,0} e^{-V_p}, 
\end{equation}
where $n_{\ast,0}, p_{\ast,0}>0$ are uniquely determined by
\begin{equation*}
n_{\ast,0} p_{\ast,0} = n_0 p_0
\qquad\text{and}\qquad
n_{\ast,0} \ol{\mu_n} - p_{\ast,0} \ol{\mu_p} = M.
\end{equation*}
\medskip

We are now in the position to formulate the EEP-inequality 
\[
E_0(n, p) - E_0(n_{\infty,0}, p_{\infty,0}) \leq C_{EEP} D_0(n, p)
\]
involving the entropy $E_0$ and its production $D_0$ by applying an appropriate limiting argument to the EEP-inequality from Theorem \ref{theoremeed}.

\begin{theorem'}{\ref*{theoremeed}'}[Entropy-Entropy Production Inequality for $\varepsilon = 0$]
\label{theoremeed'}
Let $\tau_n$, $\tau_p$, $n_0$, $p_0$, $M_1$ and $V$ be positive constants and consider $M\in \mathbb{R}$.

Then, recalling the equilibrium $(n_{\infty,0}, p_{\infty,0}) \in X_0$, the following EEP-inequality holds true for all non-negative functions $(n,p) \in L^1(\Omega)^2$ satisfying 
the conservation law 
$
\ol n - \ol p = M,
$
the $L^1$-bound 
$
\ol n, \ol p < M_1
$
as well as the conditions
$E_0(n, p)<\infty$,
$D_0(n, p),D(n, p, n^{eq}_{tr}) < \infty$
for some $\varepsilon_0 > 0$:
\begin{equation}
\label{eqeedineqabstract'}
E_0(n,p) - E_0(n_{\infty,0}, p_{\infty,0}) \leq C_{\mathrm{EEP}} D_0(n,p),
\end{equation}
where $C_{\mathrm{EEP}} > 0$ is the same constant as in Theorem \ref{theoremeed}.
\end{theorem'}

\begin{theorem'}{\ref*{theoremconvergence}'}[Exponential convergence for $\varepsilon = 0$]
\label{theoremconvergence'}
Let $(n,p)$ be a global weak solution of system \eqref{eqsystem'} 
as given in Theorem \ref{theoremsolution'}
corresponding to the non-negative initial data $(n_I,p_I) \in L^\infty(\Omega)^2$. 
Then, this solution satisfies the entropy-production law 
\begin{equation}\label{wep'}
E_0(n,p)(t_1) + \int_{t_0}^{t_1} D_0(n, p)(s) \, ds = E_0(n,p)(t_0)
\end{equation}
for all $0 < t_0 \leq t_1 < \infty$. 

Moreover, the following versions of the exponential decay towards the equilibrium $(n_{\infty,0}, p_{\infty,0}) \in X_0$ hold true:
\[
E_0(n,p)(t) - E_\infty \leq (E_I - E_\infty) e^{-K t}
\]
and
\begin{equation}\label{expo'}
\| n - n_{\infty,0} \|_{L^1(\Omega)}^2 + \| p - p_{\infty,0} \|_{L^1(\Omega)}^2 \leq C (E_I - E_\infty) e^{-K t}
\end{equation}
where $C := C_{\mathrm{CKP}}^{-1}$ and $K := C_{\mathrm{EEP}}^{-1}$ are the same constants as in Theorem \ref{theoremconvergence}. Moreover, $E_I$ and $E_\infty$ denote the initial entropy of the system and the entropy in the equilibrium, respectively.
\end{theorem'}

\begin{remark}
We believe that the entropy-entropy production inequality \eqref{eqeedineqabstract'} can also be directly proven by combining estimates of Section \ref{sectionabstract} with previous works on the entropy method for 
detailed balanced reaction-diffusion models, see e.g.
\cite{DF08,DFM08,MHM15,FT17}. We emphasise, however, 
that one key novelty of Theorem \ref{theoremeed'} is to be able to 
derive an entropy-entropy production inequality via the fast-reaction parameter $\varepsilon\to 0$.
\end{remark}

In the same way as for strictly positive $\varepsilon > 0$, we can derive uniform-in-time $L^\infty$-bounds for $n$ and $p$ also in the case $\varepsilon = 0$. As before, these bounds further improve the lower bounds on $n$ and $p$.
\begin{corollary'}{\ref*{coro}'}\label{coro'}
There exists a constant $K > 0$ such that 
\begin{equation}\label{npuniform'}
\|n(t,\cdot)\|_{\infty} , \|p(t,\cdot)\|_{\infty}\le K \qquad \text{for all } t\ge0.
\end{equation}
And for all $\tau>0$ there exist sufficiently small constants $\mu, \Gamma > 0$ such that 
\begin{equation}\label{nploweruniform'}
n(t,x),p(t,x) \ge \min\left\{\mu t, \Gamma\right\}
\end{equation}
for all $t\ge0$ and a.a. $x\in\Omega$, where $\mu \tau=\Gamma$ such that the bounds $\mu t$ and $\Gamma$ intersect at time $\tau > 0$. 
\end{corollary'}

\underline{Outline:}
The remainder of the paper is organised in the following manner. Section \ref{sectionequilibrium} contains the proof of Theorem \ref{theoremequilibrium} as well as the result on the bounds of $n_\infty$, $p_\infty$ and $n_{tr,\infty}$. In Section \ref{sectionlemma}, we collect a couple of technical lemmata, and within Section \ref{sectionprop}, we state a preliminary proposition which serves as a first result towards an EEP-inequality. An abstract version of the EEP-estimate is proven in Section \ref{sectionabstract}, first for constant concentrations and based on that also for general concentrations. Section \ref{sectionconvergence} is concerned with the proofs of the EEP-inequality from Theorem \ref{theoremeed}, the announced exponential convergence from Theorem \ref{theoremconvergence} and the uniform $L^\infty$-bounds from Corollary \ref{coro}. Moreover, the proofs of Theorem \ref{theoremeed'} and Theorem \ref{theoremconvergence'} are also part of this section. Finally, the existence proofs of Theorem \ref{theoremsolution} and Theorem \ref{theoremsolution'} are contained in the Appendix.

\section{Properties of the equilibrium}
\label{sectionequilibrium}
\begin{proof}[\bf Proof of Theorem \ref{theoremequilibrium}]
We shall prove the equivalence of the statements in the Theorem by a circular reasoning. Assume that $(n_\infty, p_\infty, n_{tr,\infty}) \in X$ is a solution of the stationary system \eqref{eqstat}. In this case, \[
J_n(n_\infty), \, J_p(p_\infty), \, R_n(n_\infty,n_{tr,\infty}), \, R_p(p_\infty,n_{tr,\infty}) \in L^2(\Omega).
\]
We test equation \eqref{eqstatn} with $\ln (n_\infty/(n_0 \mu_n))$. Due to $n_\infty \in H^1(\Omega)$ and $n_\infty \geq \gamma$ a.e. in $\Omega$, the test function $\ln (n_\infty/(n_0 \mu_n))$ belongs to $H^1(\Omega)$. We find
\vspace{-1ex}
\begin{equation*}
0 
= \int_{\Omega} \left(\frac{|J_n(n_\infty)|^2}{n_\infty} - R_n(n_\infty,n_{tr,\infty}) \ln \left( \frac{n_\infty}{n_0 \mu_n} \right) \right) dx.
\end{equation*}
In the same way, we test equation \eqref{eqstatp} with $\ln (p_\infty/(n_0 \mu_p)) \in H^1(\Omega)$. This yields
\[
0 = \int_{\Omega} \left(\frac{|J_p(p_\infty)|^2}{p_\infty} - R_p(p_\infty,n_{tr,\infty}) \ln \left( \frac{p_\infty}{p_0 \mu_p} \right) \right) dx.
\]
Moreover, we multiply \eqref{eqstatntr} with $\ln (n_{tr,\infty}/(1-n_{tr,\infty})) \in L^2(\Omega)$, integrate over $\Omega$ and obtain
\[
0 = \int_{\Omega} \left( \big( R_n(n_\infty,n_{tr,\infty}) - R_p(p_\infty,n_{tr,\infty}) \big) \ln \left( \frac{n_{tr,\infty}}{1-n_{tr,\infty}} \right) \right) dx.
\]
Taking the sum of the three expressions above, we arrive at
\vspace{-1ex}
\begin{multline*}
D(n_\infty, p_\infty, n_{tr,\infty}) = \int_{\Omega} \Bigg(\frac{|J_n(n_\infty)|^2}{n_\infty} + \frac{|J_p(p_\infty)|^2}{p_\infty} \\
- R_n(n_\infty,n_{tr,\infty}) \ln \left( \frac{n_\infty(1-n_{tr,\infty})}{n_0 \mu_n n_{tr,\infty}} \right) - R_p(p_\infty,n_{tr,\infty}) \ln \left( \frac{p_\infty n_{tr,\infty}}{p_0 \mu_p (1-n_{tr,\infty})} \right) \Bigg) dx = 0.
\end{multline*}

A closer look at the formula above shows that 
\[
- R_n(n_\infty,n_{tr,\infty}) \ln \left( \frac{n_\infty(1-n_{tr,\infty})}{n_0 \mu_n n_{tr,\infty}} \right) \geq 0
\]
where equality holds if and only if $R_n(n_\infty,n_{tr,\infty}) = 0$. The same argument also applies to the other reaction term. Hence, the relation $D(n_\infty, p_\infty, n_{tr,\infty}) = 0$ immediately implies $J_n(n_\infty) = J_p(p_\infty) = R_n(n_\infty,n_{tr,\infty}) = R_p(p_\infty,n_{tr,\infty}) = 0$ a.e. in $\Omega$.

Because of $J_n(n_\infty) = J_p(p_\infty) = 0$, we know that
\[
n_\infty = n_\ast e^{-V_n}, \quad p_\infty = p_\ast e^{-V_p}
\]
with constants $n_\ast, p_\ast > 0$. Moreover, $R_n(n_\infty,n_{tr,\infty}) = R_p(p_\infty,n_{tr,\infty}) = 0$ gives rise to
\[
n_ {tr,\infty} = \frac{n_\ast}{n_0}(1 - n_{tr,\infty}), \quad 1 - n_{tr,\infty} = \frac{p_\ast}{p_0} n_{tr,\infty}.
\]
Consequently, $n_\ast p_\ast = n_0 p_0$ and 
\[
n_{tr,\infty} = \frac{n_\ast}{n_\ast + n_0} = \frac{p_0}{p_\ast + p_0} \in (0,1).
\]
The constants $n_\ast$ and $p_\ast$ are uniquely determined by the condition 
\[
n_\ast p_\ast = n_0 p_0
\]
and the conservation law 
\[
n_\ast \ol{\mu_n} - p_\ast \ol{\mu_p} + \varepsilon \, n_{tr,\infty} = M.
\]

Finally, the state
\[
n_\infty = n_\ast e^{-V_n}, \quad p_\infty = p_\ast e^{-V_p}, \quad n_{tr,\infty} = \frac{n_\ast}{n_\ast + n_0} = \frac{p_0}{p_\ast + p_0}
\]
obviously satisfies $J_n(n_\infty) = J_p(p_\infty) = R_n(n_\infty,n_{tr,\infty}) = R_p(p_\infty,n_{tr,\infty}) = 0$ a.e. in $\Omega$ which proves $(n_\infty, p_\infty, n_{tr, \infty})$ to be a solution of the stationary system.
\end{proof}

A key equilibrium property  
are the subsequent uniform bounds for $n_\ast$, $p_\ast$ and $n_{tr,\infty}$ for all $\varepsilon \in (0, \varepsilon_0]$.

\begin{proposition}[Uniform-in-$\varepsilon$ bounds on the equilibrium]
\label{propbounds}
Let $(n_\infty, p_\infty, n_{tr,\infty}) \in X$ be the unique positive equilibrium as characterised in Theorem \ref{theoremequilibrium}. 
Then, for all $M\in \mathbb{R}$ and for all $\varepsilon\in(0,\varepsilon_0]$ and arbitrary $\varepsilon_0>0$,
there exist various constants $\gamma \in (0, 1/2)$ and $\Gamma \in (1/2, \infty)$ depending only on $\varepsilon_0$, $n_0$, $p_0$, $M$, $V$ such that
\[
n_\ast, p_\ast \in [\gamma, \Gamma], \quad n_{tr,\infty} \in \left[\gamma, 1-\gamma \right]
\qquad\text{and}\qquad
n_\infty(x), p_\infty(x) \in [\gamma, \Gamma]
\]
for a.a. $x \in \Omega$.

\begin{proof} 
We recall the equilibrium conditions \eqref{eqequilibrium}--\eqref{eqequilibrium3} from Theorem \ref{theoremequilibrium}
and observe that in the equation 
\[
n_\ast \ol{\mu_n} - \frac{n_0 p_0 \ol{\mu_p}}{n_\ast} = M - \varepsilon n_{tr,\infty} = M - \varepsilon \, \frac{n_\ast}{n_\ast + n_0},
\]
the left hand side is strictly monotone increasing 
from $-\infty$ to $+\infty$ as $n_\ast\in(0,\infty)$, while the 
right hand side is strictly monotone decreasing 
and bounded between $(M,M-\varepsilon_0)$ as $n_\ast\in(0,\infty)$. Both monotonicity and unboundedness/boundedness imply  uniform positive lower and upper bounds for $n_\ast$ as explicitly proven in the following: 
First, we derive that
\begin{equation}
\label{eqninfformula}
n_\ast = \frac{M - \varepsilon n_{tr,\infty}}{2 \ol{\mu_n}} + \sqrt{ \frac{(M - \varepsilon n_{tr,\infty})^2}{4 \ol{\mu_n}^2} + \frac{n_0 p_0 \ol{\mu_p}}{\ol{\mu_n}}} > 0
\end{equation}
for all $\varepsilon \in (0,\varepsilon_0]$. Note that \eqref{eqninfformula} is not an explicit representation of $n_\ast$ since $n_{tr,\infty}$ depends itself on $n_\ast$. Because of $n_{tr,\infty} \in (0, 1)$, we further observe that
\[
n_\ast \leq \frac{|M - \varepsilon n_{tr,\infty}|}{2\ol{\mu_n}} + \sqrt{ \frac{(M - \varepsilon n_{tr,\infty})^2}{4\ol{\mu_n}^2}} + \sqrt{\frac{n_0 p_0 \ol{\mu_p}}{\ol{\mu_n}}} \leq \frac{|M| + \varepsilon_0}{\ol{\mu_n}} + \sqrt{\frac{n_0 p_0 \ol{\mu_p}}{\ol{\mu_n}}} \leq \beta < \infty,
\]
where $\beta=\beta(\varepsilon_0,n_0,p_0,M,V)$.
And as a result of the elementary inequality
$
\sqrt{a + b} \geq \sqrt{a} + \frac{b}{2 \sqrt{a} + \sqrt{b}}
$
for $a \geq 0$ and $b > 0$, we also conclude that
\[
n_\ast \geq \frac{M - \varepsilon n_{tr,\infty}}{2 \ol{\mu_n}} + \frac{|M - \varepsilon n_{tr,\infty}|}{2 \ol{\mu_n}} + \frac{\frac{n_0 p_0 \ol{\mu_p}}{\ol{\mu_n}}}{\frac{|M - \varepsilon n_{tr,\infty}|}{\ol{\mu_n}} + \sqrt{\frac{n_0 p_0 \ol{\mu_p}}{\ol{\mu_n}}}} \geq \frac{\frac{n_0 p_0 \ol{\mu_p}}{\ol{\mu_n}}}{\frac{|M| + \varepsilon_0}{\ol{\mu_n}} + \sqrt{\frac{n_0 p_0 \ol{\mu_p}}{\ol{\mu_n}}}} \geq \alpha > 0
\]
where $\alpha=\alpha(\varepsilon_0,n_0,p_0,M,V)$. Similar arguments show that corresponding bounds $\alpha$ and $\beta$ are also available for $p_\ast$. Hence,
\[
n_{tr,\infty} \in \left[ \frac{\alpha}{\alpha + n_0}, \frac{\beta}{\beta + n_0} \right].
\]
Due to $n_\infty = n_\ast e^{-V_n}$, $p_\infty = p_\ast e^{-V_p}$ and the $L^\infty$-bounds on $V_n$ and $V_p$, the claim of the Proposition follows.
\end{proof}
\end{proposition}

\section{Some technical lemmata}
\label{sectionlemma}
A particularly useful relation between the concentrations $n$, $p$ and $n_{tr}$ is the following Lemma.

\begin{lemma}
\label{lemmanptontr}
The conservation law $\ol n - \ol p + \varepsilon \, \ol{n_{tr}} = M$ and the equilibrium condition \eqref{eqequiformulae} imply
\begin{equation}
\label{eqnptontr}
\forall \, t \geq 0: \quad (\ol n - \ol{n_\infty}) \ln \left( \frac{n_\ast}{n_0} \right) + (\ol p - \ol{p_\infty}) \ln \left( \frac{p_\ast}{p_0} \right) = \varepsilon (\ol{n_{tr}} - n_{tr,\infty}) \ln \left( \frac{1-n_{tr,\infty}}{n_{tr,\infty}} \right).
\end{equation}

\begin{proof}
With $\ol{n_\infty} - \ol{p_\infty} + \varepsilon \, n_{tr,\infty} = M$ (note that $n_{tr,\infty} = \ol{n_{tr,\infty}}$ is constant), we have $\ol p - \ol{p_\infty} = \ol n - \ol{n_\infty} + \varepsilon (\ol{n_{tr}} - n_{tr,\infty})$. We employ this relation to replace $\ol p - \ol{p_\infty}$ on the left hand side of \eqref{eqnptontr} and calculate 
\[
(\ol n - \ol{n_\infty}) \ln \left( \frac{n_\ast}{n_0} \right) + (\ol p - \ol{p_\infty}) \ln \left( \frac{p_\ast}{p_0} \right) = (\ol n - \ol{n_\infty}) \ln \left( \frac{n_\ast p_\ast}{n_0 p_0} \right) + \varepsilon (\ol{n_{tr}} - n_{tr,\infty}) \ln \left( \frac{p_\ast}{p_0} \right).
\]
Now, the first term on the right hand side vanishes due to $n_\ast p_\ast = n_0 p_0$
while we use $p_\ast/p_0 = (1-n_{tr,\infty}) / n_{tr,\infty}$ for the second term and obtain
\[
(\ol n - \ol{n_\infty}) \ln \left( \frac{n_\ast}{n_0} \right) + (\ol p - \ol{p_\infty}) \ln \left( \frac{p_\ast}{p_0} \right) = 
\varepsilon (\ol{n_{tr}} - n_{tr,\infty}) \ln \left( \frac{1-n_{tr,\infty}}{n_{tr,\infty}} \right)
\]
as claimed above. \qedhere
\end{proof}
\end{lemma}

\begin{lemma}[Relative Entropy]
\label{lemmarelativeentropy}
The entropy relative to the equilibrium reads
\vspace{-1ex}
\begin{multline*}
E(n,p,n_{tr}) - E(n_\infty, p_\infty, n_{tr,\infty}) = \\
\int_{\Omega} \left( n \ln \frac{n}{n_\infty} - (n-n_\infty) + p \ln \frac{p}{p_\infty} - (p-p_\infty) + \varepsilon \int_{n_{tr,\infty}}^{n_{tr}(x)} \left( \ln \left( \frac{s}{1-s} \right) - \ln \left( \frac{n_{tr,\infty}}{1-n_{tr,\infty}} \right) \right) ds \right) dx.
\end{multline*}
\end{lemma}

\begin{proof}
By the definition of $E(n,p,n_{tr})$ in \eqref{eqentropy}, we note that 
\vspace{-1ex}
\begin{multline*}
E(n,p,n_{tr}) - E(n_\infty, p_\infty, n_{tr,\infty}) = \int_\Omega \bigg( n \ln \left( \frac{n}{n_0 \mu_n} \right) - n_\infty \ln \left( \frac{n_\infty}{n_0 \mu_n} \right) - (n - n_\infty) \\ + p \ln \left( \frac{p}{p_0 \mu_p} \right) - p_\infty \ln \left( \frac{p_\infty}{p_0 \mu_p} \right) - (p - p_\infty) + \varepsilon \int_{n_{tr,\infty}}^{n_{tr}(x)} \ln \left( \frac{s}{1-s} \right) ds \bigg) dx.
\end{multline*}
We expand the first integrand as
$
n \ln \bigl( \frac{n}{n_0 \mu_n} \bigr) = n \ln \bigl( \frac{n}{n_\infty} \bigr) + n \ln \bigl( \frac{n_\infty}{n_0 \mu_n} \bigr).
$
Thus, with $n_\infty / \mu_n = n_\ast$, we get
\vspace{-1ex}
\begin{multline*}
\int_\Omega \bigg( n \ln \left( \frac{n}{n_0 \mu_n} \right) - n_\infty \ln \left( \frac{n_\infty}{n_0 \mu_n} \right) - (n - n_\infty) \bigg) dx \\= \int_\Omega \bigg( n \ln \left( \frac{n}{n_\infty} \right) - (n - n_\infty) \bigg) dx + (\ol n - \ol{n_\infty}) \ln \left( \frac{n_\ast}{n_0} \right).
\end{multline*}
Together with an analogous calculation of the $p$-terms, we obtain
\vspace{-1ex}
\begin{multline*}
E(n,p,n_{tr}) - E(n_\infty, p_\infty, n_{tr,\infty}) = \int_\Omega \bigg( n \ln \left( \frac{n}{n_\infty} \right) - (n - n_\infty) + p \ln \left( \frac{p}{p_\infty} \right) - (p - p_\infty) \bigg) dx \\
+ (\ol n - \ol{n_\infty}) \ln \left( \frac{n_\ast}{n_0} \right) + (\ol p - \ol{p_\infty}) \ln \left( \frac{p_\ast}{p_0} \right) + \varepsilon \int_\Omega \int_{n_{tr,\infty}}^{n_{tr}(x)} \ln \left( \frac{s}{1-s} \right) ds \, dx.
\end{multline*}
Lemma \ref{lemmanptontr} allows us to reformulate the second line as
\vspace{-1ex}
\begin{multline*}
(\ol n - \ol{n_\infty}) \ln \left( \frac{n_\ast}{n_0} \right) + (\ol p - \ol{p_\infty}) \ln \left( \frac{p_\ast}{p_0} \right) + \varepsilon \int_\Omega \int_{n_{tr,\infty}}^{n_{tr}(x)} \ln \left( \frac{s}{1-s} \right) ds \, dx \\
= \varepsilon (\ol{n_{tr}} - n_{tr,\infty}) \ln \left( \frac{1-n_{tr,\infty}}{n_{tr,\infty}} \right) + \varepsilon \int_\Omega \int_{n_{tr,\infty}}^{n_{tr}(x)} \ln \left( \frac{s}{1-s} \right) ds \, dx \\
= \varepsilon \int_\Omega \int_{n_{tr,\infty}}^{n_{tr}(x)} \left( \ln \left( \frac{s}{1-s} \right) - \ln \left( \frac{n_{tr,\infty}}{1 - n_{tr,\infty}} \right) \right) ds \, dx,
\end{multline*}
which proves the claim.
\end{proof}

\begin{lemma}[Csisz\'ar--Kullback--Pinsker inequality]
\label{lemmackp}
Let $f, g: \Omega \rightarrow \mathbb{R}$ be non-negative measureable functions. Then, 
\[
\int_\Omega \left( f \ln \Big( \frac{f}{g} \Big) - (f - g) \right) dx \geq \frac{3}{2 \ol f + 4 \ol g} \| f - g \|_{L^1(\Omega)}^2.
\]
\end{lemma}
\begin{proof}
Following a proof by Pinsker, we start with the elementary inequality $3 (x-1)^2 \leq (2x + 4) (x \ln x - (x-1))$. This allows us to derive the following Csisz\'ar--Kullback--Pinsker-type inequality:
\begin{align*}
\| f - g \|_{L^1(\Omega)} &= \int_\Omega g \left|\frac{f}{g} - 1\right| \, dx \leq \int_\Omega g \sqrt{\frac{2}{3} \frac{f}{g} + \frac{4}{3}}  \, \sqrt{\frac{f}{g} \ln \Big( \frac{f}{g} \Big) - \Big(\frac{f}{g} - 1 \Big)} \, dx \\
&= \int_\Omega \sqrt{\frac{2}{3} f + \frac{4}{3} g}  \, \sqrt{f \ln \Big( \frac{f}{g} \Big) - (f - g)} \, dx \leq \sqrt{\frac{2}{3} \ol f + \frac{4}{3} \ol g} \, \sqrt{\int_\Omega \left( f \ln \Big( \frac{f}{g} \Big) - (f - g) \right) dx}
\end{align*}
where we applied H\"older's inequality in the last step.
\end{proof}

The subsequent Lemma provides $L^1$-bounds for $n$ and $p$ in terms of the initial entropy of the system and some further constants.

\begin{lemma}[$L^1$-bounds]
\label{lemmal1bounds}
Due to the monotonicity of the entropy functional, any entropy producing solution of \eqref{eqsystem} satisfies
\[
\forall \, t \geq 0: \quad \ol n, \, \ol p \leq \frac{5}{2} \max \{ n_0 \ol{\mu_n}, p_0 \ol{\mu_p} \} + \frac{3}{4} E(n(0), p(0), n_{tr}(0)) =: M_1. 
\]

\begin{proof}
Employing Lemma \ref{lemmackp} and Young's inequality, we find
\begin{align*}
\ol n &\leq n_0 \ol{\mu_n} + \| n - n_0 \mu_n \|_{L^1(\Omega)} \leq n_0 \ol{\mu_n} + \sqrt{\frac{2}{3} \ol n + \frac{4}{3} n_0 \ol{\mu_n}} \, \sqrt{\int_\Omega \left( n \ln \Big( \frac{n}{n_0 \mu_n} \Big) - (n - n_0 \mu_n) \right) dx} \\
&\leq n_0 \ol{\mu_n} + \frac{1}{3} \ol n + \frac{2}{3} n_0 \ol{\mu_n} + \frac{1}{2} \int_\Omega \left( n \ln \Big( \frac{n}{n_0 \mu_n} \Big) - (n - n_0 \mu_n) \right) dx.
\end{align*}
Solving this inequality for $\ol n$ yields 
\[
\ol n \leq \frac{5}{2} n_0 \ol{\mu_n} + \frac{3}{4} \int_\Omega \left( n \ln \left( \frac{n}{n_0 \mu_n} \right) - (n - n_0 \mu_n) \right) dx.
\]
Therefore, we arrive at
\[
\ol n \leq \frac{5}{2} n_0 \ol{\mu_n} + \frac{3}{4} E(n, p, n_{tr}) \leq \frac{5}{2} \max\{n_0 \ol{\mu_n}, p_0 \ol{\mu_p}\} + \frac{3}{4} E(n(0), p(0), n_{tr}(0))
\]
where we used the monotonicity of the entropy functional in the last step. In the same way, we may bound $\ol p$ from above.
\end{proof}
\end{lemma}

At certain points, we will have to estimate the difference between terms like $\ol{n/n_\infty}$ and $\ol n / \ol{n_{\infty}}$. Using  Lemma \ref{lemmaflux} below, we can bound this difference by the $J_n$-flux-term and, hence, by the entropy production.

\begin{lemma}
\label{lemmaflux}
Let $f \in L^1(\Omega)$ and $g \in L^\infty(\Omega)$ such that $f \geq 0$, $g \geq \gamma > 0$ a.e. on $\Omega$ and $f/g$ is weakly differentiable. Then, there exists an explicit constant $C(\|f\|_{L^1(\Omega)}, \|g\|_{L^\infty(\Omega)}, \gamma) > 0$ such that
\[
\bigg( \frac{\ol f}{\ol{g}} - \ol{\left( \frac{f}{g} \right)} \bigg)^2 \leq C \int_\Omega \left| \nabla \sqrt{ \frac{f}{g} } \, \right|^2 dx.
\]
\begin{proof}
We define
$\delta := \frac{f}{g} - \ol{\left( \frac{f}{g} \right)}$
and obtain
$
f = g \left( \ol{\left( \frac{f}{g} \right)} + \delta \right)
$
and
\[
\frac{\ol f}{\ol{g}} = \int_\Omega \frac{f}{\ol{g}} \, dx = \int_\Omega \frac{g}{\ol{g}} \left( \ol{\left( \frac{f}{g} \right)} + \delta \right) dx = \ol{\left( \frac{f}{g} \right)} + \int_\Omega \frac{g}{\ol{g}} \, \delta \, dx.
\]
Therefore,
\[
\left| \frac{\ol f}{\ol{g}} - \ol{\left( \frac{f}{g} \right)} \right| \leq \frac{\|g\|_{L^\infty(\Omega)}}{\ol{g}} \|\delta\|_{L^1(\Omega)} \leq C_P \frac{\|g\|_{L^\infty(\Omega)}}{\ol{g}} \left\| \nabla \left( \frac{f}{g} \right) \right\|_{L^1(\Omega)}
\]
by applying Poincar\'e's inequality to $\delta$ with $\ol \delta = 0$ and some constant $C_P(\Omega) > 0$. As $g \geq \gamma > 0$ is uniformly positive on $\Omega$ and $\ol g \geq \gamma$, we arrive at
\[
\left| \frac{\ol f}{\ol{g}} - \ol{\left( \frac{f}{g} \right)} \right| \leq C_P \frac{\|g\|_{L^\infty(\Omega)}}{\gamma^2} \left\| g \, \nabla \left( \frac{f}{g} \right) \right\|_{L^1(\Omega)}.
\]
Finally, we deduce
\[
\left( \frac{\ol f}{\ol{g}} - \ol{\left( \frac{f}{g} \right)} \right)^2 \leq \left(C_P \frac{\|g\|_{L^\infty(\Omega)}}{\gamma^2} \right)^2 \left\| \sqrt{fg} \sqrt{\frac{g}{f}} \nabla \left( \frac{f}{g} \right) \right\|_{L^1(\Omega)}^2 \leq 4 \ol{fg} \left(C_P \frac{\|g\|_{L^\infty(\Omega)}}{\gamma^2} \right)^2 \int_\Omega \left| \nabla \sqrt{ \frac{f}{g} } \, \right|^2 dx
\]
employing H\"older's inequality in the second step.
\end{proof}
\end{lemma}

\section{Two preliminary propositions}
\label{sectionprop}
\begin{notation}
For arbitrary functions $f$, we define the normalised quantity
\[
\wt f := \frac{\, f \,}{\ol f}.
\]
\end{notation}

The following Logarithmic Sobolev inequality on bounded domains was proven in \cite{DF14} by following an argument of Stroock \cite{Str93}.

\begin{lemma}[Logarithmic Sobolev inequality on bounded domains]
\label{Lemma:LogSob}
Let $\Omega$ be a bounded domain in $\mathbb{R}^{m}$ such that the Poincar\'e (-Wirtinger) and Sobolev inequalities
\begin{eqnarray}
& \|\phi - \int_{\Omega} \phi \, dx\|_{L^2(\Omega)}^2 \le P(\Omega) \, \|\nabla \phi\|_{L^2(\Omega)}^2\,, \label{Poincare}\\
&\|\phi\|_{L^q(\Omega)}^2 \le C_1(\Omega) \,
\|\nabla \phi\|_{L^2(\Omega)}^2 + C_2(\Omega)\,
 \|\phi\|_{L^2(\Omega)}^2\,, 
\qquad \frac{1}{q} = \frac{1}{2}-\frac{1}{m}\,,\label{Sobolev}
\end{eqnarray}
hold. Then, the logarithmic Sobolev inequality 
\begin{equation}\label{LogSob}
\int_{\Omega} \phi^2 \ln\left(\frac{\phi^2}{\|\phi\|_2^2}\right) dx \le L(\Omega,m)
\, \|\nabla \phi\|_{L^2(\Omega)}^2
\end{equation}
holds (for some constant $L(\Omega,m)>0$).
\end{lemma}

The Log-Sobolev inequality allows to bound an appropriate part of the entropy functional by the flux-parts of the entropy production. The normalised variables on the left hand side of the subsequent inequality naturally arise when reformulating the flux-terms on the right hand side in such a way that we can apply the Log-Sobolev inequality on $\Omega$.

\begin{proposition}
\label{proplogsob}
Recall the assumptions $\|V_n\|_{L^\infty(\Omega)}, \|V_p\|_{L^\infty(\Omega)}\le V$. Then, there exists a constant $C(V)>0$ such that
\[
\int_{\Omega} \left( n \ln \bigg( \frac{\wt n}{\wt{\mu_n}} \bigg) + p \ln \bigg( \frac{\wt p}{\wt{\mu_p}} \bigg) \right) dx \leq C \int_{\Omega} \left( \frac{|J_n|^2}{n} + \frac{|J_p|^2}{p} \right) dx.
\]
\begin{proof}
From the definition of $J_n$ one obtains 
\[
\int_{\Omega} \frac{|J_n|^2}{n} \, dx 
= \int_{\Omega} \frac{\mu_n}{n} \left|\nabla \left(\frac{n}{\mu_n}\right)\right|^2 \mu_n \, dx 
= 4 \, \ol{n} \int_{\Omega} \frac{\mu_n}{\ol n} \left|\nabla \sqrt{\frac{n}{\mu_n}}\right|^2 \, dx 
= 4\, \ol{n} \int_{\Omega} \frac{\mu_n}{\ol{\mu_n}} \left|\nabla \sqrt{\frac{\wt n}{\wt{\mu_n}}}\right|^2  \, dx. 
\]
We set 
\[
\phi(x) := \sqrt{\frac{\wt n}{\wt{\mu_n}}}, \quad \alpha \coleq \int_{\Omega} \phi(x)^2 \, dx
\]
and observe due to the mean-value theorem that $\alpha=\frac{\ol{\mu_n}}{\ol{n}} 
\int_{\Omega} \frac{n}{\mu_n}\,dx = \frac{\ol{\mu_n}}{\mu_n(\theta)}\le \ol{\mu_n} e^V$ is bounded 
independently of $n$.
Next, we
introduce the rescaled variable $y \coleq \alpha^{-\frac1m} x$ where $m$ denotes the space dimension. Note that $\| \phi \|_{L^2(dx)}$ is in general different from one, whereas $\| \phi \|_{L^2(dy)} = 1$. We now estimate with $\|V_n\|_{L^\infty(\Omega)}\le V$ and the Logarithmic Sobolev Inequality \eqref{LogSob}
\begin{multline*}
\int_{\Omega} | \nabla_{\!x} \phi |^2 \, dx = \int_{\Omega} | \alpha^{-\frac1m} \nabla_{\!y} \phi |^2 \, \alpha \, dy = \alpha^{1 - \frac{2}{m}} \int_{\Omega} | \nabla_{\!y} \phi |^2 \, dy \\
\geq \alpha^{1 - \frac{2}{m}} \frac1L \int_{\Omega} \phi^2 \ln ( \phi^2) \, dy =\alpha^{1 - \frac{2}{m}} \frac1L \int_{\Omega} \frac{\wt n}{\wt{\mu_n}} \ln \Big( \frac{\wt n}{\wt{\mu_n}} \Big) \, dy = \alpha^{-\frac{2}{m}} \frac1L \frac{\ol{\mu_n}}{\ol n} \int_{\Omega} \frac{n}{\mu_n} \ln \Big( \frac{\wt n}{\wt{\mu_n}} \Big) \, dx.
\end{multline*}
The corresponding estimate involving $J_n$ reads
\begin{align*}
\int_{\Omega} \frac{|J_n|^2}{n} \, dx 
&\ge 4\, \frac{\ol{n}}{\ol{\mu_n}} e^{-V} \int_{\Omega}  \left|\nabla_{\!x} \phi \right|^2  \, dx
\ge \frac{4}{L} \alpha^{-\frac{2}{m}} e^{-2V} \int_{\Omega} n \ln \Big( \frac{\wt n}{\wt{\mu_n}} \Big) \, dx.
\end{align*}
The same arguments apply to the terms involving $p$.
\end{proof}
\end{proposition}

The following Proposition contains the first step towards an entropy-entropy production inequality. The relative entropy can be controlled by the flux-part of the entropy production and three additional terms, which mainly consist of square-roots of averaged quantities. The proof that the entropy production also serves as an upper bound for these terms will be the subject of the next section. 

\begin{proposition}
\label{propeedineq}
There exists an explicit constant $C(\gamma, \Gamma, M_1) > 0$ 
such that for $(n_\infty, p_\infty, n_{tr,\infty}) \in X$ from Theorem \ref{theoremequilibrium} and all non-negative functions $(n,p,n_{tr}) \in L^1(\Omega)^3$ satisfying $n_{tr} \leq 1$, the conservation law 
\[
\ol n - \ol p + \varepsilon \ol{n_{tr}} = M
\]
and the $L^1$-bound 
\[
\ol n, \ol p \leq M_1,
\]
the following estimate holds true:
\vspace{-1ex}
\begin{multline}
\label{eqpropeedineq}
E(n,p,n_{tr}) - E(n_\infty, p_\infty, n_{tr,\infty}) \leq C \, \Bigg( \int_\Omega \left( \frac{|J_n|^2}{n} + \frac{|J_p|^2}{p} \right) dx \\
+ \bigg( \sqrt{\ol{\Big(\frac{n}{\mu_n}\Big)}} - \sqrt{n_\ast} \bigg)^2 + \bigg( \sqrt{\ol{\Big(\frac{p}{\mu_p}\Big)}} - \sqrt{p_\ast} \bigg)^2 + \varepsilon \int_\Omega \big( \sqrt{n_{tr}} - \sqrt{n_{tr,\infty}} \big)^2 \, dx \Bigg).
\end{multline}
(Note that the right hand side of \eqref{eqpropeedineq} vanishes at the equilibrium $(n_\infty, p_\infty, n_{tr,\infty})$.)

\begin{proof}
According to Lemma \ref{lemmarelativeentropy}, we have
\vspace{-1ex}
\begin{multline*}
E(n,p,n_{tr}) - E(n_\infty, p_\infty, n_{tr,\infty}) = \\
\int_{\Omega} \left( n \ln \frac{n}{n_\infty} - (n-n_\infty) + p \ln \frac{p}{p_\infty} - (p-p_\infty) + \varepsilon \int_{n_{tr,\infty}}^{n_{tr}} \left( \ln \Big( \frac{s}{1-s} \Big) - \ln \Big( \frac{n_{tr,\infty}}{1-n_{tr,\infty}} \Big) \right) ds \right) dx.
\end{multline*}
Recall that $n = \wt n\,\ol n$, $n_\infty = \wt{n_\infty}\,\ol{n_\infty}$ and $\wt{n_\infty} = \wt{\mu_n}$. Using these relations, we rewrite the first two integrands as
\[
n \ln \Big( \frac{n}{n_\infty} \Big) - (n - n_\infty) = n \ln \Big( \frac{\wt n}{\wt{\mu_n}} \Big) + n \ln \Big( \frac{\ol n}{\ol{n_\infty}} \Big) - (n - n_\infty)
\]
and analogously for the $p$-terms. This results in
\vspace{-1ex}
\begin{multline}
\label{eqentropyparts}
E(n,p,n_{tr}) - E(n_\infty, p_\infty, n_{tr,\infty}) = \int_\Omega \bigg( n \ln \Big( \frac{\wt n}{\wt{\mu_n}} \Big) + p \ln \Big( \frac{\wt p}{\wt{\mu_p}} \Big) \bigg) dx \\
+ \, \ol{n_\infty} \left( \frac{\ol n}{\ol{n_\infty}} \ln \Big( \frac{\ol n}{\ol{n_\infty}} \Big) - \Big( \frac{\ol n}{\ol{n_\infty}} - 1 \Big) \right) + \ol{p_\infty} \left( \frac{\ol p}{\ol{p_\infty}} \ln \Big( \frac{\ol p}{\ol{p_\infty}} \Big) - \Big( \frac{\ol p}{\ol{p_\infty}} - 1 \Big) \right) \\
+ \varepsilon \int_\Omega \int_{n_{tr,\infty}}^{n_{tr}} \left( \ln \Big( \frac{s}{1-s} \Big) - \ln \Big( \frac{n_{tr,\infty}}{1-n_{tr,\infty}} \Big) \right) ds \, dx.
\end{multline}

The terms in the first line of \eqref{eqentropyparts} can be estimated using the Log-Sobolev inequality of Proposition \ref{proplogsob}. Moreover, the elementary inequality $x \ln x - (x-1) \leq (x-1)^2$ for $x>0$ gives rise to
\[
\ol{n_\infty} \left( \frac{\ol n}{\ol{n_\infty}} \ln \Big( \frac{\ol n}{\ol{n_\infty}} \Big) - \Big( \frac{\ol n}{\ol{n_\infty}} - 1 \Big) \right) \leq \ol{n_\infty} \left( \frac{\ol n}{\ol{n_\infty}} - 1 \right)^2 \leq 2 \ol{n_\infty} \left[ \bigg( \ol{\Big(\frac{n}{n_\infty}\Big)} - 1 \bigg)^2 + \bigg( \frac{\ol n}{\ol{n_\infty}} - \ol{\Big(\frac{n}{n_\infty}\Big)} \bigg)^2 \right]
\]
and an analogous estimate for the corresponding expressions involving $p$. The second term on the right hand side of the previous line can be bounded from above by applying Lemma \ref{lemmaflux}, which guarantees a constant $C(\gamma, \Gamma, M_1) > 0$ such that
\[
\bigg( \frac{\ol n}{\ol{n_\infty}} - \ol{\Big(\frac{n}{n_\infty}\Big)} \bigg)^2 \leq C \int_\Omega \left| \nabla \sqrt{ \frac{n}{n_\infty} } \right|^2 dx \leq \frac{C}{4 \inf_\Omega n_\infty} \int_\Omega \frac{1}{n} \left| n_\infty \, \nabla \Big( \frac{n}{n_\infty} \Big) \right|^2 dx \leq c_1 \int_\Omega \frac{|J_n|^2}{n} \, dx
\]
for some constant $c_1(\gamma, \Gamma, M_1) > 0$. Besides,
\begin{multline*}
\bigg( \ol{\Big(\frac{n}{n_\infty}\Big)} - 1 \bigg)^2 = \frac{1}{n_\ast^2} \bigg( \ol{\Big(\frac{n}{\mu_n}\Big)} - n_\ast \bigg)^2 = \frac{1}{n_\ast^2} \left( \sqrt{\ol{\Big(\frac{n}{\mu_n}\Big)}} + \sqrt{n_\ast} \right)^2 \left( \sqrt{\ol{\Big(\frac{n}{\mu_n}\Big)}} - \sqrt{n_\ast} \right)^2 \\
= \frac{1}{n_\ast} \left( \sqrt{\ol{\Big(\frac{n}{n_\infty}\Big)}} + 1 \right)^2 \left( \sqrt{\ol{\Big(\frac{n}{\mu_n}\Big)}} - \sqrt{n_\ast} \right)^2 \leq C(\gamma, M_1) \left( \sqrt{\ol{\Big(\frac{n}{\mu_n}\Big)}} - \sqrt{n_\ast} \right)^2.
\end{multline*}
See Proposition \ref{propbounds} and Lemma \ref{lemmal1bounds} for the bounds on $n_\ast$, $n_\infty$ and $\ol n$. We have thus verified that
\[
\ol{n_\infty} \left( \frac{\ol n}{\ol{n_\infty}} \ln \Big( \frac{\ol n}{\ol{n_\infty}} \Big) - \Big( \frac{\ol n}{\ol{n_\infty}} - 1 \Big) \right) \leq c_2 \left( \int_\Omega \frac{|J_n|^2}{n} \, dx + \bigg( \sqrt{\ol{\Big(\frac{n}{\mu_n}\Big)}} - \sqrt{n_\ast} \bigg)^2 \right)
\]
with some $c_2(\gamma, \Gamma, M_1) > 0$. 
A similar estimate holds true for the corresponding part of \eqref{eqentropyparts} involving $p$.

Considering the last line in \eqref{eqentropyparts}, we further know that for all $x \in \Omega$ there exists some mean value 
\[
\theta(x) \in (\min\{n_{tr}(x), n_{tr,\infty}\}, \max\{n_{tr}(x), n_{tr,\infty}\})
\]
such that
\begin{equation}
\label{eqlnmeanvalue}
\int_{n_{tr,\infty}}^{n_{tr}(x)} \ln \left( \frac{s}{1-s} \right) ds = \ln \left( \frac{\theta(x)}{1-\theta(x)} \right) (n_{tr}(x) - n_{tr,\infty}).
\end{equation}
Consequently,
\[
\varepsilon \int_\Omega \int_{n_{tr,\infty}}^{n_{tr}(x)} \ln \left( \frac{s}{1-s} \right) ds \, dx = \varepsilon \int_\Omega \ln \left( \frac{\theta(x)}{1-\theta(x)} \right) (n_{tr}(x) - n_{tr,\infty}) \, dx.
\]
In fact, we will prove that there even exists some constant $\xi \in (0,1/2)$ such that 
\[
\theta(x) \in (\xi, 1-\xi)
\]
for all $x \in \Omega$. Thus, the function $\theta(x)$ is uniformly bounded away from $0$ and $1$ on $\Omega$. To see this, we first note that $n_{tr,\infty} \in [\gamma, 1-\gamma]$ using the constant $\gamma \in (0, 1/2)$ from Proposition \ref{propbounds}. In addition,
\[
\left| \int_{n_{tr,\infty}}^{n_{tr}(x)} \ln \left( \frac{s}{1-s} \right) ds \right| \leq \int_{0}^1 \left| \ln \Big( \frac{s}{1-s} \Big) \right| ds = 2\ln(2)
\]
for all $x \in \Omega$. Together with \eqref{eqlnmeanvalue}, this estimate implies
\[
\left| \ln \left( \frac{\theta(x)}{1-\theta(x)} \right) \right| \big| n_{tr}(x) - n_{tr,\infty} \big| \leq 2 \ln(2).
\]
We now choose an arbitrary $x \in \Omega$ and distinguish two cases. If $| n_{tr}(x) - n_{tr,\infty} | \geq \gamma/2$, then 
\[
\left| \ln \left( \frac{\theta(x)}{1-\theta(x)} \right) \right| \leq \frac{2\ln(2)}{| n_{tr}(x) - n_{tr,\infty} |} \leq \frac{4\ln(2)}{\gamma}.
\]
As a consequence of $\ln(s/(1-s)) \rightarrow \infty$ for $s \rightarrow 1^-$ and $\ln(s/(1-s)) \rightarrow -\infty$ for $s \rightarrow 0^+$, there exists some constant $\xi \in (0,\gamma)$ depending only on $\gamma$ such that $\theta(x) \in (\xi, 1-\xi)$. If $| n_{tr}(x) - n_{tr,\infty} | < \gamma/2$, then $n_{tr,\infty} \in [\gamma, 1 - \gamma]$ implies $n_{tr}(x) \in (\gamma/2, 1-\gamma/2)$ and, hence, $\theta(x) \in (\gamma/2, 1-\gamma/2)$. Again the constant $\xi$ depends only on $\gamma$. 

As a result of the calculations above, we may rewrite the last line in \eqref{eqentropyparts} as
\vspace{-1ex}
\begin{multline*}
\varepsilon \int_\Omega \int_{n_{tr,\infty}}^{n_{tr}(x)} \left( \ln \Big( \frac{s}{1-s} \Big) - \ln \Big( \frac{n_{tr,\infty}}{1-n_{tr,\infty}} \Big) \right) ds \, dx \\
= \varepsilon \int_\Omega \left( \ln \Big( \frac{\theta(x)}{1-\theta(x)} \Big) - \ln \Big( \frac{n_{tr,\infty}}{1-n_{tr,\infty}} \Big) \right) (n_{tr}(x) - n_{tr,\infty}) \, dx.
\end{multline*}
Applying the mean-value theorem to the expression in brackets and observing that
\[
\frac{d}{ds} \ln \left( \frac{s}{1-s} \right) = \frac{1}{s(1-s)},
\]
we find 
\vspace{-1ex}
\begin{multline*}
\varepsilon \int_\Omega \left( \ln \Big( \frac{\theta(x)}{1-\theta(x)} \Big) - \ln \Big( \frac{n_{tr,\infty}}{1-n_{tr,\infty}} \Big) \right) (n_{tr}(x) - n_{tr,\infty}) \, dx \\
= \varepsilon \int_\Omega \frac{1}{\sigma(x)(1-\sigma(x))} (\theta(x) - n_{tr,\infty}) (n_{tr}(x) - n_{tr,\infty}) \, dx
\end{multline*}
with some $\sigma(x) \in (\min\{\theta(x), n_{tr,\infty}\}, \max\{\theta(x), n_{tr,\infty}\})$. Since both $\theta(x), n_{tr,\infty} \in (\xi, 1-\xi)$ for all $x \in \Omega$, we also know that $\sigma(x) \in (\xi, 1-\xi)$ for all $x \in \Omega$. Thus, $(\sigma(x)(1-\sigma(x)))^{-1}$ is bounded uniformly in $\Omega$ in terms of $\xi = \xi(\gamma)$. 
Consequently,
\vspace{-1ex} 
\begin{multline*}
\varepsilon \int_\Omega \int_{n_{tr,\infty}}^{n_{tr}(x)} \left( \ln \Big( \frac{s}{1-s} \Big) - \ln \Big( \frac{n_{tr,\infty}}{1-n_{tr,\infty}} \Big) \right) ds \, dx \\
\leq \varepsilon c_3 \int_\Omega |\theta(x) - n_{tr,\infty}| |n_{tr}(x) - n_{tr,\infty}| \, dx \leq \varepsilon c_3 \int_\Omega (n_{tr} - n_{tr,\infty})^2 \, dx \\
= \varepsilon c_3 \int_\Omega (\sqrt{n_{tr}} + \sqrt{n_{tr,\infty}})^2 (\sqrt{n_{tr}} - \sqrt{n_{tr,\infty}})^2 \, dx \leq 4 \varepsilon c_3 \int_\Omega \big(\sqrt{n_{tr}} - \sqrt{n_{tr,\infty}} \big)^2 \, dx
\end{multline*}
with a constant $c_3(\gamma) > 0$ after applying the estimate $|\theta(x) - n_{tr,\infty}| \leq |n_{tr}(x) - n_{tr,\infty}|$ for all $x \in \Omega$.
Finally, we arrive at
\vspace{-1ex}
\begin{multline*}
E(n,p,n_{tr}) - E(n_\infty, p_\infty, n_{tr,\infty}) \leq C \, \Bigg( \int_\Omega \left( \frac{|J_n|^2}{n} + \frac{|J_p|^2}{p} \right) dx \\
+ \bigg( \sqrt{\ol{\Big(\frac{n}{\mu_n}\Big)}} - \sqrt{n_\ast} \bigg)^2 + \bigg( \sqrt{\ol{\Big(\frac{p}{\mu_p}\Big)}} - \sqrt{p_\ast} \bigg)^2 + \varepsilon \int_\Omega \big( \sqrt{n_{tr}} - \sqrt{n_{tr,\infty}} \big)^2 \, dx \Bigg)
\end{multline*}
with a constant $C(\gamma, \Gamma, M_1) > 0$. 
\end{proof}
\end{proposition}

\section{Abstract versions of the EEP-inequality}
\label{sectionabstract}
\begin{notation}
\label{notabstract}
We set 
\[
n_{tr}' := 1 - n_{tr}, \quad n_{tr, \infty}' := 1 - n_{tr, \infty}
\]
and define the positive constants
\[
\nu_\infty := \sqrt{\frac{n_\ast}{n_0}} = \sqrt{\frac{n_\infty}{n_0 \mu_n}}, \quad \pi_\infty := \sqrt{\frac{p_\ast}{p_0}} = \sqrt{\frac{p_\infty}{p_0 \mu_p}}, \quad \nu_{tr, \infty} := \sqrt{n_{tr, \infty}}, \quad \nu_{tr, \infty}' := \sqrt{n_{tr, \infty}'}.
\]
\end{notation}

The motivation for introducing the additional variable $n_{tr}'$ is the possibility to symmetrise expressions like $(n(1-n_{tr}) - n_{tr})^2 + (pn_{tr} - (1-n_{tr}))^2$ as $(nn_{tr}' - n_{tr})^2 + (pn_{tr} - n_{tr}')^2$. Similar terms will appear frequently within the subsequent calculations.

\begin{remark}
\label{rem4var}
We may consider $n_{tr}'$ as a fourth independent variable within our model. In this case, the reaction-diffusion system features the following two independent conservation laws:
\begin{equation*}
\ol n - \ol p + \varepsilon\,\ol{n_{tr}} = n_0 \, \ol{\mu_n \left(\frac{n}{n_0 \mu_n}\right)} - p_0 \, \ol{\mu_p \left(\frac{p}{p_0 \mu_p}\right)} + \varepsilon \, \ol{n_{tr}} = M \in \mathbb{R}, \qquad n_{tr}(x) + n_{tr}'(x) = 1 \ \mbox{for all} \ x \in \Omega.
\end{equation*}

The special formulation of the first conservation law will become clear when looking at the following two Propositions. There, we derive relations for general variables $a$, $b$, $c$ and $d$, which correspond to $\sqrt{n/(n_0 \mu_n)}$, $\sqrt{p/(p_0 \mu_p)}$, $\sqrt{n_{tr}}$ and $\sqrt{n_{tr}'}$, respectively.

In addition, we have the following $L^1$-bound (cf. Lemma \ref{lemmal1bounds}):
\[
\ol n, \ol p \leq M_1.
\]
\end{remark}

The following Proposition \ref{propeedodeconst} establishes an upper bound for the terms in the second line of \eqref{eqpropeedineq} in the case of \emph{constant} concentrations $a$, $b$, $c$ and $d$. This result is then generalised in Proposition \ref{propeedpde} to \emph{non-constant} states $a$, $b$, $c$, $d$.

\begin{proposition}[Homogeneous Concentrations]
\label{propeedodeconst}
Let $a, b, c, d \geq 0$ be constants such that their squares satisfy the conservation laws
\[
n_0 \ol{\mu_n} a^2 - p_0 \ol{\mu_p} b^2 + \varepsilon \, c^2 = M = n_0 \ol{\mu_n} \nu_\infty^2 - p_0 \ol{\mu_p} \pi_\infty^2 + \varepsilon \, \nu_{tr,\infty}^2, \qquad c^2 + d^2 = 1 = \nu_{tr,\infty}^2 + \nu_{tr,\infty}'^{\, 2}
\]
for any $\varepsilon\in(0,\varepsilon_0]$ and arbitrary $\varepsilon_0>0$. Moreover, assume  
\[
a^2, b^2 \leq C(n_0, p_0, M_1, V).
\]
Then, there exists an explicitly computable constant $C(\varepsilon_0, n_0, p_0, M, M_1, V) > 0$  such that
\begin{equation}
\label{eqeedode}
(a - \nu_{\infty})^2 + (b - \pi_\infty)^2 + (c - \nu_{tr,\infty})^2 \leq C \left( (ad - c)^2 + (bc - d)^2 \right)
\end{equation}
for all $\varepsilon \in(0,\varepsilon_0]$.
\end{proposition}
\begin{proof}
We first introduce the following change of variable: Due to the non-negativity of the concentrations $a, b, c, d$, we define constants $\mu_1, \mu_2, \mu_3, \mu_4 \in [-1,\infty)$ such that
\[
a = \nu_\infty (1 + \mu_1), \qquad b = \pi_\infty (1 + \mu_2), \qquad c = \nu_{tr,\infty} (1 + \mu_3), \qquad d = \nu_{tr,\infty}' (1 + \mu_4),
\]
where $\nu_\infty$, $\pi_\infty$, $\nu_{tr,\infty}$ and $\nu_{tr,\infty}'$ are uniformly positive and bounded for all $\varepsilon \in(0,\varepsilon_0]$ in terms of $\varepsilon_0, n_0, p_0, M$ and $V$ by (the proof of) Propositions \ref{propbounds}. 
Thus, the boundedness of $a, b, c, d$  
implies the existence of a constant $K(\varepsilon_0, n_0, p_0, M, M_1, V)>0$, 
such that $\mu_i \in [-1,K]$ for all $1 \leq i \leq 4$. The left hand side of \eqref{eqeedode} expressed in terms of the $\mu_i$ rewrites as
\[
(a - \nu_{\infty})^2 + (b - \pi_\infty)^2 + (c - \nu_{tr,\infty})^2 = \nu_\infty^2 \mu_1^2 + \pi_\infty^2 \mu_2^2 + \nu_{tr,\infty}^2 \mu_3^2.
\]
Employing the equilibrium conditions \eqref{eqequiformulae}, we also find
\[
ad - c = \nu_\infty \nu_{tr,\infty}' (1 + \mu_1)(1 + \mu_4) - \nu_{tr,\infty}(1 + \mu_3) = \nu_{tr,\infty} \left[ (1+\mu_1)(1+\mu_4) - (1+\mu_3) \right]
\]
and
\[
bc - d = \pi_\infty \nu_{tr,\infty} (1 + \mu_2)(1 + \mu_3) - \nu_{tr,\infty}'(1 + \mu_4) = \nu_{tr,\infty}' \left[ (1+\mu_2)(1+\mu_3) - (1+\mu_4) \right].
\]
Moreover, the two conservation laws from the hypotheses rewrite as 
\begin{align}
n_0 \ol{\mu_n} \nu_\infty^2 \mu_1 (2 + \mu_1) - p_0 \ol{\mu_p} \pi_\infty^2 \mu_2 (2 + \mu_2) + \varepsilon \, \nu_{tr,\infty}^2 \mu_3(2 + \mu_3) &= 0, \label{CL1}\\
\nu_{tr,\infty}^2 \mu_3 (2 + \mu_3) + \nu_{tr,\infty}'^{\, 2} \mu_4 (2 + \mu_4) &= 0.\label{CL2}
\end{align}
The relations \eqref{CL1}--\eqref{CL2} allow to express $\varepsilon \mu_3$ and $\varepsilon \mu_4$ in terms of $\mu_1$ and $\mu_2$, although not explicitly:
\begin{align}
\varepsilon \mu_3 = - \frac{n_0 \ol{\mu_n} \nu_\infty^2}{\nu_{tr,\infty}^2} \frac{2 + \mu_1}{2 + \mu_3}\,  \mu_1 + \frac{p_0 \ol{\mu_p} \pi_\infty^2}{\nu_{tr,\infty}^2} \frac{2 + \mu_2}{2 + \mu_3} \, \mu_2 &=: - f_{1,3}(\mu_1, \mu_3) \mu_1 + f_{2,3}(\mu_2, \mu_3) \mu_2, \label{eqeedodec1} \\
\varepsilon \mu_4 = - \frac{\nu_{tr,\infty}^2}{\nu_{tr,\infty}'^{\, 2}} \frac{2 + \mu_3}{2 + \mu_4} \,\varepsilon \mu_3 =: - f_{3,4}(\mu_3, \mu_4)\, \varepsilon \mu_3 &=:  f_{1,4}(\mu_1, \mu_3,\mu_4) \mu_1 - f_{2,4}(\mu_2, \mu_3, \mu_4) \mu_2, \label{eqeedodec2}
\end{align}
where the last definition follows from inserting the previous expression \eqref{eqeedodec1} for $\varepsilon \mu_3$ while the factor 
$2+\mu_3$ is bounded in $[1,K+2]$ since $\mu_i \in [-1,K]$ for all $1 \leq i \leq 4$. Therefore, all the terms $f_{i,j}$ are uniformly positive as well as bounded from above:
\begin{align*}
&0 < \underline{C_{1,3}} \leq f_{1,3} \leq \ol{C_{1,3}} < \infty, \quad 0 < \underline{C_{2,3}} \leq f_{2,3} \leq \ol{C_{2,3}} < \infty, \\
0 < \underline{C_{3,4}} \leq f_{3,4} \leq \ol{C_{3,4}} < \infty, \quad &0 < \underline{C_{1,4}} \leq f_{1,4} \leq \ol{C_{1,4}} < \infty, \quad 0 < \underline{C_{2,4}} \leq f_{2,4} \leq \ol{C_{2,4}} < \infty.
\end{align*}
All constants $\underline{C_{i,j}}$ and $\ol{C_{i,j}}$ only depend on $\varepsilon_0$, $n_0$, $p_0$, $M$, $M_1$ and $V$, and there exist corresponding bounds $\underline{C} > 0$ and $\ol C > 0$ such that for all $i,j$
\[
\underline{C} \leq \underline{C_{i,j}}, \ol{C_{i,j}} \leq \ol C.
\]

In order to prove \eqref{eqeedode}, we show that under the constraints of the conservation laws \eqref{CL1}--\eqref{CL2}, respectively, the relations \eqref{eqeedodec1}--\eqref{eqeedodec2}, there exists a constant $C(\varepsilon_0, n_0, p_0, M, \underline{C}, \ol C) > 0$ for 
all $\varepsilon \in (0,\varepsilon_0]$ such that
\[
\frac{(a - \nu_{\infty})^2 + (b - \pi_\infty)^2 + (c - \nu_{tr,\infty})^2}{(ad - c)^2 + (bc - d)^2} \leq C,
\]
which is equivalent to 
\begin{equation}
\label{eqmufraction}
\frac{\nu_\infty^2 \mu_1^2 + \pi_\infty^2 \mu_2^2 + \nu_{tr,\infty}^2 \mu_3^2}{\nu_{tr,\infty}^2 \left[ (1+\mu_1)(1+\mu_4) - (1+\mu_3) \right]^2 + \nu_{tr,\infty}'^{\, 2} \left[ (1+\mu_2)(1+\mu_3) - (1+\mu_4) \right]^2} \leq C.
\end{equation}
Recall that $\nu_\infty^2 \leq \Gamma/n_0$, $\pi_\infty^2 \leq \Gamma/p_0$ and $\nu_{tr,\infty}^2, \nu_{tr,\infty}'^{\, 2} \in [\gamma, 1-\gamma]$ with $\gamma \in (0, 1/2)$ and $\Gamma \in (1/2, \infty)$ depending on $\varepsilon_0$, $n_0$, $p_0$ and $M$ for all $\varepsilon\in(0,\varepsilon_0]$ (cf. the proof of Proposition \ref{propbounds}). Since numerator and denominator of \eqref{eqmufraction} are sums of quadratic terms, it is sufficient to bound the denominator from below in terms of its numerator omitting the prefactors $\nu_\infty^2$, $\pi_\infty^2$, $\nu_{tr,\infty}^2$ and $\nu_{tr,\infty}'^{\,2}$, i.e. to prove that 
\begin{equation}\label{ef2}
(\ast) := \left[ (1+\mu_1)(1+\mu_4) - (1+\mu_3) \right]^2 + \left[ (1+\mu_2)(1+\mu_3) - (1+\mu_4) \right]^2
\ge C \left(\mu_1^2 + \mu_2^2 +  \mu_3^2\right).
\end{equation}
More precisely, we will prove that there exists a constant $c(\varepsilon_0, \underline{C}, \ol C) > 0$ for 
all $\varepsilon \in (0,\varepsilon_0]$ such that
\[
(\ast) = \left( \mu_1 + \mu_4 + \mu_1 \mu_4 - \mu_3 \right)^2 + \left( \mu_2 + \mu_3 + \mu_2 \mu_3 - \mu_4 \right)^2 \geq c \,(\mu_1^2 + \mu_2^2)
\]
and that
\[
(\ast) = \left( \mu_1 + \mu_4 + \mu_1 \mu_4 - \mu_3 \right)^2 + \left( \mu_2 + \mu_3 + \mu_2 \mu_3 - \mu_4 \right)^2 \geq \mu_3^2.
\]
For this reason, we distinguish four cases and we shall frequently use estimates like 
\[
\mu_i + \mu_i \mu_j = \mu_i (1 + \mu_j) \geq 0 \quad \mbox{iff} \quad \mu_i \geq 0 \quad \text{for all} \quad 1 \leq j \leq 4,
\]
since $\mu_j \geq -1$ for all $1 \leq j \leq 4$. We mention already here that all subsequent constants $c_1$, $c_2$ are strictly positive and depend only on $\varepsilon_0$,  $\underline{C}$ and $\ol C$ uniformly for $\varepsilon\in(0,\varepsilon_0]$.
\begin{description}
\item[\underline{Case 1: $\mu_1 \geq 0 \land \mu_2 \geq 0$:}] If $\mu_3 \geq 0$, then \eqref{CL2} implies $\mu_4 \leq 0$ and $\mu_2 + \mu_3 + \mu_2 \mu_3 - \mu_4 \geq \mu_2$. Moreover, $\mu_3 \geq 0$ yields 
\[
f_{2,3} \mu_2 \geq f_{1,3} \mu_1 \ \Rightarrow \ \ol{C_{2,3}} \mu_2 \geq \underline{C_{1,3}} \mu_1 \ \Rightarrow \ \mu_2 \geq \underline{C_{1,3}} / \ol{C_{2,3}} \, \mu_1
\]
and
\[
\mu_2 + \mu_3 + \mu_2 \mu_3 - \mu_4 \geq \mu_2 \geq \mu_2 / 2 + \underline{C_{1,3}} / ( 2 \, \ol{C_{2,3}} ) \mu_1 \geq c_1 (\mu_1 + \mu_2).
\]
Hence, $(\ast) \geq \left( \mu_2 + \mu_3 + \mu_2 \mu_3 - \mu_4 \right)^2 \geq c_2 (\mu_1^2 + \mu_2^2)$. Besides, $(\ast) \geq \left( \mu_2 + \mu_3 + \mu_2 \mu_3 - \mu_4 \right)^2 \geq \mu_3^2$.



If $\mu_3 < 0$, \eqref{CL2} yields $\mu_4 > 0$ and $\mu_1 + \mu_4 + \mu_1 \mu_4 - \mu_3 \geq \mu_1$. Since $\mu_3 < 0$, \eqref{eqeedodec1} implies
\[
f_{1,3} \mu_1 \geq f_{2,3} \mu_2 \ \Rightarrow \ \ol{C_{1,3}} \mu_1 \geq \underline{C_{2,3}} \mu_2 \ \Rightarrow \ \mu_1 \geq \underline{C_{2,3}} / \ol{C_{1,3}} \, \mu_2
\]
and
\[
\mu_1 + \mu_4 + \mu_1 \mu_4 - \mu_3 \geq \mu_1 \geq \mu_1 / 2 + \underline{C_{2,3}} / ( 2 \, \ol{C_{1,3}} ) \mu_2 \geq c_1 (\mu_1 + \mu_2).
\]
As above, $(\ast) \geq c_2 (\mu_1^2 + \mu_2^2)$. The signs $\mu_3 \leq 0 \le \mu_1,\mu_4$ yield $(\ast) \geq \left( \mu_1 + \mu_4 + \mu_1 \mu_4 - \mu_3 \right)^2 \geq \mu_3^2$.

\item[\underline{Case 2: $\mu_1 \geq 0 \land \mu_2 < 0$:}] \eqref{eqeedodec1} and \eqref{eqeedodec2} imply $\mu_3 \leq 0$ and $\mu_4 \geq 0$, and we deduce for all $\varepsilon\in(0,\varepsilon_0]$
\vspace{-1ex}
\begin{multline*}
\mu_1 + \mu_4 + \mu_1 \mu_4 - \mu_3 \geq \mu_4 - \mu_3 = \varepsilon^{-1}(f_{1,3} + f_{1,4}) \mu_1 - \varepsilon^{-1}(f_{2,3} + f_{2,4}) \mu_2 \\ \geq \varepsilon_0^{-1}(\underline{C_{1,3}} + \underline{C_{1,4}}) |\mu_1| + \varepsilon_0^{-1}(\underline{C_{2,3}} + \underline{C_{2,4}}) |\mu_2| \geq c_1 (|\mu_1| + |\mu_2|)
\end{multline*}
and, thus, $(\ast) \geq (\mu_1 + \mu_4 + \mu_1 \mu_4 - \mu_3)^2 \geq c_2 (\mu_1^2 + \mu_2^2)$. Since $\mu_2, \mu_3 \le 0 \le \mu_4$, we have 
\[
(\ast) \geq \left( \mu_4 - \mu_3 - \mu_2 (1 + \mu_3) \right)^2 \geq \mu_3^2.
\]

\item[\underline{Case 3: $\mu_1 < 0 \land \mu_2 \geq 0$:}] Here, $\mu_3 \geq 0$ due to \eqref{eqeedodec1} and, thus, $\mu_4 \leq 0$ by \eqref{eqeedodec2}, which yields for all $\varepsilon\in(0,\varepsilon_0]$ 
\vspace{-1ex}
\begin{multline*}
\mu_2 + \mu_3 + \mu_2 \mu_3 - \mu_4 \geq \mu_3 - \mu_4 = \varepsilon^{-1}(f_{2,3} + f_{2,4}) \mu_2 - \varepsilon^{-1}(f_{1,3} + f_{1,4}) \mu_1 \\ \geq \varepsilon_0^{-1}(\underline{C_{1,3}} + \underline{C_{1,4}}) |\mu_1| + \varepsilon_0^{-1}(\underline{C_{2,3}} + \underline{C_{2,4}}) |\mu_2| \geq c_1 (|\mu_1| + |\mu_2|)
\end{multline*}
and $(\ast) \geq (\mu_2 + \mu_3 + \mu_2 \mu_3 - \mu_4)^2 \geq c_2 (\mu_1^2 + \mu_2^2)$. And as $\mu_1,\mu_4 \le 0\le \mu_3$, one has 
\[
(\ast) \geq \left( \mu_3 - \mu_4 - \mu_1 (1 + \mu_4) \right)^2 \geq \mu_3^2.
\]

\item[\underline{Case 4: $\mu_1 < 0 \land \mu_2 < 0$:}] 
Supposing that $\mu_3 \geq 0$ and thus $\mu_4 \leq 0$ by \eqref{eqeedodec2}, we observe 
\[
|\mu_1 + \mu_4 + \mu_1 \mu_4 - \mu_3| = \mu_3 - \mu_1 - \mu_4(1 + \mu_1) \geq  -\mu_1.
\]
Furthermore, $\mu_3 \geq 0$ enables us to estimate 
\[
f_{1,3} \mu_1 \leq f_{2,3} \mu_2 \ \Rightarrow \ \ol{C_{1,3}} \mu_1 \leq \underline{C_{2,3}} \mu_2 \ \Rightarrow \ -\mu_1 \geq -\underline{C_{2,3}} / \ol{C_{1,3}} \, \mu_2.
\]
and
\[
|\mu_1 + \mu_4 + \mu_1 \mu_4 - \mu_3| \geq -\mu_1 \geq -\mu_1 / 2 - \underline{C_{2,3}} / ( 2 \, \ol{C_{1,3}} ) \mu_2 \geq c_1 (|\mu_1| + |\mu_2|).
\]
Hence, $(\ast) \geq (\mu_1 + \mu_4 + \mu_1 \mu_4 - \mu_3)^2 \geq c_2 (\mu_1^2 + \mu_2^2)$. The second estimate in terms of $\mu_3^2$ follows with $\mu_1, \mu_4 \le 0 \le \mu_3$ from
\[
(\ast) \geq \left( \mu_3 - \mu_4 - \mu_1 (1 + \mu_4) \right)^2 \geq \mu_3^2.
\]

In the opposite case that $\mu_3 < 0$ and thus $\mu_4 \geq 0$ due to \eqref{eqeedodec2}, we estimate 
\[
|\mu_2 + \mu_3 + \mu_2 \mu_3 - \mu_4| = \mu_4 - \mu_2 - \mu_3 (1+\mu_2) \geq -\mu_2
\]
and
\[
f_{2,3} \mu_2 \leq f_{1,3} \mu_1 \ \Rightarrow \ \ol{C_{2,3}} \mu_2 \leq \underline{C_{1,3}} \mu_1 \ \Rightarrow \ -\mu_2 \geq -\underline{C_{1,3}} / \ol{C_{2,3}} \, \mu_1.
\]
We, thus, arrive at
\[
|\mu_2 + \mu_3 + \mu_2 \mu_3 - \mu_4| \geq -\mu_2 \geq -\mu_2 / 2 - \underline{C_{1,3}} / ( 2 \, \ol{C_{2,3}} ) \mu_1 \geq c_1 (|\mu_1| + |\mu_2|)
\]
and $(\ast) \geq (\mu_2 + \mu_3 + \mu_2 \mu_3 - \mu_4)^2 \geq c_2 (\mu_1^2 + \mu_2^2)$. The corresponding inequality for $\mu_3$ reads
\[
(\ast) \geq \left( \mu_4 - \mu_3 - \mu_2 (1 + \mu_3) \right)^2  \geq \mu_3^2,
\]
which follows from $\mu_2,\mu_3 \leq 0 \leq \mu_4$.
\end{description}
The proof of the Proposition is now complete.
\end{proof}

\begin{notation}
From now on, $\|\cdot\|$ without further specification shall always denote the $L^2$-norm in $\Omega$.
\end{notation}

Within the subsequent Proposition \ref{propeedpde}, the expressions $(ad-c)^2$ and $(bc-d)^2$ on the right hand side of \eqref{eqeedode} will be generalised to $\|ad-c\|^2$ and $\|bc-d\|^2$ in Equation \eqref{eqeedodefunc}. We will later show in the proof of Theorem \ref{theoremeed} that $\|ad-c\|^2$ (and also $\|bc-d\|^2$) can be estimated from above via the reaction terms within the entropy production \eqref{eqproduction} when using the special choices $\sqrt{n/(n_0 \mu_n)}$, $\sqrt{p/(p_0 \mu_p)}$, $\sqrt{n_{tr}}$ and $\sqrt{n_{tr}'}$ for $a$, $b$, $c$ and $d$.

\begin{proposition}[Inhomogeneous Concentrations]
\label{propeedpde}
Let $a, b, c, d: \Omega \rightarrow \mathbb{R}$ be measurable, non-negative functions such that their squares satisfy the conservation laws
\[
n_0 \ol{\mu_n a^2} - p_0 \ol{\mu_p b^2} + \varepsilon \, \ol{c^2} = M = n_0 \ol{\mu_n} \nu_\infty^2 - p_0 \ol{\mu_p} \pi_\infty^2 + \varepsilon \, \nu_{tr,\infty}^2, \qquad \ol{c^2} + \ol{d^2} = 1 = \nu_{tr,\infty}^2 + \nu_{tr,\infty}'^{\, 2}
\]
for any $\varepsilon\in(0,\varepsilon_0]$ and arbitrary $\varepsilon_0>0$. In addition, we assume  
\[
\ol{a^2}, \, \ol{b^2} \leq C(n_0, p_0, M_1, V).
\]
Then, there exists an explicitly computable constant $C(\varepsilon_0, n_0, p_0, M, M_1, V) > 0$ such that 
\vspace{-1ex}
\begin{multline}
\label{eqeedodefunc}
\Big( \sqrt{\ol{a^2}} - \nu_\infty \Big)^2 + \Big( \sqrt{\ol{b^2}} - \pi_\infty \Big)^2 + \| c - \nu_{tr,\infty} \|^2 \\ \leq C \, \Big( \| ad - c \|^2 + \| bc - d \|^2 + \| \nabla a \|^2 + \| \nabla b \|^2 + \| a - \ol a \|^2 + \| b - \ol b \|^2 + \| c - \ol c \|^2 + \| d - \ol d \|^2 \Big)
\end{multline}
for all $\varepsilon \in(0,\varepsilon_0]$.

\begin{proof}
We divide the proof into two steps. In the first part, we shall derive lower bounds for the reaction terms $\| ad - c \|^2 + \| bc - d \|^2$ involving $( \ol a \, \ol d - \ol c )^2 + ( \ol b \, \ol c - \ol d )^2$. This will allow us to apply Proposition \ref{propeedodeconst} in the second step.
\paragraph{Step 1:} We show
\[
\| ad - c \|^2 \geq \frac{1}{2} \big(\ol a \, \ol d - \ol c \big)^2 - c_1 \left( \| a - \ol a \|^2 + \| b - \ol b \|^2 + \| c - \ol c \|^2 + \| d - \ol d \|^2 \right),
\]
and
\[
\| bc - d \|^2 \geq \frac{1}{2} \big(\ol b \, \ol c - \ol d \big)^2 - c_1 \left( \| a - \ol a \|^2 + \| b - \ol b \|^2 + \| c - \ol c \|^2 + \| d - \ol d \|^2 \right)
\]
with some explicitly computable constant $c_1>0$. For this reason, we define 
\[
\delta_1 := a - \ol a, \quad \delta_2 := b - \ol b, \quad \delta_3 := c - \ol c, \quad \delta_4 := d - \ol d
\]
and note that $\ol{\delta_1} = \ol{\delta_2} = \ol{\delta_3} = \ol{\delta_4} = 0$. Moreover, 
\[
| \ol a \, \ol d - \ol c |, \, | \ol b \, \ol c - \ol d | \leq C(n_0, p_0, M_1, V)
\]
due to Young's inequality, $\ol{a^2}$, $\ol{b^2} \leq C(n_0, p_0, M_1, V)$ and $\ol{c^2}$, $\ol{d^2} \leq 1$.

We now define
\[
S := \big\{ x \in \Omega \, \big| \, |\delta_1| \leq 1 \land |\delta_2| \leq 1 \land |\delta_3| \leq 1 \land |\delta_4| \leq 1 \big\}
\]
and split the squares of the $L^2(\Omega)$-norm as
\begin{equation}
\label{eqsplit}
\| ad - c \|^2 = \int_S (ad - c)^2 \, dx + \int_{\Omega \backslash S} (ad - c)^2 \, dx
\end{equation}
and
\[
\| bc - d \|^2 = \int_S (bc - d)^2 \, dx + \int_{\Omega \backslash S} (bc - d)^2 \, dx,
\]
respectively. In order to estimate the first integral in \eqref{eqsplit} from below, we write
\[
ad = (\ol a + \delta_1)(\ol d + \delta_4) = \ol a \ol d + \ol a \delta_4 + \ol d \delta_1 + \delta_1 \delta_4, \quad c = \ol c + \delta_3.
\]
This yields
\vspace{-1ex}
\begin{multline*}
\int_S (ad - c)^2 \, dx = \int_S (\ol a \ol d - \ol c)^2 \, dx + 2 \int_S (\ol a \ol d - \ol c)(\ol a \delta_4 + \ol d \delta_1 + \delta_1 \delta_4 - \delta_3) \, dx + \int_S (\ol a \delta_4 + \ol d \delta_1 + \delta_1 \delta_4 - \delta_3)^2 \, dx \\
\geq \frac{1}{2} \int_S (\ol a \ol d - \ol c)^2 \, dx - \int_S (\ol a \delta_4 + \ol d \delta_1 + \delta_1 \delta_4 - \delta_3)^2 \, dx \geq \frac{1}{2} \int_S (\ol a \ol d - \ol c)^2 \, dx - C(n_0, p_0, M_1, V) \left( \ol{\delta_1^2} + \ol{\delta_3^2} + \ol{\delta_4^2} \right)
\end{multline*}
where we used Young's inequality $2xy \geq -x^2/2 - 2y^2$ for $x,y \in \mathbb{R}$ in the second step and the boundedness of $\delta_i$, $1 \leq i \leq 4$, in the last step. Similarly, we deduce
\[
\int_S (bc - d)^2 \, dx \geq \frac{1}{2} \int_S (\ol b \ol c - \ol d)^2 \, dx - C(n_0, p_0, M_1, V) \left( \ol{\delta_2^2} + \ol{\delta_3^2} + \ol{\delta_4^2} \right).
\]
The second integral in \eqref{eqsplit} is mainly estimated by deriving an upper bound for the measure of $\Omega \backslash S$. For all $i \in \{1,\dots,4\}$ we have
\[
\left| \left\{ \delta_i^2 > 1 \right\} \right| = \int_{\left\{ \delta_i^2 > 1 \right\}} 1 \, dx \leq \int_{\Omega}\delta_i^2 \, dx = \ol{\delta_i^2}
\]
and, hence,
\[
|\Omega \backslash S| \leq \sum_{i=1}^4 \left| \left\{ \delta_i^2 > 1 \right\} \right| \leq \ol{\delta_1^2} + \ol{\delta_2^2} + \ol{\delta_3^2} + \ol{\delta_4^2}.
\]
As a consequence of $| \ol a \, \ol d - \ol c | \leq C(n_0, p_0, M_1, V)$, we obtain
\[
\int_{\Omega \backslash S} ( \ol a \, \ol d - \ol c )^2 \, dx \leq C(n_0, p_0, M_1, V) \, |\{\Omega \backslash S\}| \leq C(n_0, p_0, M_1, V) \left( \ol{\delta_1^2} + \ol{\delta_2^2} + \ol{\delta_3^2} + \ol{\delta_4^2} \right).
\]
This implies
\[
\int_{\Omega \backslash S} (ad - c)^2 \, dx \geq 0 \geq \frac{1}{2} \int_{\Omega \backslash S} ( \ol a \, \ol d - \ol c )^2 \, dx - C(n_0, p_0, M_1, V) \left( \ol{\delta_1^2} + \ol{\delta_2^2} + \ol{\delta_3^2} + \ol{\delta_4^2} \right)
\]
and, analogously,
\[
\int_{\Omega \backslash S} (bc - d)^2 \, dx \geq 0 \geq \frac{1}{2} \int_{\Omega \backslash S} ( \ol b \, \ol c - \ol d )^2 \, dx - C(n_0, p_0, M_1, V) \left( \ol{\delta_1^2} + \ol{\delta_2^2} + \ol{\delta_3^2} + \ol{\delta_4^2} \right).
\]
Taking the sum of both contributions to \eqref{eqsplit}, we finally arrive at
\begin{equation}
\label{eqstep1finaladc}
\|ad-c\|^2 \geq \frac{1}{2} \big(\ol a \, \ol d - \ol c \big)^2 - c_1(n_0, p_0, M_1, V) \left( \ol{\delta_1^2} + \ol{\delta_2^2} + \ol{\delta_3^2} + \ol{\delta_4^2} \right)
\end{equation}
and
\begin{equation}
\label{eqstep1finalbcd}
\|bc-d\|^2 \geq \frac{1}{2} \big(\ol b \, \ol c - \ol d \big)^2 - c_1(n_0, p_0, M_1, V) \left( \ol{\delta_1^2} + \ol{\delta_2^2} + \ol{\delta_3^2} + \ol{\delta_4^2} \right).
\end{equation}

\paragraph{Step 2:} We introduce constants $\mu_i \geq -1$, $1 \leq i \leq 4$, such that
\[
\ol{a^2} = \nu_\infty^2 (1 + \mu_1)^2, \quad \ol{b^2} = \pi_\infty^2 (1 + \mu_2)^2, \quad \ol{c^2} = \nu_{tr,\infty}^2 (1 + \mu_3)^2, \quad \ol{d^2} = \nu_{tr,\infty}'^{\, 2} (1 + \mu_4)^2.
\]
We recall that Proposition \ref{propbounds} guarantees the uniform positivity and boundedness of $\nu_\infty$, $\pi_\infty$, $\nu_{tr,\infty}$ and $\nu_{tr,\infty}'$ for all $\varepsilon \in (0, \varepsilon_0]$ in terms of $\varepsilon_0, n_0, p_0, M$ and $V$. Therefore, the bounds $\ol{a^2}$, $\ol{b^2} \leq C(n_0, p_0, M_1, V)$ and $\ol{c^2}$, $\ol{d^2} \leq 1$ give rise to a constant $K(\varepsilon_0, n_0, p_0, M, M_1, V) > 0$ such that $\mu_i \in [-1,K]$ for all $1 \leq i \leq 4$ uniformly for $\varepsilon \in (0, \varepsilon_0]$.

We now want to derive a formula for $\ol a$ in terms of $\delta_1$ and $\mu_1$. Since $\ol{a^2} - \ol a^2 = \|a - \ol a\|^2 = \|\delta_1\|^2 = \ol{\delta_1^2}$, one finds 
\begin{equation}
\label{eqaverage}
\ol a = \sqrt{\ol{a^2}} - \frac{\ol{\delta_1^2}}{\sqrt{\ol{a^2}} + \ol a} = \nu_\infty (1 + \mu_1) - \frac{\ol{\delta_1^2}}{\sqrt{\ol{a^2}} + \ol a}
\end{equation}
and analogous expressions for $\ol b$, $\ol c$ and $\ol d$: 
\[
\ol b = \pi_\infty (1 + \mu_2) - \frac{\ol{\delta_2^2}}{\sqrt{\ol{b^2}} + \ol b}, \qquad \ol c = \nu_{tr,\infty} (1 + \mu_3) - \frac{\ol{\delta_3^2}}{\sqrt{\ol{c^2}} + \ol c}, \qquad \ol d = \nu_{tr,\infty}' (1 + \mu_4) - \frac{\ol{\delta_4^2}}{\sqrt{\ol{d^2}} + \ol d}.
\]
Furthermore,
\[
\Big( \sqrt{\ol{a^2}} - \nu_\infty \Big)^{2} = \nu_\infty^2 \mu_1^2, \qquad \Big( \sqrt{\ol{b^2}} - \pi_\infty \Big)^{2} = \pi_\infty^2 \mu_2^2
\]
and, similarly,
\vspace{-1ex}
\begin{multline*}
\| c - \nu_{tr,\infty} \|^2 = \ol{c^2} - 2 \ol c \nu_{tr,\infty} + \nu_{tr,\infty}^2 \\
= \nu_{tr,\infty}^2 (1 + \mu_3)^2 - 2 \nu_{tr,\infty}^2 (1 + \mu_3) + \frac{2 \nu_{tr,\infty} \ol{\delta_3^2}}{\sqrt{\ol{c^2}} + \ol c} + \nu_{tr,\infty}^2 = \nu_{tr,\infty}^2 \mu_3^2 + \frac{2 \nu_{tr,\infty}}{\sqrt{\ol{c^2}} + \ol c} \, \ol{\delta_3^2}.
\end{multline*}
One observes that the expansions above in terms of $\ol{\delta_i^2}$ are singular if, e.g., $\ol{a^2}$ is zero. We therefore distinguish the following two cases.

\paragraph{\underline{Case 1: $\ol{a^2} \geq \kappa^2 \land \ol{b^2} \geq \kappa^2 \land \ol{c^2} \geq \kappa^2 \land \ol{d^2} \geq \kappa^2$:}} The constant $\kappa > 0$ will be chosen according to the calculations in the other Case 2. Here, we have
\[
\frac{1}{\sqrt{\ol{a^2}} + \ol a}, \quad \frac{1}{\sqrt{\ol{b^2}} + \ol b}, \quad \frac{1}{\sqrt{\ol{c^2}} + \ol c}, \quad \frac{1}{\sqrt{\ol{d^2}} + \ol d} \ \leq \frac{1}{\kappa}
\]
and 
\[
\frac{\nu_{tr,\infty}'}{\sqrt{\ol{a^2}} + \ol a}, \quad \frac{\nu_{tr,\infty}}{\sqrt{\ol{b^2}} + \ol b}, \quad \frac{\pi_\infty}{\sqrt{\ol{c^2}} + \ol c}, \quad \frac{\nu_\infty}{\sqrt{\ol{d^2}} + \ol d} \ \leq C(\kappa, \varepsilon_0, n_0, p_0, M, V)
\]
for all $\varepsilon \in (0, \varepsilon_0]$ due to the bounds on $\nu_\infty$ and $\pi_\infty$ from Proposition \ref{propbounds}. Equation \eqref{eqaverage} further implies
\vspace{-1ex}
\begin{multline*}
(\ol a \ol d - \ol c)^2 = \Bigg( \nu_\infty \nu_{tr,\infty}'(1 + \mu_1)(1 + \mu_4) - \frac{\nu_\infty (1 + \mu_1)}{\sqrt{\ol{d^2}} + \ol d} \, \ol{\delta_4^2} - \frac{\nu_{tr,\infty}' (1 + \mu_4)}{\sqrt{\ol{a^2}} + \ol a} \, \ol{\delta_1^2} \\ 
+ \frac{1}{(\sqrt{\ol{a^2}} + \ol a)(\sqrt{\ol{d^2}} + \ol d)} \, \ol{\delta_1^2} \, \ol{\delta_4^2} - \nu_{tr,\infty} (1 + \mu_3) + \frac{\ol{\delta_3^2}}{\sqrt{\ol{c^2}} + \ol c} \Bigg)^2 \\
\geq \nu_{tr,\infty}^2 \big((1 + \mu_1)(1 + \mu_4) - (1 + \mu_3) \big)^2 - c_2(\kappa, \varepsilon_0, n_0, p_0, M, M_1, V) \big( \ol{\delta_1^2} + \ol{\delta_3^2} + \ol{\delta_4^2} \big)
\end{multline*}
with some explicit constant $c_2$ thanks to $\nu_\infty \nu_{tr,\infty}' = \nu_{tr,\infty}$ (compare Equation \eqref{eqequiformulae}) and $|\mu_i|, \ol{\delta_i^2} \leq c_1(\varepsilon_0, n_0, p_0, M, M_1, V)$. In a similar fashion using $\pi_\infty \nu_{tr,\infty} = \nu_{tr,\infty}'$, one obtains
\vspace{-1ex}
\begin{multline*}
(\ol b \ol c - \ol d)^2 = \Bigg( \pi_\infty \nu_{tr,\infty}(1 + \mu_2)(1 + \mu_3) - \frac{\pi_\infty (1 + \mu_2)}{\sqrt{\ol{c^2}} + \ol c} \, \ol{\delta_3^2} - \frac{\nu_{tr,\infty} (1 + \mu_3)}{\sqrt{\ol{b^2}} + \ol b} \, \ol{\delta_2^2} \\ 
+ \frac{1}{(\sqrt{\ol{b^2}} + \ol b)(\sqrt{\ol{c^2}} + \ol c)} \, \ol{\delta_2^2} \, \ol{\delta_3^2} - \nu_{tr,\infty}' (1 + \mu_4) + \frac{\ol{\delta_4^2}}{\sqrt{\ol{d^2}} + \ol d} \Bigg)^2 \\
\geq \nu_{tr,\infty}'^{\, 2} \big((1 + \mu_2)(1 + \mu_3) - (1 + \mu_4) \big)^2 - c_2(\kappa, \varepsilon_0, n_0, p_0, M, M_1, V) \big( \ol{\delta_2^2} + \ol{\delta_3^2} + \ol{\delta_4^2} \big).
\end{multline*}

In order to finish the proof, it is --- according to Step 1 --- sufficient to show that
\vspace{-1ex}
\begin{multline*}
\nu_\infty^2 \mu_1^2 + \pi_\infty^2 \mu_2^2 + \nu_{tr,\infty}^2 \mu_3^2 + \frac{2 \nu_{tr, \infty}}{\sqrt{\ol{c^2}} + \ol c} \, \ol{\delta_3^2} \leq C_1 \, \bigg( \| \nabla a \|^2 + \| \nabla b \|^2 \\
+ \frac{1}{2} \big(\ol a \, \ol d - \ol c \big)^2 + \frac{1}{2} \big(\ol b \, \ol c - \ol d \big)^2 - 2\,c_1(n_0, p_0, M_1, V) \left( \ol{\delta_1^2} + \ol{\delta_2^2} + \ol{\delta_3^2} + \ol{\delta_4^2} \right) \bigg) + C_2 \left( \ol{\delta_1^2} + \ol{\delta_2^2} + \ol{\delta_3^2} + \ol{\delta_4^2} \right)
\end{multline*}
for appropriate constants $C_1, C_2 > 0$. But due to Step 2 it is sufficient to show that for suitable constants $C_1, C_2 > 0$, 
\vspace{-1ex}
\begin{multline*}
\nu_\infty^2 \mu_1^2 + \pi_\infty^2 \mu_2^2 + \nu_{tr,\infty}^2 \mu_3^2 + \frac{2 \nu_{tr, \infty}}{\sqrt{\ol{c^2}} + \ol c} \, \ol{\delta_3^2} \leq C_1 \, \bigg( \| \nabla a \|^2 + \| \nabla b \|^2 \\
+ \frac{\nu_{tr,\infty}^2}{2} \big((1 + \mu_1)(1 + \mu_4) - (1 + \mu_3) \big)^2 + \frac{\nu_{tr,\infty}'^{\, 2}}{2} \big((1 + \mu_2)(1 + \mu_3) - (1 + \mu_4) \big)^2 \\
- c_3(\kappa, \varepsilon_0, n_0, p_0, M, M_1, V) \left( \ol{\delta_1^2} + \ol{\delta_2^2} + \ol{\delta_3^2} + \ol{\delta_4^2} \right) \bigg) + C_2 \left( \ol{\delta_1^2} + \ol{\delta_2^2} + \ol{\delta_3^2} + \ol{\delta_4^2} \right).
\end{multline*}
Collecting all $\ol{\delta_i^2}$-terms on the right hand side, one only has to prove that
\vspace{-1ex}
\begin{multline*}
\nu_\infty^2 \mu_1^2 + \pi_\infty^2 \mu_2^2 + \nu_{tr,\infty}^2 \mu_3^2 \leq C_1 \bigg( \| \nabla a \|^2 + \| \nabla b \|^2 \\
+ \nu_{tr,\infty}^2 \big((1 + \mu_1)(1 + \mu_4) - (1 + \mu_3) \big)^2 + \nu_{tr,\infty}'^{\, 2} \big((1 + \mu_2)(1 + \mu_3) - (1 + \mu_4) \big)^2 \Bigg) \\
+ \Big(C_2 - C(C_1, \kappa, \varepsilon_0, n_0, p_0, M, M_1, V) \Big) \left( \ol{\delta_1^2} + \ol{\delta_2^2} + \ol{\delta_3^2} + \ol{\delta_4^2} \right)
\end{multline*}
or, equivalently,
\vspace{-1ex}
\begin{multline}
\label{eqcase1}
\left( \sqrt{\ol{a^2}} - \nu_\infty \right)^2 + \left( \sqrt{\ol{b^2}} - \pi_\infty \right)^2 + \left( \sqrt{\ol{c^2}} - \nu_{tr,\infty} \right)^2 \\
\leq C_1 \left( \Big( \sqrt{\ol{a^2}} \sqrt{\ol{d^2}} - \sqrt{\ol{c^2}} \Big)^2 + \Big( \sqrt{\ol{b^2}} \sqrt{\ol{c^2}} - \sqrt{\ol{d^2}} \Big)^2 + \| \nabla a \|^2 + \| \nabla b \|^2 \right) \\
+ \Big( C_2 - C(C_1, \kappa, \varepsilon_0, n_0, p_0, M, M_1, V) \Big) \left( \ol{\delta_1^2} + \ol{\delta_2^2} + \ol{\delta_3^2} + \ol{\delta_4^2} \right).
\end{multline}

In order to verify \eqref{eqcase1}, we start with the estimate
\[
\left( \sqrt{\ol{a^2}} - \nu_\infty \right)^2 \leq 2 \left[ \left( \sqrt{\frac{\ol{\mu_n a^2}}{\ol{\mu_n}}} - \nu_\infty \right)^{\! 2} + \left( \sqrt{\frac{\ol{\mu_n a^2}}{\ol{\mu_n}}} - \sqrt{\ol{a^2}} \right)^{\! 2} \, \right]
\]
and a corresponding one involving $b$. The last term on the right hand side satisfies
\[
\left( \sqrt{\frac{\ol{\mu_n a^2}}{\ol{\mu_n}}} - \sqrt{\ol{a^2}} \right)^{\! 2} = \frac{\left( \frac{\ol{\mu_n a^2}}{\ol{\mu_n}} - \ol{a^2} \right)^{\! 2}}{\left( \sqrt{\frac{\ol{\mu_n a^2}}{\ol{\mu_n}}} + \sqrt{\ol{a^2}} \right)^{\! 2}} \leq \frac{1}{\kappa^2} \left( \frac{\ol{\mu_n a^2}}{\ol{\mu_n}} - \ol{a^2} \right)^{\! 2} \leq c \int_\Omega \left| \nabla \sqrt{a^2} \right|^2 dx = c \, \| \nabla a \|^2
\]
due to Lemma \ref{lemmaflux} with a constant $c(\kappa, n_0, p_0, M_1, V) > 0$. Similarly,
\[
\left( \sqrt{\ol{b^2}} - \pi_\infty \right)^2 \leq c(\kappa, n_0, p_0, M_1, V) \left[ \left( \sqrt{\frac{\ol{\mu_p b^2}}{\ol{\mu_p}}} - \pi_\infty \right)^{\! 2} + \| \nabla b \|^2 \right].
\]
Proposition \ref{propeedodeconst} (with $a^2$, $b^2$, $c^2$ and $d^2$ therein replaced by $\ol{\mu_n a^2}/\ol{\mu_n}$, $\ol{\mu_p b^2}/\ol{\mu_p}$, $\ol{c^2}$ and $\ol{d^2}$) tells us that there exists an explicitly computable constant $C(\varepsilon_0, n_0, p_0, M, M_1, V)>0$ such that
\vspace{-1ex}
\begin{multline}
\label{eqapplypropode}
\left( \sqrt{\frac{\ol{\mu_n a^2}}{\ol{\mu_n}}} - \nu_\infty \right)^2 + \left( \sqrt{\frac{\ol{\mu_p b^2}}{\ol{\mu_p}}} - \pi_\infty \right)^2 + \left( \sqrt{\ol{c^2}} - \nu_{tr,\infty} \right)^2 \\
\leq C \left( \left( \sqrt{\frac{\ol{\mu_n a^2}}{\ol{\mu_n}}} \sqrt{\ol{d^2}} - \sqrt{\ol{c^2}} \right)^2 + \left( \sqrt{\frac{\ol{\mu_p b^2}}{\ol{\mu_p}}} \sqrt{\ol{c^2}} - \sqrt{\ol{d^2}} \right)^2 \right)
\end{multline}
for all $\varepsilon \in (0, \varepsilon_0]$. Using an analog expansion as before, we further deduce with $\ol{d^2} \leq 1$,
\vspace{-1ex}
\begin{multline*}
\left( \sqrt{\frac{\ol{\mu_n a^2}}{\ol{\mu_n}}} \sqrt{\ol{d^2}} - \sqrt{\ol{c^2}} \right)^2 = \left( \sqrt{\ol{a^2}} \sqrt{\ol{d^2}} - \sqrt{\ol{c^2}} + \left( \sqrt{\frac{\ol{\mu_n a^2}}{\ol{\mu_n}}} - \sqrt{\ol{a^2}} \right) \sqrt{\ol{d^2}} \right)^2 \\ 
\leq c(\kappa, n_0, p_0, M_1, V) \left( \Big( \sqrt{\ol{a^2}} \sqrt{\ol{d^2}} - \sqrt{\ol{c^2}} \Big)^2 + \| \nabla a \|^2 \right).
\end{multline*}
As a corresponding estimate holds true also for the other expression on the right hand side of \eqref{eqapplypropode}, we have shown that there exists a constant $C_1(\kappa, \varepsilon_0, n_0, p_0, M, M_1, V)>0$ independent of $\varepsilon$ for $\varepsilon \in (0, \varepsilon_0]$ such that
\vspace{-1ex}
\begin{multline*}
\left( \sqrt{\ol{a^2}} - \nu_\infty \right)^2 + \left( \sqrt{\ol{b^2}} - \pi_\infty \right)^2 + \left( \sqrt{\ol{c^2}} - \nu_{tr,\infty} \right)^2 \\
\leq C_1 \left( \Big( \sqrt{\ol{a^2}} \sqrt{\ol{d^2}} - \sqrt{\ol{c^2}} \Big)^2 + \Big( \sqrt{\ol{b^2}} \sqrt{\ol{c^2}} - \sqrt{\ol{d^2}} \Big)^2 + \| \nabla a \|^2 + \| \nabla b \|^2 \right).
\end{multline*}
Choosing $C_2>0$ now sufficiently large, Equation \eqref{eqcase1} holds true.

\paragraph{\underline{Case 2: $\ol{a^2} < \kappa^2 \lor \ol{b^2} < \kappa^2 \lor \ol{c^2} < \kappa^2 \lor \ol{d^2} < \kappa^2$:}} In this case, we will not need Proposition \ref{propeedodeconst} and we shall directly prove Equation \eqref{eqeedodefunc} employing only the result of Step 1. In fact, for $\kappa$ chosen sufficiently small, the states considered in Case 2 are necessarily bounded away from the equilibrium and the following arguments show that consequentially the right hand side of \eqref{eqeedodefunc} is also bounded away from zero, which allows to close the estimate \eqref{eqeedodefunc}. As a result of the hypotheses $\ol{a^2}, \, \ol{b^2} \leq C(n_0, p_0, M_1, V)$ and $\ol{c^2}, \, \ol{d^2} \leq 1$, we use Young's inequality to estimate $\ol a, \ol b, \ol c, \ol d \leq c(n_0, p_0, M_1, V)$ and
\[
\Big( \sqrt{\ol{a^2}} - \nu_\infty \Big)^2 + \Big( \sqrt{\ol{b^2}} - \pi_\infty \Big)^2 + \| c - \nu_{tr,\infty} \|^2 \leq C(\varepsilon_0, n_0, p_0, M, M_1, V)
\]
with $C > 0$ uniformly for $\varepsilon \in (0, \varepsilon_0]$. We stress that the subsequent cases are not necessarily exclusive.

\paragraph{\underline{Case 2.1: $\ol{c^2} < \kappa^2$:}} First, $\ol c = \sqrt{\ol c^2} \leq \sqrt{\ol{c^2}} < \kappa$. This yields 
\vspace{-1ex}
\begin{multline*}
\ol{d^2} = 1 - \ol{c^2} > 1 - \kappa^2 \Rightarrow \ol d^2 = \ol{d^2} - \ol{\delta_4^2} > 1 - \ol{\delta_4^2} - \kappa^2 \Rightarrow \\
\big( \ol b \, \ol c - \ol d \big)^2 \geq \ol d^2 - 2 \ol b \, \ol c \, \ol d > 1 - \ol{\delta_4^2} - \kappa^2 - 2 \ol b \, \ol d \, \kappa \geq 1 - \ol{\delta_4^2} - \kappa^2 - C(n_0, p_0, M_1, V) \, \kappa.
\end{multline*}
For $\kappa > 0$ sufficiently small, we then have $0 < 1 - C(n_0, p_0, M_1, V) \kappa - \kappa^2 \leq ( \ol b \, \ol c - \ol d )^2 + \ol{\delta_4^2}$ and, hence,
\vspace{-1ex}
\begin{multline*}
\Big( \sqrt{\ol{a^2}} - \nu_\infty \Big)^2 + \Big( \sqrt{\ol{b^2}} - \pi_\infty \Big)^2 + \| c - \nu_{tr,\infty} \|^2 \leq C(\varepsilon_0, n_0, p_0, M, M_1, V) \leq \\
K( \ol b \, \ol c - \ol d )^2 + K \ol{\delta_4^2} \leq 2 K \|bc - d\|^2 + (2 K c_1(n_0, p_0, M_1, V) + K) \left( \ol{\delta_1^2} + \ol{\delta_2^2} + \ol{\delta_3^2} + \ol{\delta_4^2} \right)
\end{multline*}
by \eqref{eqstep1finalbcd} with some $K(\kappa, \varepsilon_0, n_0, p_0, M, M_1, V) > 0$. Let us call the parameter $\kappa$ from above $\kappa_c$.

\paragraph{\underline{Case 2.2: $\ol{d^2} < \kappa^2$:}} Now $\ol d = \sqrt{\ol d^2} \leq \sqrt{\ol{d^2}} < \kappa$ and
\vspace{-1ex}
\begin{multline*}
\ol{c^2} = 1 - \ol{d^2} > 1 - \kappa^2 \Rightarrow \ol c^2 = \ol{c^2} - \ol{\delta_3^2} > 1 - \ol{\delta_3^2} - \kappa^2 \Rightarrow \\
\big( \ol a \, \ol d - \ol c \big)^2 \geq \ol c^2 - 2 \ol a \, \ol c \, \ol d > 1 - \ol{\delta_3^2} - \kappa^2 - 2 \ol a \, \ol c \, \kappa \geq 1 - \ol{\delta_3^2} - \kappa^2 - C(n_0, p_0, M_1, V) \, \kappa.
\end{multline*}
Again $\kappa > 0$ sufficiently small gives rise to $0 < 1 - C(n_0, p_0, M_1, V) \kappa - \kappa^2 \leq ( \ol a \, \ol d - \ol c )^2 + \ol{\delta_3^2}$ and
\vspace{-1ex}
\begin{multline*}
\Big( \sqrt{\ol{a^2}} - \nu_\infty \Big)^2 + \Big( \sqrt{\ol{b^2}} - \pi_\infty \Big)^2 + \| c - \nu_{tr,\infty} \|^2 \leq C(\varepsilon_0, n_0, p_0, M, M_1, V) \leq \\
K( \ol a \, \ol d - \ol c )^2 + K \ol{\delta_3^2} \leq 2 K \|ad - c\|^2 + (2 K c_1(n_0, p_0, M_1, V) + K) \left( \ol{\delta_1^2} + \ol{\delta_2^2} + \ol{\delta_3^2} + \ol{\delta_4^2} \right)
\end{multline*}
for some constant $K(\kappa, \varepsilon_0, n_0, p_0, M, M_1, V) > 0$ using \eqref{eqstep1finaladc}. This $\kappa > 0$ shall be denoted by $\kappa_d$.

\paragraph{\underline{Case 2.3: $\ol{a^2} < \kappa^2$:}} We first notice that $\ol a < \kappa$ and $2 \, \ol c \, \ol d \leq \ol c^2 + \ol d^2 \leq \ol{c^2} + \ol{d^2} = 1$. Now, we choose $\kappa_a := \kappa > 0$ sufficiently small such that $2 \kappa < \kappa_c^2$. Then, if $\ol{c^2} < 2 \kappa$, we have $\ol{c^2} < \kappa_c^2$, and the estimate
\[
\Big( \sqrt{\ol{a^2}} - \nu_\infty \Big)^2 + \Big( \sqrt{\ol{b^2}} - \pi_\infty \Big)^2 + \| c - \nu_{tr,\infty} \|^2 \leq 2 K \|bc - d\|^2 + (2 K c_1(n_0, p_0, M_1, V) + K) \left( \ol{\delta_1^2} + \ol{\delta_2^2} + \ol{\delta_3^2} + \ol{\delta_4^2} \right)
\]
with $K(\kappa, \varepsilon_0, n_0, p_0, M, M_1, V) > 0$ immediately follows from Case 2.1. And if $\ol{c^2} \geq 2 \kappa$, then
\[
\ol c^2 = \ol{c^2} - \ol{\delta_3^2} \geq 2 \kappa - \ol{\delta_3^2} \Rightarrow \big( \ol a \, \ol d - \ol c \big)^2 \geq \ol c^2 - 2 \, \ol a \, \ol c \, \ol d \geq 2 \kappa - \ol{\delta_3^2} -  \kappa = \kappa - \ol{\delta_3^2}.
\]
Consequently, $0 < \kappa \leq ( \ol a \, \ol d - \ol c )^2 + \ol{\delta_3^2}$ and
\vspace{-1ex}
\begin{multline*}
\Big( \sqrt{\ol{a^2}} - \nu_\infty \Big)^2 + \Big( \sqrt{\ol{b^2}} - \pi_\infty \Big)^2 + \| c - \nu_{tr,\infty} \|^2 \leq C(\varepsilon_0, n_0, p_0, M, M_1, V) \leq \\
K( \ol a \, \ol d - \ol c )^2 + K \ol{\delta_3^2} \leq 2 K \|ad - c\|^2 + (2 K c_1(n_0, p_0, M_1, V) + K) \left( \ol{\delta_1^2} + \ol{\delta_2^2} + \ol{\delta_3^2} + \ol{\delta_4^2} \right)
\end{multline*}
due to \eqref{eqstep1finaladc} with a constant $K(\kappa, \varepsilon_0, n_0, p_0, M, M_1, V) > 0$.

\paragraph{\underline{Case 2.4: $\ol{b^2} < \kappa^2$:}} Again $\ol b < \kappa$ and $2 \, \ol c \, \ol d \leq \ol c^2 + \ol d^2 \leq \ol{c^2} + \ol{d^2} = 1$. Here, we choose $\kappa_b := \kappa > 0$ sufficiently small such that $2\kappa < \kappa_d^2$. If $\ol{d^2} < 2\kappa$, we have $\ol{d^2} < \kappa_d^2$, and due to Case 2.2 there exists some $K(\kappa, \varepsilon_0, n_0, p_0, M, M_1, V) > 0$ such that
\[
\Big( \sqrt{\ol{a^2}} - \nu_\infty \Big)^2 + \Big( \sqrt{\ol{b^2}} - \pi_\infty \Big)^2 + \| c - \nu_{tr,\infty} \|^2 \leq 2 K \|ad - c\|^2 + (2 K c_1(n_0, p_0, M_1, V) + K) \Bigl( \ol{\delta_1^2} + \ol{\delta_2^2} + \ol{\delta_3^2} + \ol{\delta_4^2} \Bigr).
\]
If $\ol{d^2} \geq 2 \kappa$, then
\[
\ol d^2 = \ol{d^2} - \ol{\delta_4^2} \geq 2 \kappa - \ol{\delta_4^2} \Rightarrow \big( \ol b \, \ol c - \ol d \big)^2 \geq \ol d^2 - 2 \, \ol b \, \ol c \, \ol d \geq 2 \kappa - \ol{\delta_4^2} - \kappa = \kappa - \ol{\delta_4^2}.
\]
This implies $0 < \kappa \leq ( \ol b \, \ol c - \ol d )^2 + \ol{\delta_4^2}$ and
\vspace{-1ex}
\begin{multline*}
\Big( \sqrt{\ol{a^2}} - \nu_\infty \Big)^2 + \Big( \sqrt{\ol{b^2}} - \pi_\infty \Big)^2 + \| c - \nu_{tr,\infty} \|^2 \leq C(\varepsilon_0, n_0, p_0, M, M_1, V) \leq \\
K( \ol b \, \ol c - \ol d )^2 + K \ol{\delta_4^2} \leq 2 K \|bc - d\|^2 + (2 K c_1(n_0, p_0, M_1, V) + K) \left( \ol{\delta_1^2} + \ol{\delta_2^2} + \ol{\delta_3^2} + \ol{\delta_4^2} \right)
\end{multline*}
with $K(\kappa, \varepsilon_0, n_0, p_0, M, M_1, V) > 0$ employing \eqref{eqstep1finalbcd}.

All arguments within Step 2 remain valid, if we finally set $\kappa := \min(\kappa_a, \kappa_b, \kappa_c, \kappa_d)$. We also observe that the constants $K > 0$ above are independent of $\varepsilon \in (0, \varepsilon_0]$. And since $\kappa$ only depends on $n_0$, $p_0$, $M_1$ and $V$, we may skip the explicit dependence of $C_2$ on $\kappa$ at the end of Case 1. This finishes the proof.
\end{proof}
\end{proposition}

We already pointed out that $\|ad - c\|^2$ and $\|bc - d\|^2$ can be controlled by the reaction-terms of the entropy production, if we replace $a$, $b$, $c$, $d$ by $\sqrt{n/(n_0 \mu_n)}$, $\sqrt{p/(p_0 \mu_p)}$, $\sqrt{n_{tr}}$ and $\sqrt{n_{tr}'}$ (see the proof of Theorem \ref{theoremeed} in Section \ref{sectionconvergence} for details). In this proof, also $\|\nabla a\|^2$, $\|\nabla b\|^2$, $\|a - \ol a\|^2$ and $\|b - \ol b\|^2$ may be bounded by the entropy production. However, $\|c - \ol c\|^2$ and $\|d - \ol d\|^2$ may not be estimated with the help of  Poincar\'e's inequality since this would yield terms involving $\nabla n_{tr}$, which do not appear in the entropy production. 

Instead, we are able to derive the following estimates for $\|c - \ol c\|^2$ and $\|d - \ol d\|^2$, which describe an indirect diffusion transfer from $c$ to $b$ and from $d$ to $a$, respectively: Even if $c$ and $d$ are lacking an explicit diffusion term in the dynamical equations, they do experience indirect diffusive effects thanks to the reversible reaction dynamics and the diffusivity of $a$ and $b$. This is the interpretation of the following functional inequalities.

\begin{proposition}[Indirect Diffusion Transfer]
\label{propdifftrans}
Let $a,b,c,d:\Omega \rightarrow \mathbb{R}$ be non-negative functions such that 
\[
c^2 + d^2 = 1
\]
holds true a.e. in $\Omega$. Then, 
\[
\| c - \ol c \|^2 \leq 4 \big( \| b c - d \|^2 + \| b - \ol b \|^2 \big) \mbox{\qquad and \qquad} \| d - \ol d \|^2 \leq 4 \big( \| a d - c \|^2 + \| a - \ol a \|^2 \big).
\]

\begin{proof}
We only verify the second inequality; the first one can be checked along the same lines. First, we notice that 
\begin{equation}
\label{eqreacvariation}
\| \ol a d - c \| = \| a d - c + (\ol a - a) d \| \leq \| a d - c \| + \| a - \ol a \|
\end{equation}
because of the bound $0 \leq d \leq 1$. Besides, we deduce
\[
\| \ol a^2 d^2 - c^2 \| = \| (\ol a d + c)(\ol a d - c) \| \leq (1 + \ol a) \| \ol a d - c \|
\]
employing $0 \leq c,d \leq 1$. For the subsequent estimates, we need two auxiliary inequalities: For every function $f: \Omega \rightarrow \mathbb{R}$ and all $\lambda \in \mathbb{R}$, we have
\vspace{-1ex}
\begin{align}
\| f - \ol f \|^2 = \int_\Omega (f - \lambda + \lambda - \ol f)^2 dx &= \int_\Omega \Big( (f - \lambda)^2 - 2 (f - \lambda) (\ol f - \lambda) + (\ol f - \lambda)^2 \Big) dx \nonumber\\
&= \int_\Omega (f - \lambda)^2 dx - (\ol f - \lambda)^2 \leq \| f - \lambda \|^2.\label{eqaverageapprox}
\end{align}
And for all $x \geq 0$, one has
\[
\frac{1+x}{\sqrt{1 + x^2}} = \frac{\sqrt{1 + 2x + x^2}}{\sqrt{1 + x^2}} \leq \frac{\sqrt{2 (1 + x^2)}}{\sqrt{1 + x^2}} = \sqrt{2}.
\]
Since $c^2 + d^2 = 1$, we obtain
\vspace{-1ex}
\begin{multline*}
\| \ol a^2 d^2 - c^2 \| = \| \ol a^2 d^2 + d^2 - 1 \| = \| (1 + \ol a^2) d^2 - 1 \| = \big\| \big( \sqrt{1 + \ol a^2} \, d + 1 \big) \big( \sqrt{1 + \ol a^2} \, d - 1 \big) \big\| \\
\geq \big\| \sqrt{1 + \ol a^2} \, d - 1 \big\| = \sqrt{1 + \ol a^2} \, \bigg\| d - \frac{1}{\sqrt{1 + \ol a^2}} \bigg\| \geq \sqrt{1 + \ol a^2} \| d - \ol d \|.
\end{multline*}
where we applied \eqref{eqaverageapprox} in the last step. Consequently,
\[
\| d - \ol d \| \leq \frac{1}{\sqrt{1 + \ol a^2}} \| \ol a^2 d^2 - c^2 \| \leq \frac{1 + \ol a}{\sqrt{1 + \ol a^2}} \| \ol a d - c \| \leq \sqrt{2} \, \| \ol a d - c \|
\]
and 
\[
\| d - \ol d \|^2 \leq 2 \| \ol a d - c \|^2 \leq 4 ( \| a d - c \|^2 + \| a - \ol a \|^2 )
\]
using \eqref{eqreacvariation}.
\end{proof}
\end{proposition}

\section{EEP-inequality and convergence to the equilibrium}
\label{sectionconvergence}
We are now prepared to prove Theorem \ref{theoremeed}.

\begin{proof}[\bf Proof of Theorem \ref{theoremeed}]
Let $(n,p,n_{tr}) \in L^1(\Omega)^3$ be non-negative functions satisfying $n_{tr} \leq 1$, 
the conservation law $\ol n - \ol p + \varepsilon \ol{n_{tr}} = M$ and the $L^1$-bound $\ol n, \ol p \leq M_1$.
Keeping in mind that $\nu_\infty = \sqrt{n_\ast/n_0}$ and $\pi_\infty = \sqrt{p_\ast/p_0}$ (cf. Notation \ref{notabstract}), Proposition \ref{propeedineq} guarantees that there exists a positive constant $C_1(\gamma, \Gamma, M_1) > 0$ such that
\vspace{-1ex}
\begin{multline}
\label{eqprelimeedineq}
E(n,p,n_{tr}) - E(n_\infty, p_\infty, n_{tr,\infty}) \leq C_1 \, \Bigg( \int_\Omega \left( \frac{|J_n|^2}{n} + \frac{|J_p|^2}{p} \right) dx \\
+ n_0 \bigg( \sqrt{\ol{\frac{n}{n_0 \mu_n}}} - \nu_\infty \bigg)^2 + p_0 \bigg( \sqrt{\ol{\frac{p}{p_0 \mu_p}}} - \pi_\infty \bigg)^2 + \varepsilon \int_\Omega \big( \sqrt{n_{tr}} - \sqrt{n_{tr,\infty}} \big)^2 \, dx \Bigg).
\end{multline}

Next, we have to bound the second line of \eqref{eqprelimeedineq} in terms of the entropy production. To this end, we apply Proposition \ref{propeedpde} with the choices $a := \sqrt{n/(n_0 \mu_n)}$, $b := \sqrt{p/(p_0 \mu_p)}$, $c := \sqrt{n_{tr}}$ and $d := \sqrt{n_{tr}'}$ (as always $n_{tr}' = 1 - n_{tr}$). 
The hypotheses of this Proposition are fulfilled as a consequence of the conservation law $\ol n - \ol p + \varepsilon \ol{n_{tr}} = M$ and the $L^1$-bound $\ol n, \ol p \leq M_1$. As a result, we obtain
\vspace{-1ex}
\begin{multline*}
\left( \sqrt{\ol{\frac{n}{n_0 \mu_n}}} - \nu_\infty \right)^2 + \left( \sqrt{\ol{\frac{p}{p_0 \mu_p}}} - \pi_\infty \right)^2 + \| \sqrt{n_{tr}} - \sqrt{n_{tr,\infty}} \|^2 \\ 
\leq C_2 \Bigg( \left\| \sqrt{\frac{n n_{tr}'}{n_0 \mu_n}} - \sqrt{n_{tr}} \right\|^2 + \left\| \sqrt{\frac{p n_{tr}}{p_0 \mu_p}} - \sqrt{n_{tr}'} \right\|^2 + \left\| \nabla \sqrt{\frac{n}{n_0 \mu_n}} \right\|^2 + \left\| \nabla \sqrt{\frac{p}{p_0 \mu_p}} \right\|^2 \qquad\\
+ \left\| \sqrt{\frac{n}{n_0 \mu_n}} - \ol{\sqrt{\frac{n}{n_0 \mu_n}}} \right\|^2 + \left\| \sqrt{\frac{p}{p_0 \mu_p}} - \ol{\sqrt{\frac{p}{p_0 \mu_p}}} \right\|^2 + \big\| \sqrt{n_{tr}} - \ol{\sqrt{n_{tr}}} \big\|^2 + \big\| \sqrt{n_{tr}'} - \ol{\sqrt{n_{tr}'}} \big\|^2 \Bigg)
\end{multline*}
for all $\varepsilon \in (0, \varepsilon_0]$ with a constant $C_2(\varepsilon_0, n_0, p_0, M, M_1, V) > 0$. Thanks to Poincar\'e's inequality, we are able to bound the last two terms in the second line and the first two terms in the third line from above:
\[
\left\| \sqrt{\frac{n}{n_0 \mu_n}} - \ol{\sqrt{\frac{n}{n_0 \mu_n}}} \right\|^2 \leq C_P \left\| \nabla \sqrt{\frac{n}{n_0 \mu_n}} \right\|^2 = C_P \int_{\Omega} \left| \frac{1}{2} \sqrt{\frac{\mu_n}{n_0 n}} \nabla \Big( \frac{n}{\mu_n} \Big) \right|^2 dx = \frac{C_P}{4 n_0 \inf_\Omega \mu_n} \int_{\Omega} \frac{|J_n|^2}{n} \, dx
\]
and
\[
\left\| \sqrt{\frac{p}{p_0 \mu_p}} - \ol{\sqrt{\frac{p}{p_0 \mu_p}}} \right\|^2 \leq C_P \left\| \nabla \sqrt{\frac{p}{p_0 \mu_p}} \right\|^2 \leq \frac{C_P}{4 p_0 \inf_\Omega \mu_p} \int_{\Omega} \frac{|J_p|^2}{p} \, dx.
\]
Moreover, the elementary inequality $(\sqrt x - 1)^2 \leq (x-1) \ln(x)$ for $x>0$ gives rise to
\vspace{-1ex}
\begin{multline*}
\left\| \sqrt{\frac{n n_{tr}'}{n_0 \mu_n}} - \sqrt{n_{tr}} \right\|^2 = \int_\Omega n_{tr} \left( \sqrt{\frac{n n_{tr}'}{n_0 \mu_n n_{tr}}} - 1 \right)^2 dx \leq \int_\Omega n_{tr} \left( \frac{n n_{tr}'}{n_0 \mu_n n_{tr}} - 1 \right) \ln \left( \frac{n n_{tr}'}{n_0 \mu_n n_{tr}} \right) dx \\ 
= \int_\Omega \left( \frac{n (1-n_{tr})}{n_0 \mu_n} - n_{tr} \right) \ln \left( \frac{n (1-n_{tr})}{n_0 \mu_n n_{tr}} \right) dx = -\tau_n \int_\Omega R_n \ln \left( \frac{n (1-n_{tr})}{n_0 \mu_n n_{tr}} \right) dx
\end{multline*}
and similarly
\[
\left\| \sqrt{\frac{p n_{tr}}{p_0 \mu_p}} - \sqrt{n_{tr}'} \right\|^2 \leq -\tau_p \int_\Omega R_p \ln \left( \frac{p n_{tr}}{p_0 \mu_p (1-n_{tr})} \right) dx.
\]
Proposition \ref{propdifftrans} further implies that
\vspace{-1ex}
\begin{multline*}
\big\| \sqrt{n_{tr}} - \ol{\sqrt{n_{tr}}} \big\|^2 + \big\| \sqrt{n_{tr}'} - \ol{\sqrt{n_{tr}'}} \big\|^2 \leq \\
4 \left( \left\| \sqrt{\frac{n}{n_0 \mu_n}} - \ol{\sqrt{\frac{n}{n_0 \mu_n}}} \right\|^2 + \left\| \sqrt{\frac{p}{p_0 \mu_p}} - \ol{\sqrt{\frac{p}{p_0 \mu_p}}} \right\|^2 + \left\| \sqrt{\frac{n n_{tr}'}{n_0 \mu_n}} - \sqrt{n_{tr}} \right\|^2 + \left\| \sqrt{\frac{p n_{tr}}{p_0 \mu_p}} - \sqrt{n_{tr}'} \right\|^2 \right).
\end{multline*}
Combining the above estimates, we arrive at
\vspace{-1ex}
\begin{multline*}
\left( \sqrt{\ol{\frac{n}{n_0 \mu_n}}} - \nu_\infty \right)^2 + \left( \sqrt{\ol{\frac{p}{p_0 \mu_p}}} - \pi_\infty \right)^2 + \| \sqrt{n_{tr}} - \sqrt{n_{tr,\infty}} \|^2 \\
\leq C_3 \int_{\Omega} \left(\frac{|J_n|^2}{n} + \frac{|J_p|^2}{p} - R_n \ln \left( \frac{n(1-n_{tr})}{n_0 \mu_n n_{tr}} \right) - R_p \ln \left( \frac{p n_{tr}}{p_0 \mu_p (1-n_{tr})} \right) \right)
\end{multline*}
with a constant $C_3(\varepsilon_0, \tau_n, \tau_p, n_0, p_0, M, M_1, V) > 0$ uniformly for $\varepsilon \in (0, \varepsilon_0]$. With respect to \eqref{eqprelimeedineq}, we now find
\vspace{-1ex}
\begin{align*}
n_0 &\left( \sqrt{\ol{\frac{n}{n_0 \mu_n}}} - \nu_\infty \right)^2 + p_0 \left( \sqrt{\ol{\frac{p}{p_0 \mu_p}}} - \pi_\infty \right)^2 + \varepsilon \int_\Omega \big( \sqrt{n_{tr}} - \sqrt{n_{tr,\infty}} \big)^2 \, dx \\
&\qquad\leq \max\{n_0, p_0, \varepsilon_0\} \left( \left( \sqrt{\ol{\frac{n}{n_0 \mu_n}}} - \nu_\infty \right)^2 + \left( \sqrt{\ol{\frac{p}{p_0 \mu_p}}} - \pi_\infty \right)^2 + \| \sqrt{n_{tr}} - \sqrt{n_{tr,\infty}} \|^2 \right) \\
&\qquad\leq C_3 \max\{n_0, p_0, \varepsilon_0\} \int_{\Omega} \left(\frac{|J_n|^2}{n} + \frac{|J_p|^2}{p} - R_n \ln \left( \frac{n(1-n_{tr})}{n_0 \mu_n n_{tr}} \right) - R_p \ln \left( \frac{p n_{tr}}{p_0 \mu_p (1-n_{tr})} \right) \right).
\end{align*}
And since the constant $C_1$ in \eqref{eqprelimeedineq} only depends on $\varepsilon_0$, $n_0$, $p_0$, $M$, $M_1$ and $V$ (via the constants $\gamma$ and $\Gamma$), we have finally proven that
\vspace{-1ex}
\begin{multline*}
E(n,p,n_{tr}) - E(n_\infty, p_\infty, n_{tr,\infty}) \\
\leq C_4 \int_{\Omega} \left(\frac{|J_n|^2}{n} + \frac{|J_p|^2}{p} - R_n \ln \left( \frac{n(1-n_{tr})}{n_0 \mu_n n_{tr}} \right) - R_p \ln \left( \frac{p n_{tr}}{p_0 \mu_p (1-n_{tr})} \right) \right) dx
\end{multline*}
for a constant $C_4(\varepsilon_0, \tau_n, \tau_p, n_0, p_0, M, M_1, V) > 0$ independent of $\varepsilon \in (0, \varepsilon_0]$.
\end{proof}

Theorem \ref{theoremeed} provides an upper bound for the relative entropy in terms of the entropy production. This already implies exponential convergence of the relative entropy. The subsequent Proposition now yields a lower bound for the relative entropy involving the $L^1$-distance of the solution to the equilibrium. This will allow us to establish exponential convergence in $L^1$.

\begin{proposition}[Csisz\'ar--Kullback--Pinsker inequality]
\label{propentropyl1}
Let $\varepsilon_0$, $n_0$, $p_0$, $M$, $M_1$ and $V$ be positive constants. Then, there exists an explicit constant $C_{\mathrm{CKP}} > 0$ such that for all $\varepsilon \in (0, \varepsilon_0]$, the equilibrium $(n_\infty, p_\infty, n_{tr,\infty}) \in X$ from Theorem \ref{theoremequilibrium} and all non-negative functions $(n,p,n_{tr}) \in L^1(\Omega)^3$ satisfying $n_{tr} \leq 1$, the conservation law 
\[
\ol n - \ol p + \varepsilon \ol{n_{tr}} = M
\]
and the $L^1$-bound 
\[
\ol n, \ol p \leq M_1,
\]
the following Csisz\'ar--Kullback--Pinsker-type inequality holds true:
\[
E(n,p,n_{tr}) - E(n_\infty, p_\infty, n_{tr,\infty}) \geq C_{\mathrm{CKP}} \big( \| n - n_\infty \|_{L^1(\Omega)}^2 + \| p - p_\infty \|_{L^1(\Omega)}^2 + \varepsilon \| n_{tr} - n_{tr,\infty} \|_{L^1(\Omega)}^2 \big).
\]

\begin{proof}
Due to Lemma \ref{lemmarelativeentropy}, we know that the relative entropy reads
\vspace{-1ex}
\begin{multline*}
E(n,p,n_{tr}) - E(n_\infty, p_\infty, n_{tr,\infty}) = \\
\int_{\Omega} \left( n \ln \frac{n}{n_\infty} - (n-n_\infty) + p \ln \frac{p}{p_\infty} - (p-p_\infty) + \varepsilon \int_{n_{tr,\infty}}^{n_{tr}} \left( \ln \left( \frac{s}{1-s} \right) - \ln \left( \frac{n_{tr,\infty}}{1-n_{tr,\infty}} \right) \right) ds \right) dx.
\end{multline*}
Similar to Proposition \ref{propeedineq}, we employ the mean-value theorem and observe that
\[
\frac{d}{ds} \ln \left( \frac{s}{1-s} \right) = \frac{1}{s(1-s)} \geq 4
\]
for all $s \in (0,1)$. Thus, there exists some $\sigma(s)$ between $n_{tr,\infty}$ and $s$ such that
\begin{multline*}
\varepsilon \int_\Omega \int_{n_{tr,\infty}}^{n_{tr}} \left( \ln \left( \frac{s}{1-s} \right) - \ln \left( \frac{n_{tr,\infty}}{1 - n_{tr,\infty}} \right) \right) ds \, dx = \varepsilon \int_\Omega \int_{n_{tr,\infty}}^{n_{tr}} \frac{1}{\sigma(s)(1 - \sigma(s))} (s - n_{tr,\infty}) \, ds \, dx \\ 
\geq 4 \varepsilon \int_\Omega \int_{n_{tr,\infty}}^{n_{tr}} (s - n_{tr,\infty}) \, ds \, dx = 2 \varepsilon \int_\Omega (n_{tr} - n_{tr,\infty})^2 \, dx \geq 2 \varepsilon  \| n_{tr} - n_{tr,\infty} \|_{L^1(\Omega)}^2
\end{multline*}
where the last inequality holds true since $|\Omega| = 1$. Moreover, we utilise the  Csisz\'ar--Kullback--Pinsker-inequality from Lemma \ref{lemmackp} to estimate
\[
\int_\Omega \left( n \ln \left( \frac{n}{n_\infty} \right) - (n - n_\infty) \right) dx \geq \frac{3}{2 \ol n + 4 \ol{n_\infty}} \| n - n_\infty \|_{L^1(\Omega)}^2 \geq c \| n - n_\infty \|_{L^1(\Omega)}^2
\]
where $c(\varepsilon_0, n_0, p_0, M, M_1, V) > 0$ is a positive constant independent of $\varepsilon \in (0, \varepsilon_0]$. As a corresponding estimate holds true also for $p$, we have verified that
\[
E(n,p,n_{tr}) - E(n_\infty, p_\infty, n_{tr,\infty}) \geq C \big( \| n - n_\infty \|_{L^1(\Omega)}^2 + \| p - p_\infty \|_{L^1(\Omega)}^2 + \varepsilon \| n_{tr} - n_{tr,\infty} \|_{L^1(\Omega)}^2 \big)
\]
for some $C(\varepsilon_0, n_0, p_0, M, M_1, V) > 0$ uniformly for $\varepsilon \in (0, \varepsilon_0]$.
\end{proof}
\end{proposition}

Now, we are able to prove exponential convergence in relative entropy and in $L^1$.

\begin{proof}[\bf Proof of Theorem \ref{theoremconvergence}]
We first prove exponential convergence of the relative entropy 
\[
\psi(t) := E(n,p,n_{tr})(t) - E(n_\infty, p_\infty, n_{tr,\infty})
\]
using a Gronwall argument as stated in \cite{Wil65}. To this end, we choose $0 < t_0 \leq t_1 \leq t < T$ and rewrite the entropy-production law as 
\begin{equation}
\label{eqentropydisslaw}
\psi(t_1) - \psi(t) = \int_{t_1}^{t} D(n, p, n_{tr})(s) \, ds \geq K \int_{t_1}^{t} \psi(s) \, ds
\end{equation}
where we applied Theorem \ref{theoremeed} with $K := C_{\mathrm{EED}}^{-1}$ in the second step. Furthermore, we set
\[
\Psi(t_1) := \int_{t_1}^{t} \psi(s) \, ds = - \int_{t}^{t_1} \psi(s) \, ds
\]
and obtain from \eqref{eqentropydisslaw} the estimate $K \Psi(t_1) \leq \psi(t_1) - \psi(t)$ which yields
\[
\frac{d}{dt_1} \Big( \Psi(t_1) e^{K t_1} \Big) = - \psi(t_1) e^{K t_1} + K \Psi(t_1) e^{K t_1} \leq - \psi(t) e^{K t_1}.
\]
Integrating this inequality from $t_1 = t_0$ to $t_1 = t$ and observing that $\Psi(t) = 0$ gives rise to
\[
-\Psi(t_0) e^{K t_0} \leq - \frac{\psi(t)}{K} \big( e^{K t} - e^{K t_0} \big).
\]
As a consequence of \eqref{eqentropydisslaw} with $t_1 = t_0$, one has $- \Psi(t_0) \geq (\psi(t) - \psi(t_0)) / K$ and, hence,
\[
- \psi(t_0) e^{K t_0} \leq - \psi(t) e^{K t}.
\]
But this is equivalent to
\begin{equation}
\label{eqexpoprelim}
E(n,p,n_{tr})(t) - E(n_\infty, p_\infty, n_{tr,\infty}) \leq (E(n,p,n_{tr})(t_0) - E_\infty) e^{-K (t-t_0)},
\end{equation}
for all $t\ge t_0>0$. 
In order to conclude that 
\[
E(n,p,n_{tr})(t) - E(n_\infty, p_\infty, n_{tr,\infty}) \leq (E_I - E_\infty) e^{-K t},
\]
for all $t \geq 0$, we observe that the rate $K$ is independent of $t_0$ and that the entropy $E(n,p,n_{tr})(t_0)$ extends
in \eqref{eqexpoprelim} continuously to $t_0 \rightarrow 0$ since $n,p \in C([0,T);L^2(\Omega))$ for all $T>0$ by Theorem 
\ref{theoremsolution}.
This results in the announced exponential decay of the relative entropy, while the exponential convergence in $L^1$ follows from Proposition \ref{propentropyl1}.
\end{proof}

\begin{proof}[\bf Proof of Corollary \ref{coro}]
We first prove that the linearly growing $L^{\infty}$-bounds
together with parabolic regularity for system \eqref{eqsystem} 
 and assumption \eqref{eqpot}  
entail polynomially growing $W^{1,q}$-bounds, $q \in (1,\infty)$,
for $n$ and $p$. 
To this end, we consider 
\begin{equation*}
\partial_t n = \nabla \cdot J_n + \frac{1}{\tau_n} \left( n_{tr} - \frac{n}{n_0 e^{-V_n}} \bigl(1-n_{tr}\bigr) \right),\qquad J_n=e^{-V_n} \nabla \big(n\, e^{V_n} \bigr).
\end{equation*}
and introduce the variable $w=n \, e^{V_n}$. We observe that $\nabla \cdot J_n = \nabla \cdot \big(e^{-V_n} \nabla w \big)=e^{-V_n}\left( \Delta w - \nabla V_n \cdot \nabla w\right)$ and thus, 
\begin{equation}\label{eq:om}
\partial_t w = \Delta w - \nabla V_n \cdot \nabla w + \frac{e^{V_n}}{\tau_n} \left( n_{tr} - \frac{1-n_{tr}}{n_0} w  \right).
\end{equation}
Under the assumptions of Corollary \ref{coro}, Eq. \eqref{eq:om}
is of the form
\[
\partial_t w - \Delta w = f_1 + f_2 w + f_3 \cdot \nabla w
\]
where $f_i \in C([0, \infty), L^\infty(\Omega))$ for $i \in \{1,2\}$, $f_3 \in C([0, \infty), L^\infty(\Omega)^m)$ and $\hat n \cdot \nabla w = 0$ on $\partial \Omega$. Testing this equation with $-(q-1) |\nabla w|^{q-2} \Delta w$ yields
\begin{align*}
\frac1q \frac{d}{dt} &\int_\Omega |\nabla w|^q \, dx = \int_\Omega |\nabla w|^{q-2} \nabla w \cdot \nabla \partial_t w \, dx \\
&= - \int_\Omega \bigg( (q-2) |\nabla w|^{q-4} \nabla w \Delta w \cdot \nabla w + |\nabla w|^{q-2} \Delta w \bigg) \partial_t w \, dx = - \int_\Omega (q-1) |\nabla w|^{q-2} \Delta w \, \partial_t w \, dx \\
&= - \int_\Omega (q-1) |\nabla w|^{q-2} |\Delta w|^2 dx - \int_\Omega (q-1) |\nabla w|^{q-2} \Delta w \big( f_1 + f_2 w + f_3 \cdot \nabla w \big) dx.
\end{align*}
Using the inequalities $|ab| \leq (a^2 + b^2)/2$ and $(a+b+c)^2 \leq 3(a^2 + b^2 + c^2)$ for $a,b,c \in \mathbb R$, we find
\[
\frac1q \frac{d}{dt} \int_\Omega |\nabla w|^q \, dx \leq -\frac12 \int_\Omega (q-1) |\nabla w|^{q-2} |\Delta w|^2 \, dx + \frac32 \int_\Omega (q-1) |\nabla w|^{q-2} \big( f_1^2 + f_2^2 w^2 + f_3^2 |\nabla w|^2 \big) dx.
\]
Together with $C > 0$ satisfying $|f_i(t,x)^2| \leq C$ for all $i \in \{1,2,3\}$, $t \geq 0$ and a.a. $x \in \Omega$, we derive
\[
\frac1q \frac{d}{dt} \int_\Omega |\nabla w|^q \, dx \leq -\frac12 \int_\Omega (q-1) |\nabla w|^{q-2} |\Delta w|^2 \, dx + \frac{3C}{2} \int_\Omega (q-1) |\nabla w|^{q-2} \big( 1 + w^2 + |\nabla w|^2 \big) dx.
\]
An integration by parts and Young's inequality with $C_1 > 0$ give rise to 
\vspace*{-1ex}
\begin{multline*}
\int_\Omega (q-1) |\nabla w|^{q-2} \nabla w \cdot \nabla w \, dx = - \int_{\Omega} \bigg( (q-1)(q-2) |\nabla w|^{q-4} \nabla w \Delta w \cdot \nabla w + (q-1) |\nabla w|^{q-2} \Delta w \bigg) w \, dx \\
= - \int_\Omega (q-1)^2 |\nabla w|^{q-2} \Delta w \, w \, dx \leq \frac{1}{3C} \int_\Omega (q-1)^1 |\nabla w|^{q-2} |\Delta w|^2 \, dx + C_1 \int_\Omega |\nabla w|^{q-2} w^2 \, dx.
\end{multline*}
Hence, there exists a constant $C_2 > 0$ such that
\[
\frac{d}{dt} \int_\Omega |\nabla w|^q \, dx \leq C_2 \int_\Omega |\nabla w|^{q-2} (1 + w^2) \, dx \leq (A + B \, t^2) \int_\Omega |\nabla w|^{q-2} \, dx,
\]
where $A, B > 0$ result from the linearly growing $L^\infty$-bounds from \eqref{npLinfty}. For any fixed $t_0 > 0$ and all $t \geq t_0$, we now have
\[
\| \nabla w(t) \|_{L^q(\Omega)}^q \leq \| \nabla w(t_0) \|_{L^q(\Omega)}^q + \int_{t_0}^t (A + B \, s^2) \| \nabla w(s) \|_{L^q(\Omega)}^{q-2} \, ds.
\]
A Gronwall lemma (see e.g. \cite{Bee75}) now proves the desired polynomial growth of $\| \nabla w \|_{L^q(\Omega)}$ and $\| \nabla n \|_{L^q(\Omega)}$:
\[
\| \nabla w(t) \|_{L^q(\Omega)} \leq \bigg( \| \nabla w(t_0) \|_{L^q(\Omega)}^2 + A (t-t_0) + \frac{B}{3} (t^3 - t_0^3) \bigg)^{\frac12}.
\]

Next, we use (see e.g. \cite{Tay}) the Gagliardo-Nirenberg-Moser interpolation  
inequality
\begin{equation}
\|n\|_{L^{\infty}} \le G(\Omega) \| n \|_{W^{1,2n}}^{\frac12} \| n \|_{L^{2n}}^{\frac12} \,. \label{GagNir}
\end{equation}
Then, interpolating with the exponentially decaying $L^1$-norm of $n-n_\infty$, we obtain
\begin{align*}
\|n(t,\cdot)\|_{L^{\infty}} &\le \|n_{\infty}\|_{L^{\infty}} +
\|n-n_{\infty}\|_{L^{\infty}} \le  \|n_{\infty}\|_{L^{\infty}} +
G\|n-n_\infty\|_{W^{1,2n}}^{\frac12} \|n-n_{\infty}\|_{L^{\infty}}^{\frac12 - \frac{1}{4n}}
\|n-n_{\infty}\|_{L^1}^{\frac{1}{4n}}\le K
\end{align*}
due to the exponential convergence to equilibrium \eqref{expo}. The estimate for $p$ follows in the same way.
\end{proof}

\begin{proof}[\bf Proof of Theorem \ref{theoremeed'}]
Our goal is to derive an estimate of the form 
\[
E_0(n, p) - E_0(n_{\infty,0}, p_{\infty,0}) \leq C_{EEP} D_0(n, p)
\]
by applying the EEP-inequality from Theorem \ref{theoremeed} directly to the functions $n$, $p$ and $n^{eq}_{tr}$. However, since we assume that $n$ and $p$ satisfy 
\[
\ol n - \ol p = M,
\]
the triple $(n, p, n^{eq}_{tr})$ does not satisfy the conservation law with right hand side $M$ but 
\[
\ol n - \ol p + \varepsilon \ol{n^{eq}_{tr}} = M + \varepsilon \ol{n^{eq}_{tr}}.
\]
In order to resolve this issue, we shall apply the EEP-inequality from Theorem \ref{theoremeed} to a suitably defined sequence of functions $(n_\varepsilon, p_\varepsilon, n_{tr,\varepsilon}) \in L^1(\Omega)^3$ which fulfil $\| n_{tr,\varepsilon} \|_{L^\infty(\Omega)} \leq 1$, the $L^1$-bound $\ol{n_\varepsilon}, \ol{p_\varepsilon} \leq M_1$ and the conservation law 
\[
\ol{n_\varepsilon} - \ol{p_\varepsilon} + \varepsilon \ol{n_{tr,\varepsilon}} = M.
\]
A convenient choice is $n_\varepsilon \coleq n$, $p_\varepsilon \coleq p + \varepsilon \ol{n_{tr}^{eq}}$ and $n_{tr,\varepsilon} \coleq n^{eq}_{tr}$, where $n^{eq}_{tr}=n^{eq}_{tr}(n,p)$ as defined in \eqref{ntreq}.
For this choice, we derive the stated EEP-estimate for the case $\varepsilon = 0$ via the following steps, which are proven below:
\begin{align}
E_0(n, p) - E_0(n_{\infty,0}, p_{\infty,0}) &= \lim\limits_{\varepsilon \rightarrow 0} \big( E(n_\varepsilon, p_\varepsilon, n_{tr,\varepsilon}) - E(n_\infty, p_\infty, n_{tr,\infty}) \big) \label{eqconventropy} \\ 
&\leq \lim\limits_{\varepsilon \rightarrow 0} \big( C_{EEP} D(n_\varepsilon, p_\varepsilon, n_{tr,\varepsilon}) \big) \label{eqeedineq} \\
&= C_{EEP} D(n, p, n^{eq}_{tr}) \label{eqconvdissip} 
= C_{EEP} D_0(n, p) 
\end{align}
We recall that $n$ and $p$ are assumed to satisfy $E_0(n, p)<\infty$ and $D_0(n, p),D(n, p, n^{eq}_{tr})<\infty$, which implies that $D_0(n, p)=D(n, p, n^{eq}_{tr})$ as discussed in the introduction.

\paragraph{Step 1. Proof of \eqref{eqconventropy}:} We first show, that with $(n_\varepsilon, p_\varepsilon, n_{tr,\varepsilon})=(n, p_\varepsilon, n^{eq}_{tr})$
\begin{equation}\label{llll}
E_0(n, p) = \lim\limits_{\varepsilon \rightarrow 0} E(n_\varepsilon, p_\varepsilon, n_{tr,\varepsilon}).
\end{equation}
Recalling that 
\[
E(n,p_\varepsilon,n^{eq}_{tr}) = \int_{\Omega} \left( n \ln \frac{n}{n_0 \mu_n} - (n-n_0\mu_n) + p_\varepsilon \ln \frac{p_\varepsilon}{p_0 \mu_p} - (p_\varepsilon-p_0\mu_p) + \varepsilon \int_{1/2}^{n^{eq}_{tr}} \ln \left( \frac{s}{1-s} \right) ds \right) dx,
\]
we first notice that $p_\varepsilon=p + \varepsilon \ol{n_{tr}^{eq}} \rightarrow p$ monotonically decreasing for $\varepsilon \rightarrow 0$ for all $x\in\Omega$. Thus, by using $\ol{n_{tr}^{eq}} \le 1$
and the elementary estimate  
$p_\varepsilon \ln p_\varepsilon
\le 2 p\, (\ln p + \ln 2)$ for $p\ge \max\{\varepsilon_0,1\}$, 
the Lebesgue dominated  convergence theorem, the $L^1$-bounds $\ol{n},\ol{n_\varepsilon}, \ol{p},\ol{p_\varepsilon} \leq M_1$ and
$E_0(n, p)<\infty$ imply the convergence of the 
$p_\varepsilon$-integral in \eqref{llll}.
The convergence of the third integral follows directly from
\[
\left| \varepsilon \int_{1/2}^{n^{eq}_{tr}(x)} \ln \frac{s}{1-s} \, ds \right| \leq \varepsilon \int_{1/2}^1 \ln \frac{s}{1-s} \, ds \xrightarrow{\varepsilon \rightarrow 0} 0.
\]

Using analog arguments, the convergence 
\[
E_0(n_{\infty,0}, p_{\infty,0}) = \lim\limits_{\varepsilon \rightarrow 0} E(n_\infty, p_\infty, n_{tr,\infty})
\]
follows from observing the monotone convergence 
$n_\ast \rightarrow n_{\ast,0}$ and $p_\ast \rightarrow p_{\ast,0}$ for $\varepsilon \rightarrow 0$ due to \eqref{eqninfformula} in Proposition \ref{propbounds}, which directly implies the monotone convergence 
$n_\infty \rightarrow n_{\infty,0}$ and $p_\infty \rightarrow p_{\infty,0}$ for all $x\in\Omega$, where 
$(n_\infty, p_\infty, n_{tr,\infty})$ and  $(n_{\infty,0}, p_{\infty,0})$
are defined in 
 \eqref{eqequilibrium} and  \eqref{eqequilibrium0}, respectively.


\paragraph{Step 2. Proof of \eqref{eqeedineq}:} The functions $(n_\varepsilon, p_\varepsilon, n_{tr,\varepsilon}) = (n, p + \varepsilon \ol{n_{tr}^{eq}}, n^{eq}_{tr}) \in L^1(\Omega)^3$ satisfy $\| n_{tr,\varepsilon} \|_{L^\infty(\Omega)} \leq 1$, the conservation law
\[
\ol n_{\varepsilon} - \ol{p_\varepsilon} + \varepsilon \ol{n_{tr,\varepsilon}} = \ol n - \ol p = M
\]
as well as the $L^1$-bounds $\ol n_{\varepsilon} \leq M_1$ and $\ol{p_\varepsilon} \leq \ol{p} + \varepsilon'$ where $\varepsilon \in (0, \varepsilon'] \subset (0, \varepsilon_0]$. Because of $\ol{p} < M_1$, we have $\ol{p_\varepsilon} \leq M_1$ for $\varepsilon' > 0$ sufficiently small.
As a consequence, 
\[
E(n_\varepsilon, p_\varepsilon, n_{tr,\varepsilon}) - E(n_\infty, p_\infty, n_{tr,\infty}) \leq C_{EEP} D(n_\varepsilon, p_\varepsilon, n_{tr,\varepsilon})
\]
where $C_{EEP} > 0$ is the same constant as in Theorem \ref{theoremeed}.

\paragraph{Step 3. Proof of \eqref{eqconvdissip}:} As the constant $C_{EEP} > 0$ is independent of $\varepsilon \in (0, \varepsilon_0]$, it suffices to show that 
\[
\lim\limits_{\varepsilon \rightarrow 0} D(n_\varepsilon, p_\varepsilon, n_{tr,\varepsilon}) = D(n, p, n^{eq}_{tr}).
\]
To this end, we consider the representation
\begin{multline*}
D(n_\varepsilon, p_\varepsilon, n_{tr,\varepsilon}) = \int_{\Omega} \bigg(\frac{|J_n|^2}{n} + \frac{|\nabla p|^2}{p_\varepsilon} + 2 \nabla p \cdot \nabla V_p + p_\varepsilon |\nabla V_p|^2 \\
- R_n \ln \left( \frac{n(1-n^{eq}_{tr})}{n_0 \mu_n n^{eq}_{tr}} \right) + \frac{1}{\tau_p} \biggl( \frac{p_\varepsilon}{p_0 \mu_p} n^{eq}_{tr} - (1 - n^{eq}_{tr}) \biggr) \biggl( \ln  \frac{p_\varepsilon n^{eq}_{tr}}{p_0 \mu_p}  - \ln(1 - n^{eq}_{tr}) \biggr) \bigg) \, dx,
\end{multline*}
where we have already taken into account that $n_\varepsilon=n$, $\nabla p_\varepsilon = \nabla p$ and $n_{tr,\varepsilon}=n^{eq}_{tr}$ for all $\varepsilon > 0$. 

We note first that the convergence of the second, third and forth 
integral follows from the pointwise convergence of $p_{\varepsilon}$ for all $x\in \Omega$ and from the 
Lebesgue dominated convergence theorem by estimating 
\[
0 \leq \frac{|\nabla p|^2}{p_\varepsilon} + 2 \nabla p \cdot \nabla V_p + p_\varepsilon |\nabla V_p|^2 \leq \frac{|\nabla p|^2}{p} + 2 \nabla p \cdot \nabla V_p + p |\nabla V_p|^2 + (p_\varepsilon - p) |\nabla V_p|^2 \leq \frac{|J_p|^2}{p} + \varepsilon_0 \ol{n^{eq}_{tr}} |\nabla V_p|^2,
\]
where the function on the right hand side is integrable due to the finiteness of $D(n, p, n^{eq}_{tr})$. 

Secondly, the product 
\[
\left( \frac{p_\varepsilon}{p_0 \mu_p} n^{eq}_{tr} - (1 - n^{eq}_{tr}) \right)\!\left( \ln  \frac{p_\varepsilon n^{eq}_{tr}}{p_0 \mu_p} \!-\!\ln(1 - n^{eq}_{tr}) \right)\!\rightarrow\!\left( \frac{p}{p_0 \mu_p} n^{eq}_{tr} - (1 - n^{eq}_{tr}) \right)\!\left( \ln  \frac{p n^{eq}_{tr}}{p_0 \mu_p} -\ln(1 - n^{eq}_{tr}) \right)
\]
converges pointwise for all $x\in\Omega$ as $\varepsilon\to0$. 
In order to conclude the convergence of the corresponding integral via the Lebesgue dominated convergence theorem, 
we use similar to Step 1 the elementary 
inequality $p_\varepsilon \ln p_\varepsilon
\le 2 p\, (\ln p + \ln 2)$ for $p\ge \max\{\varepsilon_0,1\}$ and the finiteness of $D(n, p, n^{eq}_{tr})$.
This yields
\[
\lim\limits_{\varepsilon \rightarrow 0} \int_\Omega \frac{1}{\tau_p} \left( \frac{p_\varepsilon}{p_0 \mu_p} n^{eq}_{tr}\!-\!(1 - n^{eq}_{tr}) \right)\!\left( \ln  \frac{p_\varepsilon n^{eq}_{tr}}{p_0 \mu_p} -\ln(1 - n^{eq}_{tr}) \right) dx = -\int_\Omega R_p \ln \left( \frac{p n^{eq}_{tr}}{p_0 \mu_p (1-n^{eq}_{tr})} \right) dx
\]
and therefore, $D(n_\varepsilon, p_\varepsilon, n_{tr,\varepsilon}) \rightarrow D(n, p, n^{eq}_{tr})$ for $\varepsilon \rightarrow 0$.
\end{proof}

\begin{proof}[\bf Proof of Theorem \ref{theoremconvergence'}]
We only have to check that the assuptions on the finiteness of the entropy $E$ and its production $D$ within Theorem \ref{theoremeed'} are satisfied. The claim of this Theorem then follows from the same arguments as in the proof of Theorem \ref{theoremconvergence}.

Due to the uniform $L^\infty$-bounds \eqref{npLinfty'} of $n(t)$ and $p(t)$ for all $t \geq 0$, we know that $E_0(n, p) < \infty$ for all $t \geq 0$. Similarly, we deduce that $D(n, p, n^{eq}_{tr})$ and $D_0(n, p)$ are finite for all strictly positive $t > 0$ since $n$, $p$ are bounded away from zero 
and $n^{eq}_{tr}$ is bounded away from zero and one uniformly in $\Omega$. 

Finally, the lower bounds \eqref{npbounds'} guarantee similar to Theorem \ref{theoremeed} that solutions satisfy the 
weak entropy production law \eqref{wep'} for all $t_0 > 0$.
\end{proof}

\section*{Appendix: Proof of the existence-theorem}
\label{sectionexistence}
\begin{proof}[\bf Proof of Theorem \ref{theoremsolution}] In order to simplify the notation, we set the parameter $\tau_n := \tau_p := 1$ and $n_0 := p_0 := 1$ throughout the proof. All arguments also apply in the case of arbitrary values for $\tau_n$, $\tau_p$, $n_0$ and $p_0$. The structure of system \eqref{eqsystem} can be further simplified via introducing new variables 
\[
u := e^\frac{V_n}{2}n, \qquad v := e^\frac{V_p}{2}p.
\]
One obtains
\[
\nabla u = \frac{1}{2} e^\frac{V_n}{2} \nabla V_n n + e^\frac{V_n}{2} \nabla n \qquad \mbox{and} \qquad \Delta u = e^\frac{V_n}{2} \Big( \Delta n + \nabla n \cdot \nabla V_n + \frac{1}{4} n |\nabla V_n|^2 + \frac{1}{2} n \Delta V_n \Big)
\]
which results in 
\begin{align*}
\partial_t u &= e^\frac{V_n}{2} \partial_t n = e^\frac{V_n}{2} \Big( \Delta n + \nabla n \cdot \nabla V_n + n \Delta V_n + R_n \Big) = \Delta u - e^\frac{V_n}{2} \Big( \frac{1}{4} n |\nabla V_n|^2 - \frac{1}{2} n \Delta V_n \Big) + e^\frac{V_n}{2} R_n \\
&= \Delta u + \Big( \frac{1}{2} \Delta V_n - \frac{1}{4} |\nabla V_n|^2 \Big) u + e^\frac{V_n}{2} n_{tr} - e^{V_n} u (1 - n_{tr}).
\end{align*}
Analogously, we derive 
\[
\partial_t v = \Delta v + \Big( \frac{1}{2} \Delta V_p - \frac{1}{4} |\nabla V_p|^2 \Big) v + e^\frac{V_p}{2} (1-n_{tr}) - e^{V_p} v n_{tr}.
\]
For convenience, we also introduce the abbreviations
\[
A_n := \frac{1}{2} \Delta V_n - \frac{1}{4} |\nabla V_n|^2 \in L^\infty(\Omega), \qquad A_p := \frac{1}{2} \Delta V_p - \frac{1}{4} |\nabla V_p|^2 \in L^\infty(\Omega)
\]
as well as $\alpha, \beta > 0$ such that the following estimates hold true a.e. in $\Omega$:
\[
|A_n|, |A_p| \leq \alpha \qquad \mbox{and} \qquad e^\frac{V_n}{2}, e^\frac{V_p}{2}, e^{V_n}, e^{V_p} \leq \beta.
\]

Next, we introduce the new variable 
\begin{equation}
\label{eqsumntrntr'}
n_{tr}' := 1 - n_{tr}
\end{equation}
for reasons of symmetry. In fact, we can prove the positivity of $n_{tr}'$ in the same way as for $n_{tr}$, which then implies the desired bound $0 \leq n_{tr} \leq 1$. A further ingredient for establishing the positivity of the variables $u$, $v$, $n_{tr}$ and $n_{tr}'$ is to project them onto $[0, \infty)$ and $[0, 1]$, respectively, on the right hand side of the PDE-system. In this context, we use $X^+\! := \max(X, 0)$ to denote the positive part of an arbitrary function $X$ and $X^{[0,1]} := \min(\max(X, 0), 1)$ for the projection of $X$ to the interval $[0,1]$. The modified system now reads
\begin{equation}
\label{eqsystemmodified}
\begin{cases}
\begin{aligned}
\partial_t u - \Delta u &= A_n u^+ + e^\frac{V_n}{2} n_{tr}^{[0,1]} - e^{V_n} u^+ n_{tr}'^{[0,1]}, \\
\partial_t v - \Delta v &= A_p v^+ + e^\frac{V_p}{2} n_{tr}'^{[0,1]} - e^{V_p} v^+ n_{tr}^{[0,1]}, \\
\varepsilon \, \partial_t n_{tr} &= n_{tr}'^{[0,1]} - e^\frac{V_p}{2} v^+ n_{tr}^{[0,1]} - n_{tr}^{[0,1]} + e^\frac{V_n}{2} u^+ n_{tr}'^{[0,1]}, \\
\varepsilon \, \partial_t n_{tr}' &= n_{tr}^{[0,1]} - e^\frac{V_n}{2} u^+ n_{tr}'^{[0,1]} - n_{tr}'^{[0,1]} + e^\frac{V_p}{2} v^+ n_{tr}^{[0,1]}.
\end{aligned}
\end{cases}
\end{equation}
The no-flux boundary conditions of \eqref{eqsystem} transfer to similar conditions on $u$ and $v$. In detail, we have
\[
e^{-\frac{V_n}{2}} \nabla u = \nabla n + \frac{1}{2} n \nabla V_n
\]
and, hence,
\[
\nabla n + n \nabla V_n = e^{-\frac{V_n}{2}} \Big( \nabla u + \frac{1}{2} u \nabla V_n \Big).
\]
Therefore, the corresponding boundary conditions for $u$ and $v$ read
\begin{equation}
\label{eqbcuv}
\hat n \cdot \Big(\nabla u + \frac{1}{2} u \nabla V_n \Big) = \hat n \cdot \Big(\nabla v + \frac{1}{2} v \nabla V_p \Big) = 0.
\end{equation}
Furthermore, we assume that the corresponding initial states satisfy
\begin{equation}
\label{eqicuv}
(u_I, v_I, n_{tr,I}, n_{tr,I}') \in L_+^\infty(\Omega)^4, \qquad n_{tr,I} + n_{tr,I}' = 1.
\end{equation}
In this situation, $\|n_{tr,I}\|_{L^\infty(\Omega)} + \|n_{tr,I}'\|_{L^\infty(\Omega)} \geq 1$ and we set
\[
I := \|u_I\|_{L^\infty(\Omega)} + \|v_I\|_{L^\infty(\Omega)} + \|n_{tr,I}\|_{L^\infty(\Omega)} + \|n_{tr,I}'\|_{L^\infty(\Omega)} \geq 1.
\]
We now aim to apply Banach's fixed-point theorem to obtain a solution of \eqref{eqsystemmodified}--\eqref{eqicuv}. 

\paragraph{Step 1: Definition of the fixed-point iteration.} For any time $T > 0$ (to be chosen sufficiently small in the course of the fixed-point argument), we introduce the space
\[
X_T := C([0, T], L^2(\Omega))^4 
\]
and the closed subspace
\vspace{-1ex}
\begin{multline*}
M_T := \big\{ (u, v, n_{tr}, n_{tr}') \in X_T \, \big| \, (u(0), v(0), n_{tr}(0), n_{tr}'(0)) = (u_I, v_I, n_{tr,I}, n_{tr,I}') \; \land \\ 
\max_{0 \leq t \leq T} \| u(t) \|_{L^2(\Omega)}, \, \max_{0 \leq t \leq T} \| v(t) \|_{L^2(\Omega)}, \, \max_{0 \leq t \leq T} \| n_{tr}(t) \|_{L^2(\Omega)}, \, \max_{0 \leq t \leq T} \| n_{tr}'(t) \|_{L^2(\Omega)} \leq 2I \; \land \\
\| u \|_{L^\infty((0, T) \times \Omega)}, \, \| v \|_{L^\infty((0, T) \times \Omega)} \leq 2I \big\} \subset X_T.
\end{multline*}
The fixed-point mapping $\mathcal{S}: X_T \rightarrow X_T$ is now defined via
\[
\mathcal{S}(\wt u, \wt v, \wt n_{tr}, \wt n_{tr}') := (u, v, n_{tr}, n_{tr}')
\]
where $(u, v, n_{tr}, n_{tr}')$ is the solution of the following PDE-system subject to the boundary and initial conditions specified above:
\begin{equation}
\label{eqsystemiter}
\begin{cases}
\begin{aligned}
\partial_t u - \Delta u &= A_n \wt u^+ + e^\frac{V_n}{2} \wt n_{tr}^{[0,1]} - e^{V_n} \wt u^+ \wt n_{tr}'^{[0,1]} && \hspace{-8pt} =: \wt f_1, \\
\partial_t v - \Delta v &= A_p \wt v^+ + e^\frac{V_p}{2} \wt n_{tr}'^{[0,1]} - e^{V_p} \wt v^+ \wt n_{tr}^{[0,1]} && \hspace{-8pt} =: \wt f_2, \\
\varepsilon \, \partial_t n_{tr} &= \wt n_{tr}'^{[0,1]} - e^\frac{V_p}{2} \wt v^+ \wt n_{tr}^{[0,1]} - \wt n_{tr}^{[0,1]} + e^\frac{V_n}{2} \wt u^+ \wt n_{tr}'^{[0,1]} && \hspace{-8pt} =: \wt f_3, \\
\varepsilon \, \partial_t n_{tr}' &= \wt n_{tr}^{[0,1]} - e^\frac{V_n}{2} \wt u^+ \wt n_{tr}'^{[0,1]} - \wt n_{tr}'^{[0,1]} + e^\frac{V_p}{2} \wt v^+ \wt n_{tr}^{[0,1]} && \hspace{-8pt} =: \wt f_4.
\end{aligned}
\end{cases}
\end{equation}

We first show that $(u, v, n_{tr}, n_{tr}') = \mathcal{S}(\wt u, \wt v, \wt n_{tr}, \wt n_{tr}') \in X_T$ provided $(\wt u, \wt v, \wt n_{tr}, \wt n_{tr}') \in X_T$. Due to $\wt f_1$, $\wt f_2 \in L^2((0, T) \times \Omega)$, it is known from classical PDE-theory (see e.g. \cite{Chi00}) that 
\[
u, v \in W_2(0, T) = \big\{ f \in L^2((0, T), H^1(\Omega)) \, | \, \partial_t f \in L^2((0, T), H^1(\Omega)^*) \big\} \hookrightarrow C([0, T], L^2(\Omega)).
\]
And since
\[
n_{tr}(t) = n_{tr}(0) + \frac{1}{\varepsilon} \int_0^t \wt f_3(s) \, ds
\]
for almost all $t \in [0, T]$, we deduce 
\[
\|n_{tr}(t)\|_{L^2(\Omega)} \leq \|n_{tr}(0)\|_{L^2(\Omega)} + \frac{1}{\varepsilon} \int_0^t \|\wt f_3(s)\|_{L^2(\Omega)} \, ds \leq I + \frac{T}{\varepsilon} \max_{0 \leq t \leq T} \| \wt f_3(s) \|_{L^2(\Omega)}.
\]
Hence, $n_{tr} \in L^\infty((0, T), L^2(\Omega))$. Moreover, we observe that for $[0, T] \ni t_n \rightarrow t \in [0, T]$ that 
\[
\| n_{tr}(t_n) - n_{tr}(t) \|_{L^2(\Omega)} \leq \frac{1}{\epsilon} \left| \int_t^{t_n} \| \wt f_3(s) \|_{L^2(\Omega)} \, ds \right| \leq \frac{|t_n - t|}{\epsilon} \max_{0 \leq t \leq T} \| \wt f_3(s) \|_{L^2(\Omega)} \xrightarrow{n\to\infty} 0.
\]
This proves $n_{tr} \in C([0, T], L^2(\Omega))$. The same arguments can be applied to $n_{tr}'$.

\paragraph{Step 2: Invariance of $M_T$.} Now, let $(\wt u, \wt v, \wt n_{tr}, \wt n_{tr}') \in M_T$.
Similar to the strategy of e.g. \cite{Ali79,GMS07,MWZ08}, we perform the subsequent calculations for any $q \in 2 \mathbb{N}_+$ and every $t \in [0, T]$:
\vspace{-1ex}
\begin{multline*}
\int_0^t \frac{d}{ds} \int_\Omega \frac{u^q}{q} \, dx \, ds = \int_0^t \int_\Omega u^{q-1} \partial_t u \, dx \, ds = \int_0^t \int_\Omega u^{q-1} \Delta u \, dx \, ds + \int_0^t \int_\Omega u^{q-1} \wt f_1 \, dx \, ds \\
\leq -(q-1) \int_0^t \int_\Omega u^{q-2} |\nabla u|^2 \, dx \, ds - \frac{1}{2} \int_0^t \int_{\partial \Omega} u^q \, \hat n \cdot \nabla V_n \, d\sigma \, ds + \| \wt f_1 \|_{L^\infty((0, T) \times \Omega)} \int_0^t \int_\Omega |u|^{q-1} \, dx \, ds \\
\leq (2 \alpha I + \beta + 2 \beta I) \int_0^t \|u\|_{L^q(\Omega)}^{q-1} \, ds.
\end{multline*}
Note that the first two terms in the second line are both non-positive due to $q \in 2\mathbb{N}$ and assumption \eqref{eqpot}. Introducing $\gamma := 2 \alpha I + \beta + 2 \beta I$, we obtain
\begin{equation}
\label{eqinvarbound}
\| u(t) \|_{L^q(\Omega)}^q - \| u(0) \|_{L^q(\Omega)}^q \leq q \gamma \int_0^t \| u(s) \|_{L^q(\Omega)}^{q-1} \, ds.
\end{equation}
This inequality already implies a linear bound on the $L^\infty$-norm of $u$ as we shall see below (cf. \cite{Bee75}). We define 
\[
U(t) := q\gamma \int_0^t \| u(s) \|_{L^q(\Omega)}^{q-1} \, ds
\]
and note that $U(0) = 0$. Estimate \eqref{eqinvarbound} entails
\[
U'(t) = q\gamma \Big( \| u(t) \|_{L^q(\Omega)}^{q} \Big)^\frac{q-1}{q} \leq q\gamma \Big( \eta + \| u(0) \|_{L^q(\Omega)}^{q} + U(t) \Big)^\frac{q-1}{q}
\]
for all $t \in [0, T]$, where $\eta > 0$ is an arbitrary constant, which guarantees that the expession $X:=\eta + \| u(0) \|_{L^q(\Omega)}^{q} + U(t)$ is strictly positive. Multiplying both sides with $X^{(1-q)/q}$ and integrating from $0$ to $t$ gives
\[
\int_0^t \Big( \eta + \| u(0) \|_{L^q(\Omega)}^{q} + U(s) \Big)^\frac{1-q}{q} U'(s) \, ds \leq \int_0^t q\gamma \, ds.
\]
We now substitute $\sigma := U(s)$ and deduce
\vspace{-1ex}
\begin{multline*}
q\gamma t \geq \int_0^{U(t)} \Big( \eta + \| u(0) \|_{L^q(\Omega)}^{q} + \sigma \Big)^{\frac{1}{q}-1} \, d\sigma = q \Big( \eta + \| u(0) \|_{L^q(\Omega)}^{q} + \sigma \Big)^\frac{1}{q} \, \Big|_0^{U(t)} \\
= q \Big( \eta + \| u(0) \|_{L^q(\Omega)}^{q} + U(t) \Big)^\frac{1}{q} - q \Big( \eta + \| u(0) \|_{L^q(\Omega)}^{q} \Big)^\frac{1}{q} \geq q \Big( \| u(t) \|_{L^q(\Omega)}^{q} \Big)^\frac{1}{q} - q \Big( \eta + \| u(0) \|_{L^q(\Omega)}^{q} \Big)^\frac{1}{q}
\end{multline*}
where we have used \eqref{eqinvarbound} in the last step. Therefore,
\[
\| u(t) \|_{L^q(\Omega)} \leq \Big( \eta + \| u(0) \|_{L^q(\Omega)}^{q} \Big)^\frac{1}{q} + \gamma t
\]
and, taking the limit $\eta \rightarrow 0$,
\[
\| u(t) \|_{L^q(\Omega)} \leq \| u(0) \|_{L^q(\Omega)} + \gamma t \leq I + \gamma t.
\]
As the bound on the right hand side is independent of $q$, we even obtain
\begin{equation}
\label{eqlingrowth}
\| u(t) \|_{L^\infty(\Omega)} \leq I + \gamma t,
\end{equation}
for all $t \in [0, T]$.
This result naturally gives rise to 
\[
\|u\|_{L^\infty((0, T) \times \Omega)} \leq I + \gamma T.
\]
An analogous estimate is valid for $v$. As a result, we obtain
\[
\| u \|_{L^\infty((0, T) \times \Omega)}, \, \| v \|_{L^\infty((0, T) \times \Omega)} \leq 2I
\]
for $T > 0$ chosen sufficiently small.

Employing \eqref{eqlingrowth}, we also derive
\[
\max_{0 \leq t \leq T} \| u(t) \|_{L^2(\Omega)} \leq \max_{0 \leq t \leq T} \| u(t) \|_{L^\infty(\Omega)} \leq I + \gamma T.
\]
The same argument is applicable to $v$, which results in
\[
\max_{0 \leq t \leq T} \| u(t) \|_{L^2(\Omega)}, \, \max_{0 \leq t \leq T} \| v(t) \|_{L^2(\Omega)} \leq 2 I
\]
for sufficiently small $T>0$.
The corresponding bounds on $n_{tr}$ and $n_{tr}'$ can be deduced from the formula
\[
n_{tr}(t) = n_{tr}(0) + \frac{1}{\varepsilon} \int_0^t \wt f_3(s) \, ds
\]
and from an analogous one for $n_{tr}'$. In fact,
\[
\| n_{tr}(t) \|_{L^2(\Omega)} \leq \| n_{tr}(0) \|_{L^2(\Omega)} + \frac{1}{\varepsilon} \int_0^t \big\| \wt f_3(s) \big\|_{L^2(\Omega)} \, ds \leq I + \frac{T}{\varepsilon} (2 + 4 \beta I)
\]
and, hence,
\[
\max_{0 \leq t \leq T} \| n_{tr}(t) \|_{L^2(\Omega)}, \, \max_{0 \leq t \leq T} \| n_{tr}'(t) \|_{L^2(\Omega)} \leq 2I
\]
for $T>0$ sufficiently small.

\paragraph{Step 3: Contraction property of $\mathcal{S}$.} We consider $(\wt u_1, \wt v_1, \wt n_{tr,1}, \wt n_{tr,1}'), (\wt u_2, \wt v_2, \wt n_{tr,2}, \wt n_{tr,2}') \in M_T$ and the corresponding solutions $(u_1, v_1, n_{tr,1}, n_{tr,1}') = \mathcal{S}(\wt u_1, \wt v_1, \wt n_{tr,1}, \wt n_{tr,1}') \in M_T$ and $(u_2, v_2, n_{tr,2}, n_{tr,2}') = \mathcal{S}(\wt u_2, \wt v_2, \wt n_{tr,2}, \wt n_{tr,2}') \in M_T$. We further introduce the notation
\[
u := u_1 - u_2, \quad \wt u := \wt u_1 - \wt u_2
\]
and similarly $v$, $n_{tr}$, $n_{tr}'$, $\wt v$, $\wt n_{tr}$ and $\wt n_{tr}'$. Then, we have to show that 
\[
\| (u, v, n_{tr}, n_{tr}') \|_{X_T} \leq c \| (\wt u, \wt v, \wt n_{tr}, \wt n_{tr}') \|_{X_T}
\]
with a constant $c \in (0, 1)$ on a time interval $[0, T]$ small enough. The norm in $X_T$ is defined as
\[
\| (u, v, n_{tr}, n_{tr}') \|_{X_T} := \max_{0 \leq t \leq T} \| u(t) \|_{L^2(\Omega)} + \max_{0 \leq t \leq T} \| v(t) \|_{L^2(\Omega)} + \max_{0 \leq t \leq T} \| n_{tr}(t) \|_{L^2(\Omega)} + \max_{0 \leq t \leq T} \| n_{tr}'(t) \|_{L^2(\Omega)}.
\]
We obtain the following system by taking the difference of corresponding equations of the system for the $1$- and the $2$-variables, respectively:
\begin{equation}
\begin{cases}
\begin{aligned}
\partial_t u - \Delta u &= A_n \big( \wt u_1^+ - \wt u_2^+ \big) + e^\frac{V_n}{2} \Big( \wt n_{tr,1}^{[0,1]} - \wt n_{tr,2}^{[0,1]} \Big) - e^{V_n} \Big( \wt u_1^+ \wt n_{tr,1}'^{[0,1]} - \wt u_2^+ \wt n_{tr,2}'^{[0,1]} \Big) && \hspace{-8pt} =: \wt f_1, \\
\partial_t v - \Delta v &= A_p \big( \wt v_1^+ - \wt v_2^+ \big) + e^\frac{V_p}{2} \Big( \wt n_{tr,1}'^{[0,1]} - \wt n_{tr,2}'^{[0,1]} \Big) - e^{V_p} \Big( \wt v_1^+ \wt n_{tr,1}^{[0,1]} - \wt v_2^+ \wt n_{tr,2}^{[0,1]} \Big) && \hspace{-8pt} =: \wt f_2, \\
\varepsilon \, \partial_t n_{tr} &= \wt n_{tr,1}'^{[0,1]} - \wt n_{tr,2}'^{[0,1]} - e^\frac{V_p}{2} \Big( \wt v_1^+ \wt n_{tr,1}^{[0,1]} - \wt v_2^+ \wt n_{tr,2}^{[0,1]} \Big) \\
&\hspace{3.52cm} - \wt n_{tr,1}^{[0,1]} + \wt n_{tr,2}^{[0,1]} + e^\frac{V_n}{2} \Big( \wt u_1^+ \wt n_{tr,1}'^{[0,1]} - \wt u_2^+ \wt n_{tr,2}'^{[0,1]} \Big) && \hspace{-8pt} =: \wt f_3, \\
\varepsilon \, \partial_t n_{tr}' &= \wt n_{tr,1}^{[0,1]} - \wt n_{tr,2}^{[0,1]} - e^\frac{V_n}{2} \Big( \wt u_1^+ \wt n_{tr,1}'^{[0,1]} - \wt u_2^+ \wt n_{tr,2}'^{[0,1]} \Big) \\
&\hspace{3.57cm} - \wt n_{tr,1}'^{[0,1]} + \wt n_{tr,2}'^{[0,1]} + e^\frac{V_p}{2} \Big( \wt v_1^+ \wt n_{tr,1}^{[0,1]} - \wt v_2^+ \wt n_{tr,2}^{[0,1]} \Big) && \hspace{-8pt} =: \wt f_4.
\end{aligned}
\end{cases}
\end{equation}
We mention that $u$ and $v$ are subject to the boundary conditions 
\[
\hat n \cdot \Big(\nabla u + \frac{1}{2} u \nabla V_n \Big) = \hat n \cdot \Big(\nabla v + \frac{1}{2} v \nabla V_p \Big) = 0
\]
and the homogeneous initial conditions
\[
u(0) = v(0) = n_{tr}(0) = n_{tr}'(0) = 0.
\]

First, one finds
\[
\max_{0 \leq t \leq T} \| u(t) \|_{L^2(\Omega)} \leq C_1 \| u \|_{W_2(0, T)} \leq C_1 C_2 \| \wt f_1 \|_{L^2((0, T) \times \Omega)}
\]
where $C_1 > 0$ is the constant resulting from the embedding $W_2(0, T) \hookrightarrow C([0, T], L^2(\Omega))$. The constant $C_2 > 0$ originates from well-known parabolic regularity estimates for $\| u \|_{W_2(0, T)}$ in terms of the $L^2$-norms of $\wt f_1$ and $u(0)=0$. Therefore,
\vspace{-1ex}
\begin{align*}
\max_{0 \leq t \leq T} \| u(t) \|_{L^2(\Omega)} &\leq C_1 C_2 \Big( \alpha \big\| \wt u_1^+ - \wt u_2^+ \big\|_{L^2((0, T) \times \Omega)} + \beta \big\| \wt n_{tr,1}^{[0,1]} - \wt n_{tr,2}^{[0,1]} \big\|_{L^2((0, T) \times \Omega)} \\
&\qquad\qquad\, + \beta \big\| \wt u_1^+ - \wt u_2^+ \big\|_{L^2((0, T) \times \Omega)} \big\| \wt n_{tr,1}'^{[0,1]} \big\|_{L^\infty((0, T) \times \Omega)} \\
&\qquad\qquad\, + \beta \big\| \wt u_2^+ \big\|_{L^\infty((0, T) \times \Omega)} \big\| \wt n_{tr,1}'^{[0,1]} - \wt n_{tr,2}'^{[0,1]} \big\|_{L^2((0, T) \times \Omega)} \Big) \\
&\leq C_1 C_2 \left( \beta \| \wt n_{tr} \|_{L^2((0, T) \times \Omega)} + (\alpha + \beta) \| \wt u \|_{L^2((0, T) \times \Omega)} + 2 \beta I \| \wt n_{tr}' \|_{L^2((0, T) \times \Omega)} \right).
\end{align*}
Moreover, every $f \in C([0, T], L^2(\Omega))$ fulfils
\begin{equation*}
\| f \|_{L^2((0, T) \times \Omega)}^2 = \int_0^T \! \int_\Omega f^2 \, dx \, dt \leq \int_0^T  \, dt \, \max_{0 \leq t \leq T} \| f(t) \|_{L^2(\Omega)}^2 = T \| f \|_{C([0, T], L^2(\Omega))}^2
\end{equation*}
and we proceed with the previous estimates to derive
\[
\max_{0 \leq t \leq T} \| u(t) \|_{L^2(\Omega)} \leq C_1 C_2 (\alpha + 2 \beta I) \sqrt{T} \big( \| \wt n_{tr} \|_{C([0, T], L^2(\Omega))} + \| \wt u \|_{C([0, T], L^2(\Omega))} + \| \wt n_{tr}' \|_{C([0, T], L^2(\Omega))} \big).
\]
In a similar way, we arrive at
\[
\max_{0 \leq t \leq T} \| v(t) \|_{L^2(\Omega)} \leq C_1 C_2 (\alpha + 2 \beta I) \sqrt{T} \big( \| \wt n_{tr}' \|_{C([0, T], L^2(\Omega))} + \| \wt v \|_{C([0, T], L^2(\Omega))} + \| \wt n_{tr} \|_{C([0, T], L^2(\Omega))} \big).
\]

Due to $n_{tr}(0) = 0$, one obtains
\[
n_{tr}(t) = \frac{1}{\varepsilon} \int_0^t \wt f_3 \, ds
\]
for $t \in [0, T]$ and, using similar techniques as above,
\vspace{-1ex}
\begin{align*}
\max_{0 \leq t \leq T} &\| n_{tr}(t) \|_{L^2(\Omega)} \leq \frac{1}{\varepsilon} \int_0^T \| \wt f_3 \|_{L^2(\Omega)} \, ds \leq \frac{\sqrt{T}}{\varepsilon} \| \wt f_3 \|_{L^2((0, T) \times \Omega)} \\
&\leq \frac{1 + 2 \beta I}{\varepsilon} \sqrt{T} \big( \| \wt u \|_{L^2((0, T) \times \Omega)} + \| \wt v \|_{L^2((0, T) \times \Omega)} + \| \wt n_{tr} \|_{L^2((0, T) \times \Omega)} + \| \wt n_{tr}' \|_{L^2((0, T) \times \Omega)} \big) \\
&\leq \frac{1 + 2 \beta I}{\varepsilon} T \big( \| \wt u \|_{C([0, T], L^2(\Omega))} + \| \wt v \|_{C([0, T], L^2(\Omega))} + \| \wt n_{tr} \|_{C([0, T], L^2(\Omega))} + \| \wt n_{tr}' \|_{C([0, T], L^2(\Omega))} \big).
\end{align*}
Note that because of $\wt f_4 = -\wt f_3$, the last estimate equally serves as an upper bound for $\| n_{tr}'(t) \|_{L^2(\Omega)}$. Taking the sum of the above estimates and choosing $T > 0$ sufficiently small yields
\[
\| (u, v, n_{tr}, n_{tr}') \|_{X_T} \leq c\, \| (\wt u, \wt v, \wt n_{tr}, \wt n_{tr}') \|_{X_T}
\]
with some $c \in (0, 1)$.

\paragraph{Step 4: Solution of \eqref{eqsystem}.} Step 2 and Step 3 imply that for $T > 0$ sufficiently small the mapping $\mathcal{S}: M_T \rightarrow M_T$ is a contraction. Banach's fixed point theorem, thus, guarantees that there exists a unique $(u, v, n_{tr}, n_{tr}') \in M_T$ such that $\mathcal{S}(u, v, n_{tr}, n_{tr}') = (u, v, n_{tr}, n_{tr}')$. Moreover, due to standard parabolic regularity for $(u, v)$, the fixed-point $(u, v, n_{tr}, n_{tr}')$ is the unique weak solution of
\begin{equation}
\label{eqsystemfixpt}
\begin{cases}
\begin{aligned}
\partial_t u - \Delta u &= A_n u^+ + e^\frac{V_n}{2} n_{tr}^{[0,1]} - e^{V_n} u^+ n_{tr}'^{[0,1]}, \\
\partial_t v - \Delta v &= A_p v^+ + e^\frac{V_p}{2} n_{tr}'^{[0,1]} - e^{V_p} v^+ n_{tr}^{[0,1]}, \\
\varepsilon \, \partial_t n_{tr} &= n_{tr}'^{[0,1]} - e^\frac{V_p}{2} v^+ n_{tr}^{[0,1]} - n_{tr}^{[0,1]} + e^\frac{V_n}{2} u^+ n_{tr}'^{[0,1]}, \\
\varepsilon \, \partial_t n_{tr}' &= n_{tr}^{[0,1]} - e^\frac{V_n}{2} u^+ n_{tr}'^{[0,1]} - n_{tr}'^{[0,1]} + e^\frac{V_p}{2} v^+ n_{tr}^{[0,1]}.
\end{aligned}
\end{cases}
\end{equation}
In order to prove the non-negativity of $u$, $v$, $n_{tr}$ and $n_{tr}'$, we adapt an argument from \cite{MWZ08}. First, we define
\[
h := \min(0, u)
\]
on $[0, T] \times \Omega$ and notice that $h \leq 0$ and $h(t=0) = 0$ a.e. since $u(0) \geq 0$ a.e. We now multiply the first equation in \eqref{eqsystemfixpt} with $h$ and integrate over $(0, t) \times \Omega$ for $t \in [0, T]$. This yields
\begin{equation}
\label{eqpositive}
\int_0^t \int_\Omega \partial_s u \, h \, dx\,ds = \int_0^t \int_\Omega \Delta u \, h \, dx\,ds + \int_0^t \int_\Omega A_n u^+ h \, dx\,ds + \int_0^t \int_\Omega \Big( e^\frac{V_n}{2} n_{tr}^{[0,1]} - e^{V_n} u^+ n_{tr}'^{[0,1]} \Big) \, h \, dx\,ds.
\end{equation}
The first term on the right hand side of \eqref{eqpositive} can be seen to be non-positive using integration by parts and the boundary condition from \eqref{eqbcuv}:
\vspace{-1ex}
\begin{align*}
\int_0^t \int_\Omega \Delta u \, h \, dx\,ds &= - \int_0^t \int_\Omega \nabla u \cdot \nabla h \, dx\,ds - \frac{1}{2} \int_0^t \int_{\partial \Omega} u\, h\, \hat n \cdot \nabla V_n \, d\sigma \, ds \\
&\leq - \int_0^t \int_\Omega \nabla u \cdot \nabla h \, dx\,ds = - \int_0^t \int_\Omega \nabla h \cdot \nabla h \, dx\,ds \leq 0
\end{align*}
due to $u h \geq 0$, $\hat n \cdot \nabla V_n \geq 0$, and since $\nabla h \neq 0$ holds true only in the case $u < 0$, where we have $\nabla u = \nabla h$ in $L^2$, see e.g. \cite{GT77}. Moreover,
\[
\int_0^t \int_\Omega A_n u^+ h \, dx\,ds = 0,
\]
and the third term in \eqref{eqpositive} is again non-positive as an integral over non-positive quantities:
\[
\int_0^t \int_\Omega \Big( e^\frac{V_n}{2} n_{tr}^{[0,1]} - e^{V_n} u^+ n_{tr}'^{[0,1]} \Big) \, h \, dx\,ds = \int_0^t \int_\Omega e^\frac{V_n}{2} n_{tr}^{[0,1]} \, h \, dx\,ds \leq 0
\]
as a consequence of $u^+ h = 0$ in $L^2(\Omega)$. 
The left hand side of \eqref{eqpositive} can be reformulated as 
\[
\int_0^t \int_\Omega \partial_s u \, h \, dx\,ds = \int_0^t \int_\Omega \partial_s h \, h \, dx\,ds = \frac{1}{2} \int_\Omega \int_0^t \Big( \frac{d}{ds} \, h^2 \Big) \, ds\,dx = \frac{1}{2} \| h(t) \|_{L^2(\Omega)}^2.
\]
For the first step, we have used that the integrand $\partial_s u \, h$ only contributes to the integral if $h < 0$. But in this case, $u = h$ and, hence, $\partial_s u = \partial_s h$ in $L^2$, see e.g. \cite{GT77}. This proves $\| h(t) \|_{L^2(\Omega)} \leq 0$ for all $t \in [0, T]$, which establishes $h(t) = 0$ in $L^2(\Omega)$ for all $t \in [0, T]$, and thus $u(t, x) \geq 0$ for all $t \in [0, T]$ and a.e. $x \in \Omega$. In the same way, one can show that $v(t, x) \geq 0$ for all $t \in [0, T]$ and a.e. $x \in \Omega$. 

The non-negativity of $n_{tr}$ follows from a similar idea using 
\[
h := \min(0, n_{tr}).
\]
Again, $h \leq 0$ and $h(t=0) = 0$ due to $n_{tr}(0) \geq 0$. Multiplying the third equation of \eqref{eqsystemfixpt} with $h$ and integrating over $(0, t) \times \Omega$, $t \in [0, T]$, we find
\[
\varepsilon \int_0^t \int_\Omega \partial_s n_{tr} \, h \, dx\,ds = \int_0^t \int_\Omega \Big( n_{tr}'^{[0,1]} - e^\frac{V_p}{2} v^+ n_{tr}^{[0,1]} - n_{tr}^{[0,1]} + e^\frac{V_n}{2} u^+ n_{tr}'^{[0,1]} \Big) h \, dx \, ds.
\]
As before, all terms under the integral on the right hand side involving $n_{tr}^{[0, 1]}$ vanish. Consequently,
\[
\frac{\varepsilon}{2} \| h(t) \|_{L^2(\Omega)}^2 = \varepsilon \int_0^t \int_\Omega \partial_s h \, h \, dx\,ds = \int_0^t \int_\Omega \Big( n_{tr}'^{[0,1]} + e^\frac{V_n}{2} u^+ n_{tr}'^{[0,1]} \Big) h \, dx \, ds \leq 0
\]
for all $t \in [0, T]$. The same result holds true for $n_{tr}'$. Therefore, we have verified that $n_{tr}(t, x)$, $n_{tr}'(t, x) \geq 0$ for all $t \in [0, T]$ and a.e. $x \in \Omega$.

The non-negativity of $n_{tr}$ and $n_{tr}'$ together with $n_{tr}' = 1 - n_{tr}$ from \eqref{eqsumntrntr'} now even imply
\[
n_{tr}(t, x), \, n_{tr}'(t, x) \in [0, 1], \qquad\text{for all}\quad t \in [0, T] \quad\text{and a.e.}\quad x \in \Omega. 
\]
This allows us to identify the unique weak solution $(u, v, n_{tr}, n_{tr}')$ of \eqref{eqsystemfixpt} to equally solve 
\begin{equation}
\label{eqsystemresult}
\begin{cases}
\begin{aligned}
\partial_t u - \Delta u &= A_n u + e^\frac{V_n}{2} n_{tr} - e^{V_n} u (1-n_{tr}), \\
\partial_t v - \Delta v &= A_p v + e^\frac{V_p}{2} (1-n_{tr}) - e^{V_p} v\, n_{tr}, \\
\varepsilon \, \partial_t n_{tr} &= 1-n_{tr} - e^\frac{V_p}{2} v\, n_{tr} - n_{tr} + e^\frac{V_n}{2} u (1-n_{tr}),
\end{aligned}
\end{cases}
\end{equation}
which is the transform version of the original problem \eqref{eqsystem}.

Up to now, we have proven that there exists a unique solution $(u, v, n_{tr}) \in C([0, T], L^2(\Omega))^3$ such that $(u, v, n_{tr}, 1 - n_{tr}) \in M_T$ on a sufficiently small time-interval $[0, T]$.

\paragraph{Step 5: Global solution.} We now fix $T^\ast>0$ in such a way that $[0, T^\ast)$ is the maximal time-interval of existence for the solution $(u, v, n_{tr}) \in C([0, T], L^2(\Omega))^3$ of \eqref{eqsystemresult}. Moreover, we choose some arbitrary $q \in \mathbb{N}_{\geq 2}$ and multiply the first equation in \eqref{eqsystemresult} with $u^{q-1}$. Integrating over $\Omega$ at time $t \in [0, T^\ast)$ gives
\[
\frac{d}{dt} \int_{\Omega} \frac{u^q}{q} \, dx = \int_{\Omega} u^{q-1} \partial_t u \, dx = \int_{\Omega} u^{q-1} \Delta u \, dx + \int_{\Omega} A_n u^{q} \, dx + \int_{\Omega} u^{q-1} \Big( e^\frac{V_n}{2} n_{tr} - e^{V_n} u (1-n_{tr}) \Big) \, dx.
\]
Integration by parts and the estimates $|A_n| \leq \alpha$, $\big|e^\frac{V_n}{2} n_{tr} - e^{V_n} u (1-n_{tr})\big| \leq \beta (1 + u)$ further yield
\[
\frac{d}{dt} \int_{\Omega} \frac{u^q}{q} \, dx \leq - (q-1) \int_\Omega u^{q-2} | \nabla u |^2 \, dx - \frac{1}{2} \int_{\partial \Omega} u^q \, \hat n \cdot \nabla V_n \, d\sigma \, ds + \alpha \int_{\Omega} u^{q} \, dx + \beta \int_\Omega (u^{q-1} + u^q) \, dx.
\]
Moreover, we derive
\[
\int_\Omega u^{q-1} \, dx = \int_{\{ u \leq 1 \}} u^{q-1} \, dx + \int_{\{ u > 1 \}} \frac{u^q}{u} \, dx \leq \int_\Omega 1 \, dx + \int_\Omega u^q \, dx = 1 + \int_\Omega u^q \, dx
\]
where we used $|\Omega| = 1$. Hence,
\begin{equation}
\label{eqsolutionl2h1}
\frac{d}{dt} \int_{\Omega} \frac{u^q}{q} \, dx \leq \beta + (\alpha + 2 \beta) \int_\Omega u^q \, dx \leq \gamma \left(1 + \int_\Omega u^q \, dx \right)
\end{equation}
after defining $\gamma := \alpha + 2 \beta$. This results in 
\[
\frac{d}{dt} \int_{\Omega} u^q \, dx \leq \gamma q \left( 1 + \int_\Omega u^q \, dx \right),
\]
which can be integrated over time from $0$ to $t$:
\[
\| u(t) \|_{L^q(\Omega)}^q \leq \| u(0) \|_{L^q(\Omega)}^q + \gamma q \int_0^t \Big( 1 + \| u(s) \|_{L^q(\Omega)}^q \Big) \, ds.
\]
From this generalised Gronwall-type inequality, we deduce (cf. \cite{Bee75})
\[
\| u(t) \|_{L^q(\Omega)}^q \leq \| u(0) \|_{L^q(\Omega)}^q e^{\gamma qt} + e^{\gamma qt} - 1 < \big(1 + \| u(0) \|_{L^q(\Omega)}^q \big) e^{\gamma qt}
\]
and
\[
\| u(t) \|_{L^q(\Omega)} \leq \big(1 + \| u(0) \|_{L^q(\Omega)} \big) e^{\gamma t} \leq I e^{\gamma t}
\]
since $1+\| u(0) \|_{L^q(\Omega)} \leq 1+\| u(0) \|_{L^\infty(\Omega)} \leq I$. As $I e^{\gamma t}$ is independent of $q$, we even arrive at
\[
\| u(t) \|_{L^\infty(\Omega)} \leq I e^{\gamma t}.
\]
In the same way, we can show that $\| v(t) \|_{L^\infty(\Omega)} \leq I e^{\gamma t}$ for all $t \in [0, T^\ast)$. As a consequence, we obtain that the solution $(u, v, n_{tr}) \in C([0, T], L^2(\Omega))^3$ can be extended for all times, i.e. $T^\ast = \infty$.

\paragraph{Step 6: $L^\infty$-bounds for $n$ and $p$.}
We now prove the linearly growing $L^{\infty}$-bounds \eqref{npLinfty} for $n$ and $p$. We only detail the bound for $p$ and sketch how the bound for $n$ follows in a similar fashion. 
After recalling (with $\tau_p=1$ w.l.o.g.)
\begin{equation*}
\partial_t p = \nabla \cdot J_p + \left( 1 - n_{tr} - \frac{p}{p_0 e^{-V_p}} n_{tr} \right),\qquad J_p=e^{-V_p} \nabla \left( p\, e^{V_p} \right),
\end{equation*}
we introduce the variable $w=p \,e^{V_p}$ and 
observe that $\nabla \cdot J_p = \nabla \cdot \left(e^{-V_p} \nabla w \right) =e^{-V_p}\left( \Delta w - \nabla V_p \cdot \nabla w\right)$ and thus, 
\begin{equation}\label{eq:w}
\partial_t w = \Delta w - \nabla V_p \cdot \nabla w + e^{V_p}\left( 1 - n_{tr} - \frac{n_{tr}}{p_0} w  \right), 
\end{equation}
while the no-flux boundary condition $\hat n \cdot J_p = 0$ on $\partial \Omega$ transforms to the homogeneous Neumann condition $\hat n \cdot \nabla w = 0$ on $\partial \Omega$.

Next, by testing \eqref{eq:w} with the positive part $(w-r-st)_{+}:=\max\{0, w-r-st\}$ for two constant $r,s>0$ to be chosen, we calculate
by integration by parts in the first two terms
\begin{align*}
\frac{d}{dt} \frac{1}{2} \int_{\Omega} &(w-r-st)_{+}^2\,dx = 
\int_{\Omega} (w-r-st)_{+} \left(\Delta w - \nabla V_p \cdot \nabla w + e^{V_p}\Bigl( 1 - n_{tr} - \frac{n_{tr}}{p_0} w \Bigr) -s \right)dx \\
&=- \int_{\Omega} \mathbb{1}_{w\ge r+st} |\nabla w|^2 \,dx 
- \int_{\Omega} \nabla V_p \cdot \nabla \frac{(w-r-st)_{+}^2}{2}\,dx \\
&\quad+ \int_{\Omega} (w-r-st)_{+}  \left(e^{V_p}\Bigl( 1 - n_{tr} - \frac{n_{tr}}{p_0} w \Bigr) -s\right)dx\\
&\le \frac{\|\Delta V_p\|_{\infty}}{2} \int_{\Omega} (w-r-st)_{+}^2\,dx + \int_{\Omega} (w-r-st)_{+}  \left(e^{V_p}\Bigl( 1 - n_{tr} - \frac{n_{tr}}{p_0} w \Bigr) -s\right)dx,
\end{align*}
since $\hat n \cdot V_p\ge0$ by assumption \eqref{eqpot}.
Moreover, since $n_{tr}\in[0,1]$ and $w\ge0$, we have 
\begin{equation*}
\frac{d}{dt} \frac{1}{2} \int_{\Omega} (w-r-st)_{+}^2\,dx 
\le \frac{\|\Delta V_p\|_{\infty}}{2} \int_{\Omega} (w-r-st)_{+}^2\,dx + \int_{\Omega} (w-r-st)_{+}  \Bigl(\|e^{V_p}\|_{\infty} -s\Bigr)dx.
\end{equation*}
Thus, by choosing $s \coleq \|e^{V_p}\|_{\infty}$ and $r \coleq \|w(\tau,\cdot)\|_{\infty}$ for some time $\tau\ge0$, we conclude that 
\begin{equation*}
\frac{d}{dt} \int_{\Omega} (w-r-st)_{+}^2\,dx 
\le \|\Delta V_p\|_{\infty} \int_{\Omega} (w-r-st)_{+}^2\,dx, 
\end{equation*}
and a Gronwall lemma implies 
\begin{equation}\label{wLinfty}
\|w(t,\cdot)\|_{\infty} \le \|w(\tau,\cdot)\|_{\infty}+\|e^{V_p}\|_{\infty}\,t,\qquad
\text{for all } t\ge\tau\ge0.
\end{equation}
Transforming back, this yields 
\begin{equation}\label{pLinfty}
\|p(t,\cdot)\|_{\infty} \le \frac{1}{\inf\{e^{V_p}\}}\left(\|p(\tau,\cdot)\|_{\infty}\|e^{V_p}\|_{\infty}+\|e^{V_p}\|_{\infty}\,t\right),\qquad
\text{for all } t\ge\tau\ge0.
\end{equation}

In order to deduce the analog bound for $n$ in \eqref{npLinfty}, we consider (with $\tau_n=1$ w.l.o.g.) 
\begin{equation*}
\partial_t n = \nabla \cdot J_n + \left( n_{tr} - \frac{n}{n_0 e^{-V_n}} \bigl(1-n_{tr}\bigr) \right),\qquad J_n=e^{-V_n} \nabla \big(n\, e^{V_n} \bigr).
\end{equation*}
We introduce the variable $\omega=n \,e^{V_n}$ and obtain in the same way as in \eqref{eq:om}
\begin{equation*}
\partial_t \omega = \Delta \omega - \nabla V_n \cdot \nabla \omega + e^{V_n}\left( n_{tr} - \frac{1-n_{tr}}{n_0} \omega  \right).
\end{equation*}
Following the same arguments as above, 
\begin{equation}\label{omLinfty}
\|\omega(t,\cdot)\|_{\infty} \le \|\omega(\tau,\cdot)\|_{\infty}+\|e^{V_n}\|_{\infty}\,t,\qquad
\text{for all } t\ge\tau\ge0.
\end{equation}
Transforming back, this yields 
\begin{equation}\label{nLinfty}
\|n(t,\cdot)\|_{\infty} \le \frac{1}{\inf\{e^{V_n}\}}\left(\|n(\tau,\cdot)\|_{\infty}\|e^{V_n}\|_{\infty}+\|e^{V_n}\|_{\infty}\,t\right),\qquad
\text{for all } t\ge\tau\ge0
\end{equation}
and thus \eqref{npLinfty}.

\paragraph{Step 7: Regularity and bounds for $n_{tr}$.}
We still have to verify $n_{tr} \in C([0, T], L^\infty(\Omega))$ for all $T > 0$. Now, let $T>0$ and recall that
\[
n_{tr}(t) = n_{tr}(0) + \frac{1}{\varepsilon} \int_0^t \big( 1-n_{tr} - e^{V_p} p \, n_{tr} - n_{tr} + e^{V_n} n (1-n_{tr}) \big) \, ds
\]
in $L^2(\Omega)$ for all $t \in [0,T]$. Considering a sequence $(t_n)_{n \in \mathbb{N}} \subset [0, T]$, $t_n \rightarrow t$, we thus arrive at
\[
\| n_{tr}(t_n) - n_{tr}(t) \|_{L^\infty(\Omega)} \leq \frac{1}{\epsilon} \left| \int_t^{t_n}\! \| 1-n_{tr} - e^{V_p} p \, n_{tr} - n_{tr} + e^{V_n} n (1-n_{tr}) \|_{L^\infty(\Omega)} \, ds \right| \leq \frac{|t_n\! -\! t|}{\epsilon} C_{T} \rightarrow 0
\]
for $n \rightarrow \infty$. This proves the assertion. 

The claim $\partial_t n_{tr} \in C([0, T], L^2(\Omega))$ for all $T > 0$ is an immediate consequence of the last equation in \eqref{eqsystemresult} together with the $L^2$-continuity and $L^\infty$-bounds of $u$, $v$ and $n_{tr}$.

Next, concerning the bounds \eqref{ntrbounds}, we recall system \eqref{eqsystem} and observe that for all $\varepsilon\in(0,\varepsilon_0]$
\begin{align*}
\varepsilon \partial_t n_{tr} &= h(n_{tr}):=R_p(p, n_{tr}) - R_n(n, n_{tr}),
\end{align*}
in the sense of $L^2(\Omega)$, where $h(n_{tr}=0)\ge \frac{1}{\tau_p}>0$ and $h(n_{tr}=1)\le -\frac{1}{\tau_n} < 0$ uniformly for all non-negative $n$ and $p$. 
Therefore, wherever $n_{tr,I}(x)=0$ (or analogous $n_{tr,I}(x)=1$), 
an elementary argument proves that $n_{tr}(t,x)$ grows (or decreases) 
linearly in time and decays back to $0$ (or $1$) at most like $(a+bt)^{-1}$. More precisely, we reuse the transformed variable $w = p \, e^{V_p}$ and find 
\begin{equation*}
\varepsilon \partial_t n_{tr} \ge 
\frac{1}{\tau_p}\left[1 - \left(1+\frac{\tau_p}{\tau_n} + \frac{\|w\|_{\infty}}{ p_0}\right)n_{tr}\right]
\ge 
\frac{1}{\tau_p}\left[1 - \left(1+\frac{\tau_p}{\tau_n} + \frac{r+st}{ p_0}\right)n_{tr}\right]
\end{equation*}
for some constants $r$ and $s$ due to the estimate \eqref{wLinfty}. Setting $\tau_n \coleq \tau_p \coleq 1$ w.l.o.g., we have
\[
\varepsilon \partial_t n_{tr} \geq 1 - (\wt r + \wt s t) n_{tr}.
\]
with appropriate $\wt r, \wt s > 0$ independent of $\varepsilon$. 
By observing that the ODE $\varepsilon_0\dot{y} = 1 - (\wt r + \wt s t) y$ features the positive nullcline $y_0(t) = 1/(\wt r + \wt s t)$, 
which moreover attracts all solution trajectories, standard comparison arguments (pointwise in $x\in\Omega$) imply that for all times $\tau > 0$, there exists 
a constant exist positive constants $\eta = \eta(\varepsilon_0, \tau, \tau_n, \tau_p)$, $\theta = \theta(C_n,C_p,K_n,K_p)$ and a sufficiently small constant $\gamma(\tau,C_n,C_p,K_n,K_p)>0$   
such that 
\begin{equation*}
n_{tr}(t,x) \ge \min\Bigl\{\eta t, \frac{\gamma}{1+\theta t}\Bigr\} \quad \text{for all } t\ge0 \text{ and a.a. } x\in\Omega   
\end{equation*}
where $\eta \tau = \frac{\gamma}{1+\theta \tau}$ such that the linear and the inverse linear bound intersect at time $\tau$.

Finally, the upper bounds \eqref{ntrbounds} follow from analog arguments.

\paragraph{Step 8: Lower bounds for $n$ and $p$.}
Finally, we prove the bounds \eqref{npbounds}. We will only detail the argument for the lower bound on $n$, as the bound for $p$ follows in an analog way. Recalling the transformed equation for $\omega=e^{V_n}n$ (satisfying $\hat n \cdot \nabla \omega = 0$ on $\partial \Omega$), we estimate
\begin{equation}\label{aaa}
\partial_t \omega = \Delta \omega - \nabla V_n \cdot \nabla \omega + e^{V_n}\left( n_{tr} - \frac{1-n_{tr}}{n_0} \omega  \right)
\ge \Delta \omega -\nabla V_n \cdot \nabla \omega - \alpha \, \omega + c n_{tr},
\end{equation}
where $\alpha > 0$ and $c>0$ are positive constants due to the assumptions \eqref{eqpot} and $e^{V_n} \omega n_{tr}\ge0$.

Next, we use \eqref{ntrbounds}, i.e. that for all $\tau>0$ fixed, there exist constants $\eta$, $\theta$ and $\gamma$ such that $n_{tr}(t,x) \ge \eta t$ for all $0\le t \le \tau$ and a.a. $x\in\Omega$, while $n_{tr}(t,x) \ge \gamma/(1 + \theta t)$ for all $t \ge \tau$ and a.a. $x\in\Omega$.
Then, by introducing the negative part $(\omega)_{-}:= \min\{\omega,0\}$ and testing \eqref{aaa} with $\bigl(\omega-\frac{\mu t^2}{2}\bigr)_{-}$ for a constant $\mu>0$ to be chosen below, we estimate
\begin{align*}
&\frac{d}{dt} \frac{1}{2} \int_{\Omega} \Bigl(\omega-\frac{\mu t^2}{2}\Bigr)_{-}^2 \,dx =  \int_{\Omega} \Bigl(\omega-\frac{\mu t^2}{2}\Bigr)_{-}\left(\partial_t \omega -\mu t\right)dx=
\int_{\Omega} \left|\Bigl(\omega-\frac{\mu t^2}{2}\Bigr)_{-}\right|\left(-\partial_t \omega +\mu t\right)dx
\\
\le  &\int_{\Omega} \left|\Bigl(\omega-\frac{\mu t^2}{2}\Bigr)_{-}\right|
\left(- \Delta \omega + \nabla V_n \cdot \nabla \omega + \alpha \omega - c n_{tr} +\mu t\right)dx\\
= &\int_{\Omega} \Bigl(\omega-\frac{\mu t^2}{2}\Bigr)_{-} \left(\Delta \omega - \nabla V_n \cdot \nabla \omega \right) dx + \int_{\Omega} \left|\Bigl(\omega-\frac{\mu t^2}{2}\Bigr)_{-}\right|
\left(\alpha \omega - c n_{tr} +\mu t\right)dx\\
\le &- \int_{\Omega} \mathbb{1}_{\omega \le \frac{\mu t^2}{2}} |\nabla \omega|^2\,dx - \frac12 \int_\Omega \nabla \Bigl(\omega-\frac{\mu t^2}{2}\Bigr)_{-}^2 \cdot \nabla V_n \, dx
+ \int_{\Omega} \left|\Bigl(\omega-\frac{\mu t^2}{2}\Bigr)_{-}\right|
\left( \alpha \omega -c n_{tr}+ \mu t\right)dx.
\end{align*}
Thus, for $0\le t \le \tau$ when $n_{tr}(t,x) \ge \eta t$, we have 
\begin{align*}
\frac{d}{dt} \frac{1}{2} \int_{\Omega} \Bigl(\omega-\frac{\mu t^2}{2}\Bigr)_{-}^2 \,dx 
\le \int_\Omega \Bigl(\omega-\frac{\mu t^2}{2}\Bigr)_{-}^2 \frac{\Delta V_n}{2} \, dx + \int_{\Omega} \left|\Bigl(\omega-\frac{\mu t^2}{2}\Bigr)_{-}\right|
\left( \alpha \frac{\mu t^2}{2}  -c \eta t + \mu t\right)dx.
\end{align*}
If we choose $\mu \bigl(\alpha\frac{\tau}{2}+1\bigr)\le c \eta$, we obtain
\[
\frac{d}{dt} \int_{\Omega} \Bigl(\omega-\frac{\mu t^2}{2}\Bigr)_{-}^2 \,dx 
\le \|\Delta V_n\|_{L^\infty(\Omega)} \int_\Omega \Bigl(\omega-\frac{\mu t^2}{2}\Bigr)_{-}^2 \, dx.
\]
Hence, since $\int_{\Omega} \bigl(\omega(0,x)\bigr)_{-}^2dx=0$, we deduce from a Gronwall lemma
$$
\int_{\Omega} \Bigl(\omega-\frac{\mu t^2}{2}\Bigr)_{-}^2 \,dx =0, 
\qquad \text{for all } 0\le t \le \tau,
$$
which yields in particular $\omega(t,x)\ge \frac{\mu t^2}{2}$ for all $0 \leq t \leq \tau$ and a.a. $x\in\Omega$. 

Moreover, for $t\ge \tau$ when $n_{tr}(t,x) \ge \frac{\gamma}{1+\theta t}$, we test \eqref{aaa} with $\bigl(\omega-\frac{\Gamma}{1+\theta t}\bigr)_{-}$ for a constant $\Gamma>0$ to be chosen below, and estimate similar to above
\begin{align*}
\frac{d}{dt} \frac{1}{2} &\int_{\Omega} \Bigl(\omega-\frac{\Gamma}{1+\theta t}\Bigr)_{-}^2 \,dx =  \int_{\Omega} \Bigl(\omega-\frac{\Gamma}{1+\theta t}\Bigr)_{-} \left(\partial_t \omega+\frac{\Gamma\theta}{(1+\theta t)^2}\right)dx\\
&\le  \int_{\Omega} \Bigl(\omega-\frac{\Gamma}{1+\theta t}\Bigr)_{-}
\left( \Delta \omega - \nabla V_n \cdot \nabla \omega - \alpha \omega + c n_{tr} + \frac{\Gamma\theta}{(1+\theta t)^2}\right)dx\\
& \le - \int_{\Omega} \mathbb{1}_{\omega\le \frac{\Gamma}{1+\theta t}} |\nabla \omega|^2\,dx - \frac12 \int_\Omega \nabla \Bigl(\omega-\frac{\Gamma}{1+\theta t}\Bigr)_{-}^2 \cdot \nabla V_n \, dx \\
&\qquad + \int_{\Omega} \left|\Bigl(\omega-\frac{\Gamma}{1+\theta t}\Bigr)_{-}\right|
\left( \alpha \omega - c n_{tr} - \frac{\Gamma\theta}{(1+\theta t)^2}\right)dx.
\end{align*}
And as $n_{tr}(t,x) \ge \frac{\gamma}{1+\theta t}$ for $t \geq \tau$, we find
\begin{multline*}
\frac{d}{dt} \frac{1}{2} \int_{\Omega} \Bigl(\omega-\frac{\Gamma}{1+\theta t}\Bigr)_{-}^2 \,dx \leq 
\int_{\Omega} \Bigl(\omega-\frac{\Gamma}{1+\theta t}\Bigr)_{-}^2 \frac{\Delta V_n}{2} \, dx \\
+ \int_{\Omega} \left|\Bigl(\omega-\frac{\Gamma}{1+\theta t}\Bigr)_{-}\right|
\left( \alpha \frac{\Gamma}{1+\theta t} - c \frac{\gamma}{1+\theta t}-\frac{\Gamma\theta}{(1+\theta t)^2}\right)dx.
\end{multline*}
Choosing $\alpha \Gamma \le c \gamma$, we arrive at
\[
\frac{d}{dt} \int_{\Omega} \Bigl(\omega-\frac{\Gamma}{1+\theta t}\Bigr)_{-}^2 \,dx \leq 
\|\Delta V_n\|_{L^\infty(\Omega)} \int_{\Omega} \Bigl(\omega-\frac{\Gamma}{1+\theta t}\Bigr)_{-}^2 \, dx.
\]
By further reducing either $\Gamma$ or $\mu$, we are able to satisfy $\frac{\Gamma}{1+\theta \tau} = \frac{\mu \tau^2}{2}$. On the one hand, this implies that $\int_{\Omega} \Bigl(\omega(\tau,x) - \frac{\Gamma}{1+\theta \tau}\Bigr)_{-}^2 \,dx=0$, which results --- by using a Gronwall argument --- in 
$$
\int_{\Omega} \Bigl(\omega(t,x)-\frac{\Gamma}{1+\theta t}\Bigr)_{-}^2 \,dx=0, 
\qquad \text{for all }  t \ge \tau,
$$
and, hence, $\omega(t,x)\ge \frac{\Gamma}{1+\theta t}$ for all $t \geq \tau$ and a.a. $x\in\Omega$. On the other hand, the increasing and decreasing bounds now again intersect at time $\tau$ as desired.
%
\end{proof}

\vskip 0.5cm
\noindent{\bf Acknowledgements.} 	
The second author is supported by the International Research Training Group IGDK 1754 ``Optimization and Numerical Analysis for Partial Differential Equations with Nonsmooth Structures'', funded by the German Research Council (DFG) and the Austrian Science Fund (FWF): [W 1244-N18].

\bibliographystyle{alpha}

\end{document}